\documentclass[11pt,  reqno]{amsart}

\usepackage{amsmath,amssymb,amscd,amsthm,amsxtra, esint}

\usepackage[english, french]{babel}
\usepackage{hyperref}

\allowdisplaybreaks[2]

\usepackage{bbm}
\newcommand{\sharpind}{\mathbbm{1}}

\usepackage{verbatim}
\DeclareFontFamily{U}{mathx}{\hyphenchar\font45}
\DeclareFontShape{U}{mathx}{m}{n}{
      <5> <6> <7> <8> <9> <10>
      <10.95> <12> <14.4> <17.28> <20.74> <24.88>
      mathx10
      }{}
\DeclareSymbolFont{mathx}{U}{mathx}{m}{n}
\DeclareFontSubstitution{U}{mathx}{m}{n}
\DeclareMathAccent{\widecheck}{0}{mathx}{"71}
\DeclareMathAccent{\wideparen}{0}{mathx}{"75}

\newtheorem{theorem}{Theorem} [section]

\newtheorem{lemma}[theorem]{Lemma}
\newtheorem{proposition}[theorem]{Proposition}

\theoremstyle{definition}
\newtheorem{remark}[theorem]{Remark}

\DeclareMathOperator*{\supp}{supp}

\newcommand{\noi}{\noindent}
\newcommand{\Z}{\mathbb{Z}}
\newcommand{\R}{\mathbb{R}}
\newcommand{\C}{\mathbb{C}}
\newcommand{\T}{\mathbb{T}}
\newcommand{\Tl}{\T_{\lambda}}
\newcommand{\Zl}{\Z_{\lambda}}

\let\Re=\undefined\DeclareMathOperator*{\Re}{Re}
\let\Im=\undefined\DeclareMathOperator*{\Im}{Im}

\newcommand{\ind}{\mathbf 1}

\numberwithin{equation}{section}
\numberwithin{theorem}{section}

\begin{document}
\selectlanguage{english}

\title[GWP of DNLS on the torus]{
Global well-posedness of the derivative nonlinear Schr\"odinger equation 
with periodic boundary condition in $H^{\frac12}$
}
\author{Razvan Mosincat}
\date{\today}

\address{
School of Mathematics\\
The University of Edinburgh 
and The Maxwell Institute for the Mathematical Sciences\\
James Clerk Maxwell Building\\
The King's Buildings\\
Peter Guthrie Tait Road, 
EH9 3FD, Edinburgh, United Kingdom}
\email{r.o.mosincat@sms.ed.ac.uk}




\begin{abstract}
We establish the global well-posedness of the derivative nonlinear Schr\"{o}dinger equation with periodic boundary condition in the Sobolev space $H^{1/2}$, provided that the mass of  initial data is less than $4\pi$. 
This result matches the one by Miao, Wu, and Xu and its recent mass threshold improvement by Guo and Wu 
in the non-periodic setting.  
Below $H^{1/2}$, 
we show that the uniform continuity of the solution map on bounded subsets 
of $H^s$ does not hold, for any gauge equivalent equation.
\end{abstract}

\maketitle
\tableofcontents

\newpage
\section{Introduction}

We consider the initial value problem for the derivative nonlinear Schr\"{o}dinger equation (DNLS) 
\footnote{A parameter in front of the nonlinearity is irrelevant in this study since on the Fourier side the nonlinearity is unsigned; the equation does not have a definite focusing or defocusing character.} 
with periodic boundary condition 
\begin{equation}
\label{DNLS}
\begin{cases}
 i\partial_t u +\partial_x^2 u = i\partial_x (|u|^2u) \ ,\ (t,x)\in\R\times\T \\
 u(0,x)=u_0(x)\ ,\ u_0\in H^s(\T)
\end{cases} ,
\end{equation}
where $\T:=\R/2\pi\Z\simeq [0,2\pi)$. 
This equation was derived in the plasma physics literature in the 1970s \cite{Rogister71,Mio76, Mjolhus76}   
and it is a particular case of a perturbed complex Ginzburg-Landau equation \cite{vanSaarloosHohenberg}. 
Kaup and Newell \cite{KaupNewell} showed that \eqref{DNLS}  
is completely integrable, 
in the sense that it is the compatibility condition for a certain pair of linear differential equations.  
In particular, it possesses an infinite family of conservation laws, as well as a two-parameter family of solitons. 
In this work, we only employ the conservation of the following integrals of motion,  referred to as the 
\emph{mass}, \emph{momentum}, and \emph{energy} (of a solution $u$), respectively:
\begin{align}
\label{eqnintro:M}
M[u] &:= \int |u|^2dx ,\\
\label{eqnintro:P}
P[u] &:= \int\Im{(u\partial_x \overline{u})} + \frac12 |u|^4dx ,\\
\label{eqnintro:E}
E[u] &:=\int |\partial_xu|^2 +\frac32 |u|^2 \Im(u\partial_x\overline{u}) + \frac12 |u|^6\, dx.
\end{align} 
We note that these quantities are at the levels $L^2$, $H^{\frac12}$, and $H^1$, respectively; 
the Hamiltonian for \eqref{DNLS} is $P[u]$. 
Given a sufficiently regular solution $u$, 
one can check by direct computation the conservation of the above functionals 
(see for example \cite[Appendix~B]{HerrIMRN06}).  
In addition, one can similarly verify that the \emph{mean} $\int u(t,x)\,dx$ is also conserved. 

The scaling symmetry of this equation in the Euclidean setting is given by the invariance of solutions under the following transformation:
\begin{equation}
\label{naturalscaling}
u(t,x)\mapsto \frac{1}{\lambda^{\frac12}}u\left(\frac{t}{\lambda^2},\frac{x}{\lambda}\right)=:u^{\lambda}(t,x).
\end{equation}
In particular, $\|u^{\lambda}(t)\|_{L_x^2}= \|u(t/\lambda^{2})\|_{L^2_x}$,  
and thus $s_c=0$ is the scaling critical Sobolev index 
(reasonable well-posedness theory is to be expected for $s>s_c$, and possibly for $s=s_c$). 
In the periodic setting, the above transformation changes the underlying domain $\T$ 
to $\Tl:=\R/2\pi\lambda\Z$  (hence, the scaling transformation above is no longer a symmetry of the equation on $\T$). 
Nevertheless, we employ \eqref{naturalscaling} in this article  
and for the most part, we work on the dilated torus $\Tl$, with $\lambda\geq1$.

The aim of this paper is to study the long-time dynamics of low-regularity solutions of  \eqref{DNLS}. 
The main result of this work reads as follows:

\begin{theorem}
\label{thm1}
The derivative nonlinear Schr\"{o}dinger equation with periodic boundary condition \eqref{DNLS} 
is globally well-posed in $H^{\frac12}$ 
for initial data $u_0$ with $M[u_0]<4\pi$. 
\end{theorem}

Before delving into the prerequisites and proof of this theorem, 
let us promptly review the well-posedness results in the Euclidean setting. 
For a certain class of  Schwartz initial data, 
by using the inverse scattering method, 
Lee \cite{LeeThesis,Lee89} obtained the local and global\footnote{We refer to the recent article of Pelinovsky and Shimabukuro \cite[p.~4]{PelinovskyShimabukuro16} for a possible issue regarding the result of \cite{Lee89}.} 
solvability, respectively. 
Tsutsumi and Fukuda \cite{TsutsumiFukuda1} established the local existence and uniqueness 
of $H^s$-solutions, with $s>\frac32$, by the method of parabolic regularization. 
Furthermore, in \cite{TsutsumiFukuda2}, 
they obtained the global existence of solutions for $u_0\in H^2$ with sufficiently small $H^1$-norm. 

In the energy space, Hayashi \cite{Hayashi93} 
proved the global existence for sufficiently small initial data. 
In his work, of particular importance is the gauge transformation defined by
\begin{equation}
\label{gaugetransform}
\mathcal{G}_{\beta}:L^2\to L^2\,,\ \mathcal{G}_{\beta}f (x):= e^{-i\beta \mathcal{J}(f)(x)}f(x) ,
\end{equation}
with $\mathcal{J}(f)(x) := \int_{-\infty}^x |f(y)|^2\,dy$, $x\in\R$. 
Through the transformation $v=\mathcal{G}_{\frac12}(u)$, 
DNLS reduces to
\begin{equation}
i\partial_t v +\partial_x^2 v =i |v|^2\partial_xv
\end{equation}
and the theory of \cite{TsutsumiFukuda2} for smooth initial data with sufficiently small $H^1$-norm can be applied. 
This result was improved in the papers by Hayashi and Ozawa \cite{HayashiOzawa,HayashiOzawaSIAM} by reducing 
DNLS to a system of two semi-linear Schr\"{o}dinger equations (with no derivatives in the nonlinearities), 
where it was obtained\footnote{
Hayashi and Ozawa \cite{HayashiOzawa,HayashiOzawaSIAM} also showed that 
the solution map $u_0\mapsto u(t)$ preserves Sobolev regularity and spatial decay, 
for any $t\in\R$.} the global existence of $H^1$-solutions under the assumption 
\begin{equation}
\label{eqnintro:2pithreshold}
M[u_0]<2\pi .
\end{equation}
This mass threshold follows from noticing that under the transformation 
$v=\mathcal{G}_{\frac34}(u)$, the energy functional becomes
\footnote{If $u$ solves DNLS, then $v=\mathcal{G}_{\frac34}(u)$ solves 
$i\partial_t v + \partial_x^2v = \frac{i}{2}|v|^2\partial_x v -\frac{i}{2} v^2\partial_x\overline{v} -\frac{3}{16}|v|^4v$.}
\begin{equation}
\label{eqnintro:E34}
 E[u]= E[\mathcal{G}_{-\frac34}(v)]=\|\partial_x v\|_{L^2(\R)}^2 - \frac{1}{16} \|v\|_{L^6(\R)}^6,
\end{equation}
while the norms of interest remain essentially unchanged, i.e. 
\begin{align}
\|\partial_x v\|_{L^2(\R)} &\sim_{M[u]} \|\partial_x u\|_{L^2(\R)} , \\
\|v\|_{L^p(\R)} &= \|u\|_{L^p(\R)} .
\end{align}
Via the sharp Gagliardo-Nirenberg inequality due to Weinstein \cite{Weinstein83}, i.e. 
\begin{equation}
\label{eqnintro:GN1}
\|v\|_{L^6(\R)}^6 \leq \frac{4}{\pi^2} \|\partial_x v\|_{L^2(\R)}^{2} \|v\|_{L^2(\R)}^{4} ,
\end{equation}
one obtains that the energy $E[u]$ controls the $\dot{H}^1$-norm of a solution $u$ of DNLS, 
provided that   \eqref{eqnintro:2pithreshold} holds. 
Finally, we mention that in \cite{Hayashi93,HayashiOzawa}, uniqueness of $H^1(\R)$-solutions was also obtained, but conditional to the auxiliary spaces $L_t^{12}(\R;H_x^{1,3}(\R))$, $L^4_t(\R; W_x^{1,\infty}(\R))$, respectively. 
The unconditional well-posedness in $H^1(\R)$ was settled by Win \cite{WinKyoto08}. 

In low-regularity spaces, Takaoka  \cite{TakaokaADE} 
used the Fourier restriction norm spaces introduced by Bourgain \cite{BourgainGAFA93} and proved 
\begin{equation}
\label{eqnintro:Takaokatrilinest}
\|v^2\partial_x\overline{v}\|_{X^{s,b-1}(\R\times\R)} \lesssim  \|v\|^3_{X^{s,b}(\R\times\R)}, 
\end{equation}
for $\frac12\leq s<1$ and $\frac12<b\leq\frac58$, 
in a similar fashion to the estimate for the KdV equation \cite{KPVjams96}. 
We remind the reader that the Fourier restriction norms $\|\cdot\|_{X^{s,b}(\R\times\R)}$ 
are adapted to the linear part of the equation at hand 
(defined by \eqref{defn:Xsbnorms} on $\T$, their definition on $\R$ being analogous). 
On the other hand, 
he noted that for estimates of the form 
\begin{equation}
\||v|^2\partial_x v\|_{X^{s,b-1}(\R\times\R)} \lesssim  \|v\|^3_{X^{s,b}(\R\times\R)}, 
\end{equation}
``the Fourier restriction norm method seems inapplicable.''  
However,  the transformation\footnote{This reduction of DNLS 
to \eqref{eqnintro:gaugedDNLSonR} was also employed by Lee \cite{Lee89},  
to which he attached a certain spectral problem.} 
$v=\mathcal{G}_1(u)$ removes the nonlinearity $|u|^2\partial_x u$ from \eqref{DNLS},  
as $v$ solves 
\begin{equation}
\label{eqnintro:gaugedDNLSonR}
i\partial_t v + \partial_x^2 v = -i v^2\partial_x \overline{v} -\frac12|v|^4v. 
\end{equation}
Therefore,\footnote{The estimate for the quintic term of \eqref{eqnintro:gaugedDNLSonR} 
is easier to prove than for the cubic-derivative term.} Takaoka established 
the local well-posedness of \eqref{DNLS} down to $H^{\frac12}(\R)$. 
Since the mappings $u_0 \mapsto \mathcal{G}_1(u_0)=v_0$ 
and 
$v(t)\mapsto\mathcal{G}_{1}^{-1}(v(t))=u(t)$,  
viewed as applications from $H^{\frac12}(\R)$ to itself, 
are locally Lipschitz continuous (uniformly in $t$),  
so is the dependence of $u$ on the initial data $u_0$.   
Takaoka also showed that the above estimate \eqref{eqnintro:Takaokatrilinest} 
does not hold if $s<\frac12$, for any $b\in\R$. 
Moreover, for $0\leq s<\frac12$, 
the solution map $u_0\in H^s(\R) \mapsto u(t)\in H^s(\R)$ 
fails to be $C^3$, for any $t\neq0$. 
Another mild ill-posedness result for DNLS in $H^s(\R)$  ($0\leq s<\frac12$) was given by 
Biagioni and Linares in \cite{BiagioniLinares}. 
They used the solitary waves of DNLS \cite{KaupNewell,vanSaarloosHohenberg} 
and showed that 
the local uniform continuity of the same solution map does not hold. 
Hence, the fixed point argument for the gauge equivalent equation \eqref{eqnintro:gaugedDNLSonR} 
is no longer the tool to construct $H^s(\R)$-solutions for DNLS in the range 
$0\leq s<\frac12$. 

Let us now elaborate on the assumption \eqref{eqnintro:2pithreshold}. 
The energy functional in \eqref{eqnintro:E34}
is also shared by the focusing quintic NLS 
(with some appropriate constant in front of the nonlinearity),  
for which the condition \eqref{eqnintro:2pithreshold} is sharp, 
in the sense that 
finite-time blowup solutions with $M[u]\geq 2\pi$ exist 
(see \cite{Weinstein83} and references therein). 
Thus, the question of whether the same is true for DNLS appears naturally (see also \cite{HayashiOzawa}). 
However, in a recent series of articles, Wu \cite{Wu2013,WuAPDE2}, and Guo and Wu \cite{GuoWu} 
obtained global existence for DNLS above the mass threshold $2\pi$. 
They showed how to 
incorporate the momentum $P[u]$ in controlling the $\dot{H}^1$-norm of $u$. 
The key observation is the following:  
the change in energy incurred by modulating $u$ 
resembles the first term of the momentum (see \eqref{eqnintro:P}), 
and then  the second term of $P[u]$ is handled 
by another sharp Gagliardo-Nirenberg inequality  
due to  Agueh \cite{Agueh06},  
that interpolates $L^6(\R)$ between $\dot{H}^1(\R)$ and $L^4(\R)$ (rather than $L^2(\R)$), i.e. 
\begin{equation}
\label{eqnintro:GN2}
\|u\|_{L^6(\R)} \leq C_{\textup{GN}} \|\partial_x u\|_{L^2(\R)}^{\frac19} \|u\|_{L^4(\R)}^{\frac89} , 
\end{equation}
where $C_{\textup{GN}}=3^{\frac16} (2\pi)^{-\frac19}$. 
The upshot is the control of the $\dot{H}^1$-norm of a solution $u$, under 
\begin{equation}
\label{eqnintro:4pithreshold}
M[u_0]<4\pi.
\end{equation}
It is known that the mass thresholds $2\pi$ and $4\pi$ correspond to the masses 
of ground state solutions 
to some elliptic equations 
and extremising functions in the  Gagliardo-Nirenberg inequalities 
\eqref{eqnintro:GN1} and \eqref{eqnintro:GN2}, respectively. 
In a recent article studying the orbital stability of solitary waves of DNLS with mass $4\pi$, 
Kwon and Wu \cite{KwonWu} proposed a criterion for blowup solutions with this critical mass. 

We now turn to the periodic setting.  
The adaptation of the gauge transformation is due to Herr 
\cite{HerrIMRN06} where he proved 
the local well-posedness of \eqref{DNLS} in $H^s(\T)$ for $s\geq \frac{1}{2}$, 
by using the same transformation \eqref{gaugetransform}, but with 
$\mathcal{J}(f)$ defined as the mean-zero antiderivative of $|f|^2-\frac{1}{2\pi}M[f]$. 
Compared to the real-line case, 
the gauge transformation is no longer enough to satisfactorily reduce DNLS to a manageable equation,  
and thus it needs to be augmented with a translation operator (see \eqref{defn:transloperator}).  
After these transformations, 
one works instead with a 
\emph{periodic gauge equivalent derivative nonlinear Schr\"{o}dinger equation}, that still 
resembles \eqref{eqnintro:gaugedDNLSonR}, of the form
\begin{equation}
\label{gDNLS}
i\partial_t v +\partial^2_x v = 
-iv^2\partial_x \overline{v} +\mathcal{Q}(v) ,
\end{equation}
where  $\mathcal{Q}(v)$ gathers (pure-power) quintic  and lower order terms 
(a precise formulation of the above equation is given by \eqref{g1DNLS}).  
For the key estimate in \cite{HerrIMRN06}, i.e. 
\begin{equation}
\label{eqnintro:Herrcubicderivest}
\|v^2\partial_x\overline{v}\|_{X^{s,-\frac12}(\R\times\T)} \lesssim \|v\|^3_{X^{s,\frac12}(\R\times\T)}
\end{equation}
for $s\geq\frac12$, 
the local smoothing and maximal function estimates for the linear Schr\"{o}dinger propagator 
are no longer available -- one has to rely merely on the $L^4$-Strichartz estimate of Bourgain \cite{BourgainGAFA93} 
and on Sobolev inequalities. 
Since the embedding of $X^{s,\frac12}(\R\times\T)$ into $C(\R;H^s(\T))$ fails, 
one  works instead with a slightly stronger norm (see \eqref{defn:Ysbnorm}-\eqref{defn:Zsnorm} below). 

Regarding possible improvements to the estimate \eqref{eqnintro:Herrcubicderivest}, 
a remarkable property of \eqref{gDNLS} 
was uncovered 
by Gr\"{u}nrock and Herr \cite{GrunrockHerr}
while working in the scale of Fourier-Lebesgue spaces 
$\mathcal{F}L^{s,r}(\T)$ (defined by \eqref{defn:FLnorm} below). 
By appropriately modifying the classical $X^{s,b}$ norms, 
they established a local well-posedness result in 
$\mathcal{F}L^{\frac12,r}(\T)$, in the range $2\leq r <4$. 
(we note that for $s=\frac12$, 
these spaces scale like $H^{\frac{1}{r}}(\T)$). 
It turns out that the nonlinearity of \eqref{gDNLS} can be rearranged into  
$\widetilde{\mathcal{T}}(v)+\widetilde{\mathcal{Q}}(v)$, with 
\begin{equation}
\label{defnofmathcalTprime}
\widetilde{\mathcal{T}}(v):= -i \left(v\partial_x \overline{v} -2i \Im\fint_{\T} v\partial_x \overline{v}\,dx\right)v ,\\
\end{equation}
owing to the periodic gauge transformation. 
This writing reveals frequency cancelations (see also \eqref{freqrestr})    
that are essential for dealing with the cubic-derivative nonlinearity, 
and hence for the local well-posedness in $\mathcal{F}L^{s,r}(\T)$. 
However, \eqref{defnofmathcalTprime} does not help 
when trying to prove \eqref{eqnintro:Herrcubicderivest} for some $s<\frac12$,  
as the failure of this estimate is of another nature than the lack of such frequency cancelations. 

We also point out that, specific to the periodic case,  
the Lipschitz continuity of the solution map of DNLS on bounded subsets of $H^s(\T)$ 
is further restricted to subsets with prescribed $L^2$-norm 
due to the use of a translation operator 
when reversing the transformations leading to \eqref{gDNLS} back to \eqref{DNLS}  
(see Lemma~\ref{lemma:gaugecontinuity}). 
In fact, the local uniform continuity of the solution map of the periodic DNLS fails 
without fixing the mass on bounded subsets of $H^s(\T)$, at any regularity level
(see \cite[Theorem~3.1.1.(ii)]{HerrThesis}). 
Although counterexamples to 
the trilinear estimate \eqref{eqnintro:Herrcubicderivest} for $s<\frac12$ 
were given in \cite{GrunrockHerr}, 
to the best of the author's knowledge, 
a direct argument towards the (mild) ill-posedness of \eqref{DNLS} in $H^s(\T)$ 
can not be found in the literature. 
Note that for the gauge equivalent equation \eqref{gDNLS}, 
one does not face the local uniform continuity bottleneck due to the translation operator 
and it was for this equation that the contraction mapping argument was applied in \cite{HerrIMRN06}. 
Hence, we provide here the following mild ill-posedness result. 
The mild sense refers to the fact that the result shows  
that the contraction mapping argument cannot be applied 
for the gauge equivalent equation \eqref{gDNLS}. 
In fact, the same is true for any equation obtained from \eqref{DNLS} 
through a gauge transformation of the form \eqref{gaugetransform} 
(see Appendix~\ref{appdx:illposedness}). 

\begin{proposition}
\label{prop1p2}
Suppose $v_0\in H^s(\T)\mapsto v=:S_T(v_0)\in C([-T,T];H^s(\T))$ 
is the solution map for \eqref{gDNLS}, for some $0\leq s<\frac12$, $T>0$. 
Then, $S_T$ is not uniformly continuous on bounded subsets of $H^s(\T)$. 
\end{proposition}

In view of all of the existing results and the two main results established in this work 
(Theorem~\ref{thm1} and Proposition~\ref{prop1p2}), 
at the moment 
the analytical existence theory in Sobolev spaces for DNLS on the torus 
and on the real-line 
match exactly. 
As noted by Herr \cite{HerrIMRN06}, 
``this is different 
for numerous other nonlinear Schr\"{o}dinger or KdV equations.'' 
We refer to the remarks at the end of this section for some further speculative comments. 

Turning our attention to global-in-time solutions of DNLS, 
we recall the reader that the global existence  theory below $H^1(\R)$, 
under the smallness of mass condition \eqref{eqnintro:2pithreshold}, 
was developed  by Colliander, Keel, Staffilani, Takaoka, and Tao \cite{CKSTT,CKSTTrefined}.  
First, in \cite{CKSTT}, they introduced the so called ``$I$-method''  
which aims to control the growth of $\mathcal{E}^1[u]:=\mathcal{E}[Iu]$, 
where $I$ (defined in \eqref{defnofIoperator}) 
is a smoothing operator on large frequencies and it is the identity operator 
for functions supported on small frequencies  
(here, $\mathcal{E}$ stands for the energy of \eqref{eqnintro:gaugedDNLSonR} rather than of DNLS). 
Relying on the $L^6$-Strichartz estimate and on a bilinear $L^4$-Strichartz inequality 
(in addition to the usual Sobolev embeddings), 
they obtain the global existence in $H^s(\R)$, for $s>2/3$. 
We also mention here that the precursor to this method, the high-low method of Bourgain,   
was also implemented to DNLS on the real-line by Takaoka in \cite{TakaokaEJDE},  
where he constructed global solutions for $s>{32}/{33}$. 
A stronger result was established in \cite{CKSTTrefined} 
for the second generation modified energy $\mathcal{E}^2[u]$ 
which is obtained from $\mathcal{E}^1[u]$ via a small correction term, 
thus lowering the regularity to $s>1/2$. 
What is  critical to the later result 
is an improved trilinear estimate \eqref{eqnintro:Takaokatrilinest}
in the presence of the $I$-operator 
that can reach $b=3/4-\varepsilon$, 
as well as delicate cancelations 
owing to the nonlinear structure $-i v^2\partial_x \overline{v} -\frac12|v|^4v$ of 
\eqref{eqnintro:gaugedDNLSonR}. 
Miao, Wu, and Xu \cite{MiaoWuXu} 
closed the gap between the local well-posedness range 
and the output of the $I$-method by 
iterating the scheme further via a third generation modified energy 
$\mathcal{E}^3[u]$  by adding another correction term to $\mathcal{E}^2[u]$. 
This time, the correction term has a singular set more complicated than 
that of the correction term in \cite{CKSTTrefined}, 
and thus it requires an intricate resonant decomposition. 
Finally, we mention here that recently, 
Guo and Wu \cite{GuoWu} improved the mass threshold 
to \eqref{eqnintro:4pithreshold} 
for global $H^{1/2}(\R)$-solutions. 

Regarding the global existence question for the periodic problem \eqref{DNLS}, 
Herr  \cite{HerrIMRN06} answered in affirmative  for $H^1(\T)$-solutions,  
under the assumption that the mass is smaller than $2/3$.  
This followed routinely from iterating the local well-posedness result 
together with a coercivity property of $E[\,\cdot\,]$ 
(obtained 
by using a non-optimal Gagliardo-Nirenberg-type inequality 
and without appealing to another gauge choice). 
In the present work, 
we also show that the mass threshold under which the energy functional 
associated to \eqref{DNLS} has the coercivity property is the same as in the real-line case 
(see Lemma~\ref{EnergyscriptEcontrolshomogH1} below).  
However, the smallness of mass condition for  global $H^1(\T)$-solutions 
was already improved to \eqref{eqnintro:4pithreshold} 
by the author and Oh in \cite{MosincatOh2015}. 
Below the energy space, in $H^s(\T)$, for $s>1/2$, and unquantified small mass initial data,  
the global existence of solutions to \eqref{DNLS} 
was  studied by Win in \cite{WinFE2010} 
by using a second generation modified energy for the $I$-method.

The main tools we use in obtaining Theorem~\ref{thm1} are the following: 
(i) the gauge transformation (and its Lipschitz continuity property) and the multi-linear estimates 
 due to Herr \cite{HerrIMRN06}, 
(ii) a third generation modified energy in the $I$-method scheme of 
Colliander, Keel, Staffilani, Takaoka, and Tao  \cite{CKSTT,CKSTTrefined},  
(iii)  a resonant decomposition of one of the correction terms 
as implemented in the real-line case by Miao, Wu, and Xu \cite{MiaoWuXu}, 
(iv) a revised bilinear $L^4$-Strichartz estimate of De Silva, Pavlovi\'{c}, Staffilani, and Tzirakis \cite{deSilva2007}, 
and 
(v) the sharp Gagliardo-Nirenberg inequalities 
\eqref{eqnintro:GN1}, \eqref{eqnintro:GN2} 
adapted to the periodic setting. 

Before outlining the contents of this article, we end the introduction with the following remarks. 

\begin{remark}
Our method is an analytical argument, 
in particular we do not employ the integrable structure of \eqref{DNLS} in an explicit manner. 
Therefore, the technique can be applied 
(cf. \cite{CKSTTrefined,HerrIMRN06, TakaokaADE, Takaoka2015}) 
to equations of the form 
\begin{equation}
i\partial_t u+\partial_x^2 u = i\alpha|u|^2\partial_x u + i\beta u^2\partial_x\overline{u} +  \gamma |u|^4u,
\end{equation}
where $\alpha,\beta,\gamma\in\R$. 
The first term is eliminated by the gauge transformation $v=\mathcal{G}_{\alpha/2}(u)$. 
\end{remark}

\begin{remark} 
\label{rmk:KwonWu} 
Whether DNLS exhibits finite-time blowup solutions  for initial data with high mass 
 remains an important open question. 
A numerical study by Liu, Simpson and Sulem \cite{LiuSimpsonSulem} indicates that there is finite time singularity 
for the $L^2$-supercritical nonlinearity $i|u|^{2\sigma}\partial_x u$,  
$\sigma>1$. 
We also refer to \cite{CherSimpsonSulem} for a further numerical investigation of the structure of the singular profile 
near blowup times. 
Here, we also point out that 
negative energy solutions 
to DNLS on the half-line with Dirichlet boundary condition 
blow up in finite time (see \cite{Wu2013}). 
\end{remark}

\begin{remark}
The equation \eqref{DNLS} 
has a rich structure being completely integrable, 
and the inverse scattering transform (IST) on the torus 
might reveal some fundamental differences compared 
to the real-line case. 
In recent articles, Liu, Perry, and Sulem \cite{LiuPerrySulem15}, 
analyze more closely the IST method on $\R$ initiated by Lee \cite{LeeThesis} 
and prove global existence 
in a ``spectrally determined'' open subset of $H^{2,2}(\R)$ neighboring $0$.  
Also via IST, Pelinovsky and Shimabukuro \cite{PelinovskyShimabukuro16} construct unique global solutions in 
$H^2(\R)\cap H^{1,1}(\R)$ with Lipschitz continuity on the initial data. 
\end{remark}

\begin{remark}
In view of the local well-posedness result in the scale of Fourier-Lebesgue spaces 
by Gr\"{u}nrock and Herr \cite{GrunrockHerr}, 
it would be interesting to investigate via the $I$-method the global dynamics of the DNLS flow 
 in $\mathcal{F}L^{\frac12,r}(\T)$, for the appropriate range in $r$, 
 to complement the almost sure global well-posedness result of Nahmod, Oh,  Rey-Bellet, and Staffilani \cite{NORS}. 
This is also to be studied in the Euclidean case, 
 where the local well-posedness was established by Gr\"{u}nrock in \cite{GrunrockIMRN}. 
In the same direction of thought, we mention that recently, 
Takaoka \cite{Takaoka2015} 
proved the existence of local weak $H^s(\T)$-solutions with small (unquantified) mass 
in the range $\frac{12}{25}<s<\frac12$ 
by establishing a priori estimates 
 for the same gauge equivalent equation \eqref{gDNLS}; 
also, his work \cite{Takaoka2016} on the energy exchange behavior for a variant\footnote{
The equation \eqref{gDNLS}  falls under what he terms the ``focusing'' case, 
while the article \cite{Takaoka2016} is concerned with the ``defocusing'' situation; 
also \eqref{gDNLS} has some additional lower order terms.  
} of  \eqref{gDNLS}, might provide further insight on the DNLS dynamics above the mass threshold $4\pi$. 
\end{remark}

As a summary, we outline the content of the present article. 
In Section~\ref{Section2}, we introduce function spaces and review linear estimates  
(including a revised bilinear $L^4$-Strichartz estimate) 
that are used throughout the paper.   
After recalling the gauge transformation augmented with a translation operator, 
Theorem~\ref{thm1} is reduced to Proposition~\ref{prop:GWPofg1DNLS} 
concerning the global solutions of the periodic gauge equivalent equation \eqref{gDNLS}. 
Also in Section~\ref{Section2}, 
we provide the adaptation to the periodic setting of the Gagliardo-Nirenberg inequalities, 
and we introduce some further notation and constitutive elements 
for the $I$-method as close to previous 
implementations \cite{CKSTT,CKSTTrefined,deSilva2007,MiaoWuXu} as possible.  
Then, we  build up the $I$-method apparatus, beginning in Section~\ref{Sect:coercivity}. 
We first show that the coercivity property of the energy functional 
carries over to the periodic setting under the same mass condition \eqref{eqnintro:2pithreshold} 
as in the Euclidean setting,  
and then we incorporate the momentum functional 
to obtain $\dot{H}^1$-norm control for solutions corresponding to the improved mass threshold \eqref{eqnintro:4pithreshold}. 
In Section~\ref{sect:proofofLWPofIsystem},  
we provide a modified local well-posedness result 
based on  existing local multi-linear estimates and an interpolation lemma for the $I$-operator.   
In this instantiation of the $I$-method scheme, 
we construct a third generation modified energy functional in Section~\ref{sect:modifiedenergy} 
after revisiting the first and second generation energies, 
as well as discussing the frequency regions that previously 
did not  allow reaching the regularity $s=1/2$. 
In the same section, we also revisit the crafting of the resonant set  from the real-line setting 
and we provide pointwise bounds on multipliers which are used in the following two sections. 
Hence, in Section~\ref{Section:AlmostConservationEstimates} 
we analyze the growth of the third generation modified energy 
and conclude with its almost conservation property, 
whereas in Section~\ref{Sect:AlmostconservedEandP} we show 
that it stays close to the first generation modified energy. 
The almost conservation of the modified momentum follows similarly to the real-line case 
and is also established in Section~\ref{Sect:AlmostconservedEandP}. 
In Section~\ref{sect:proofofGWPg1DNLS}, we modify the usual $I$-method argument 
to include the almost conserved momentum and we finish the proof of Proposition~\ref{prop:GWPofg1DNLS}. 
Finally, the Appendix provides the proof of a technical lemma 
that immediately implies Proposition~\ref{prop1p2}. 

\vspace{3mm}
\noindent
\textbf{Acknowledgements.} 
The author would like to thank his advisor, Tadahiro Oh, 
for suggesting this problem 
and for providing several fruitful ideas. 
He is also thankful to Yuzhao Wang, Soonsik Kwon, and Professor Yoshio Tsutsumi 
for their availability in discussing various aspects of this problem.

\section{Notations and basic estimates}
\label{Section2}

A quantity of the form $\alpha\pm$ is a shorthand for $\alpha\pm\varepsilon$ with $\alpha\in\R$ and 
$\varepsilon>0$ arbitrarily small 
(if two or more such quantities appear in the same relation, 
the dependence between $\varepsilon$'s is straightforward and can be ignored). 
By $A\lesssim B$ 
($A\ll B$) we mean 
the existence of a large positive constant $C$ such that 
$A\leq CB$ ($CA\leq B$);  
we write $A\sim B$ if and only if $A\lesssim B$ and $B\lesssim A$. 
Also, we write $A=B+O(C)$ if and only if $|A-B|\lesssim C$. 
Throughout this work, 
$\eta$ denotes a smooth time cut-off function 
with $\eta\equiv 1$ on $[-1,1]$ and $\eta\equiv0$ outside $(-2,2)$. 

For reasons that are made clear in Section~\ref{sect:proofofGWPg1DNLS}, 
we need to use the  scaling transformation \eqref{naturalscaling}.   
Thus, we work on the parametrized torus $\Tl:=\R/2\pi\lambda\Z\simeq [0,2\pi\lambda)$, 
and Fourier modes in $\Zl:=\frac{1}{\lambda}\Z$. 
The convention 
we use  for the  (spatial) Fourier transform of a $2\pi \lambda$-periodic function  is 
$$\widehat{f}(k) = \int_0^{2\pi\lambda} e^{-ikx}f(x)dx\quad,\quad k\in\Zl$$
which is inverted by 
$$\widecheck{g}(x) =\frac{1}{2\pi\lambda} \sum_{k\in \Zl} e^{ikx}g(k)\quad,\quad x\in [0,2\pi\lambda] .$$
The convolution products on $\Tl$ and $\Zl$ are given by 
\begin{align*}
f*g(x)&=\int_0^{2\pi\lambda}f(x-y)g(y)\,dy,\\
a\star b(k) &= \frac{1}{2\pi\lambda}\sum_{h\in\Zl} a(k-h)b(h),
\end{align*}
respectively. 
We have 
$\widehat{fg}(k) = \widehat{f}\star\widehat{g}(k)$, 
and 
by endowing $\Zl$ with the scaled counting measure $(dk)_{\lambda}:=\frac{1}{2\pi\lambda}d\#$, 
the inner products on $L^2(\Tl)$ and $L^2(\Zl)$ are 
\begin{align*}
\langle f,g\rangle_{L^2(\Tl)} &= \int_0^{2\pi\lambda}f(x)\overline{g(x)}\,dx,\\
\langle a,b\rangle_{L^2(\Zl)} &= \frac{1}{2\pi\lambda} \sum_{k\in\Zl}a(k)\overline{b(k)} = 
\int_{\Zl}a(k)\overline{b(k)}(dk)_{\lambda} ,
\end{align*}
respectively. 
Then, the Parseval and Plancherel  identities are written as
\begin{align*}\langle f,\widecheck{a}\rangle_{L^2(\Tl)}  &= \langle \widehat{f} , {a}\rangle_{L^2(\Zl)} ,\\
\|f\|_{L^2(\Tl)} &= \|\widehat{f}\|_{L^2(\Zl)} .
\end{align*}
The Sobolev space $H^s(\T_{\lambda})$, respectively the Fourier Lebesgue space 
$\mathcal{F}L^{s,r}(\Tl)$ 
are the completion of the $2\pi\lambda$-periodic $C^{\infty}$ functions with respect to the norms 
\begin{align}
\label{defn:Sobnorm}
\|f\|_{H^s(\T_{\lambda})} &:=  \|\langle k\rangle^s \widehat{f}(k)\|_{L^2(\Zl)},\\
\label{defn:FLnorm}
 \|f\|_{\mathcal{F}L^{s,r}(\T_{\lambda})} &:= \|\langle k\rangle^s \widehat{f}(k)\|_{L^r(\Zl)} ,
\end{align}
where $\langle k\rangle:= (1+|k|^2)^{\frac12}$, $k\in\Zl$. 

\begin{remark}
\label{rmk:japanesebracketwithparam}
Notice that for any $k\neq0$, uniformly in $\lambda\geq 1$, we have
\begin{equation}
\label{uniformboundsofjapanesebracket}
|k|\leq \langle k\rangle \lesssim \lambda |k| , 
\end{equation}
and thus, in the periodic setting, 
$\langle k\rangle \sim 1$ for all $|k|\lesssim \lambda$. 
\end{remark}

By $\mathcal{S}_{\lambda}$ we denote the class of  
functions $u^{\lambda}:\R\times\Tl\to\C$ which are Schwartz in $t$, 
$2\pi\lambda$-periodic and $C^{\infty}$ in $x$. 
With a slight abuse of notation, the time-space Fourier transform and its inverse are
\begin{align*}
\widehat{u}(\tau,k) &=  \int_{\R} \int_{\Tl} e^{-i(\tau t +kx)}u(t,x)\,dx\,dt \quad,\quad \tau\in\R, k\in\Zl , \\
\widecheck{v}(t,x) &= \int_{\Zl} \int_{\R} e^{i(\tau t+ kx)}v(\tau,k)\, d\tau\, (dk)_{\lambda}  \quad,\quad t\in\R, x\in\Tl .
\end{align*}
Nonlinear interactions take on the Fourier side the form 
\begin{align*}
\widehat{uv}(\tau,k) = \widehat{u}\star\widehat{v}(\tau,k) &= 
 \frac{1}{2\pi\lambda}\sum_{k_1\in\Zl} \int_{\R} \widehat{u}(\tau_1,k_1)\widehat{v}(\tau-\tau_1,k-k_1)\, d\tau_1\\
&= \int_{k_1+k_2=k} \int_{\tau_1+\tau_2=\tau} \widehat{u}(\tau_1,k_1)\widehat{v}(\tau_2,k_2)\, d\tau_1\,(dk_1)_{\lambda} .
\end{align*}

The unitary group on $L^2(\Tl)$ determined by the linear Schr\"{o}dinger equation on $\Tl$ is given by
\begin{equation}
(U_{\lambda}(t)f)(x) = \frac{1}{2\pi\lambda} \sum_{k\in\Zl} e^{ikx+itk^2}\widehat{f}(k) .
\end{equation}
For $s,b\in\R$ (spatial and temporal regularity indices), we define the 
$X^{s,b}(\R\times\Tl)$ space as the completion of $\mathcal{S}_{\lambda}$ under
the norm
\begin{equation}
\label{defn:Xsbnorms}
\|u\|_{X^{s,b}(\R\times\Tl)} := 
 \|\langle k\rangle^s \langle \tau+k^2\rangle^b \widehat{u}(t,k)\|_{L^2_{\tau}L^2_k(\R\times\Zl)} .
\end{equation}
It is well known that the (continuous) embedding $X^{s,b}(\R\times\Tl)\subset C_t^0H_x^s(\R\times\Tl)$ holds if $b>\frac12$ and fails for $b=\frac12$. Since the trilinear estimate needed for the local well-posedness theory 
(see Lemma~\ref{HerrsLWPestimates} below) holds only at $b=\frac12$, we 
introduce the spaces $Y^{s,b}$ and 
$Z^s$ (with a slightly stronger norm than the $X^{s,\frac12}$-norm) via the norms
\begin{align}
\label{defn:Ysbnorm}
\|u\|_{Y^{s,b}(\R\times\Tl)} &:= 
 \|\langle k\rangle^s \langle \tau+k^2\rangle^b \widehat{u}(t,k)\|_{L^2_kL^1_{\tau}(\Zl\times\R)} ,\\
\label{defn:Zsnorm}
\|u\|_{Z^s(\R\times\Tl)} &:= 
  \|u\|_{X^{s,\frac12}(\R\times\Tl)} + \|u\|_{Y^{s,0}(\R\times\Tl)} ,
\end{align}
and the companion space $\widetilde{Z}^s$ by using
\begin{align}
\|u\|_{\widetilde{Z}^s(\R\times\Tl)} &:= 
  \|u\|_{X^{s,-\frac12}(\R\times\Tl)} + \|u\|_{Y^{s,-1}(\R\times\Tl)} .
\end{align}
We have $Y^{s,0}(\R\times\Tl)\subset C_t^0H_x^s(\R\times\Tl)$ and therefore 
$Z^s=X^{s,\frac12}\cap Y^{s,0}\subset C_t^0H_x^s$. 

For a given time interval $J$, the Fourier restriction norms 
are defined via 
\begin{align}
\label{restrictionXsb}
\|u\|_{X^{s,b}(J\times\Tl)} := \inf\{\|v\|_{X^{s,b}(\R\times\Tl)} : v_{|J}=u\} ,
\end{align}
and similarly for $Y^{s,b}(J\times\Tl)$, $Z^s(J\times\Tl)$, and $\widetilde{Z}^s(J\times\Tl)$. 

By the Riemann-Lebesgue lemma and H\"{o}lder inequality, we have 
\begin{align*}
\|u\|_{L^{\infty}_{t,x}(\R\times\Tl)} &\lesssim 
\left\| \, \|\widehat{u}\|_{L^1_{\tau}(\R)} \right\|_{L^1_k(\Zl)} \\
&\leq \left( \frac{1}{2\pi\lambda} \sum_{k\in\Zl} \langle k\rangle ^{-1-} \right)^{\frac12} 
\left(\frac{1}{2\pi\lambda} \sum_{k\in\Zl} \langle k\rangle^{1+} \|\widehat{u}(\tau,k)\|^2_{L^1_{\tau}(\R)} \right)^{\frac12}\\
&\lesssim \left\| \langle k\rangle^{\frac12+} 
  \|\widehat{u}\|_{L^1_{\tau}(\R)}\right\|_{L_k^2(\Zl)} 
\end{align*}
and thus
\begin{equation}
\label{embeddinginLinfty}
\|u\|_{L^{\infty}_{t,x}(\R\times\Tl)} 
\lesssim \|u\|_{Y^{\frac12+,0}(\R\times\Tl)} .
\end{equation}
Similarly, by Minkowski's integral inequality, Riemann-Lebesgue lemma and Plancherel's identity, one obtains
\begin{equation}
\label{embeddinginLtinftyHxs}
\|u\|_{L_t^{\infty}H_x^s(\R\times\Tl)} 
 \lesssim \|u\|_{Y^{s,0}(\R\times\Tl)} , 
\end{equation}
for any $s\in\R$. 

Additionally, we have the following linear estimates. 

\begin{lemma}{\cite[Lemma~3.6]{HerrIMRN06}}
Let $s\in\R$. There exists $c>0$ such that 
\begin{align}
\label{linearhomogest}
&\|\eta(t)U_{\lambda}(t)f\|_{Z^s(\R\times\Tl)} \leq c \|f\|_{H^s(\Tl)}\\
\label{linearinhomogest}
&\left\|\eta(t)\int_0^t U_{\lambda}(t-\tau)F(\tau,\cdot)d\tau\right\|_{Z^s(\R\times\Tl)} \leq 
 c \|F\|_{\widetilde{Z}^s(\R\times\Tl)}
\end{align}
for all $f\in H^s$ and all $F\in \mathcal{S}_{\lambda}$.
\end{lemma}

\begin{lemma}
\label{Lem:SobolevStrichartz}
Let $2\leq p,q< \infty$, $b\geq \frac12-\frac1p$, $s\geq \frac12-\frac1q$, $\lambda\geq1$. 
For $u\in\mathcal{S}_{\lambda}$, we have 
\begin{enumerate}
\item Sobolev estimates:
\begin{align}
& \|u\|_{L_t^pH_x^s(\R\times\Tl)}\lesssim \|u\|_{X^{s,b}(\R\times\Tl)} \label{Sobolev1} ,\\
& \|u\|_{L_t^{\infty}H_x^s(\R\times\Tl)}\lesssim \|u\|_{X^{s,\frac12+}(\R\times\Tl)} \label{Sobolev1infty},\\
& \|u\|_{L_t^pL_x^q(\R\times\Tl)}\lesssim \|u\|_{X^{s,b}(\R\times\Tl)} \label{Sobolev2},\\
& \|u\|_{L_t^{\infty}L_x^{\infty}(\R\times\Tl)}\lesssim \|u\|_{X^{\frac12+,\frac12+}(\R\times\Tl)}; \label{Sobolev2infty}
\end{align} 

\item Strichartz estimates:
\label{Lemma:PeriodicStrichartzEstimates}
\begin{align}
\label{L4Strichartz} & \|u\|_{L^4_{t,x}(\R\times\Tl)}\lesssim \|u\|_{X^{0,\frac38}(\R\times\Tl)} , \\
 \label{L6Strichartz}  
 & \|u\|_{L^6_{t,x}(\R\times\Tl)}\lesssim 
  \lambda^{0+} \|u\|_{X^{0+,\frac12+}(\R\times\Tl)}
\end{align}
\end{enumerate}
with implicit constants independent of $\lambda\geq 1$. 
\end{lemma} 
\noindent
One can prove the first part by using the interaction representation 
\begin{equation}
\label{interactreprofXsbnorm}
\|u\|_{X^{s,b}(\R\times\Tl)}= 
 \|U_{\lambda}(-t)u(t,x)\|_{H_x^s H_t^b(\Tl\times\R)}, 
\end{equation}
the classical Sobolev inequalities, 
Minkowski's integral inequality and the fact that the operators 
$U_{\lambda}(t)$ are unitary on $H_x^s(\Tl)$.
The second part can be justified by going over the 
Stichartz estimates due to Bourgain \cite{BourgainGAFA93} 
and revisiting the counting arguments, but
now accounting for Fourier modes in $\Zl$ rather than $\Z$ 
(e.g. there are $O(\lambda M)$ elements $k$ in $\Zl$ satisfying 
$|k|\lesssim M$, 
there is a normalizing factor in the measure placed on $\Zl$, etc.). 
It turns out that the $L^4$-Strichartz estimate has 
an implicit constant independent of $\lambda$, 
while the $L^6$-Strichartz estimate has a logarithmic loss in $\lambda$ 
(in addition to the loss in derivative). 

By interpolating the Strichartz estimate \eqref{L6Strichartz} 
with the Sobolev inequality \eqref{Sobolev2} (for $p=q=6$), 
we also have
\begin{equation}
\label{interpL6Strichartz}
\|u\|_{L^6_{t,x}(\R\times\Tl)}\lesssim 
  \lambda^{0+} \|u\|_{X^{0+,\frac12-}(\R\times\Tl)} .
\end{equation}
We note that the estimates \eqref{Sobolev1}-\eqref{interpL6Strichartz} also hold for Fourier restriction norms on a time interval $J$ rather than on the entire real line. 

We record the following scaling properties 
of the space-time norms introduced above  
when using \eqref{naturalscaling} and a parameter $\lambda\geq1$:
\begin{align*}
\|u^{\lambda}\|_{L_t^pL_x^q(\R\times\Tl)} &= 
 \lambda^{\frac2p +\frac1q- \frac12} \|u\|_{L_t^pL_x^q(\R\times\T)} , \\
\|u^{\lambda}\|_{L_t^p\dot{H}_x^s(\R\times\Tl)} &= 
  \lambda^{-s+\frac2p} \|u\|_{L_t^p\dot{H}_x^s(\R\times\T)} ,
 \end{align*}
 and
 \begin{align*}
\lambda^{-s+\frac2p}\|u\|_{L_t^pH_x^s(\R\times\T)}& \lesssim   
\|u^{\lambda}\|_{L_t^pH_x^s(\R\times\Tl)} 
\lesssim \lambda^{\frac2p} \|u\|_{L_t^pH_x^s(\R\times\T)}. 
\end{align*}
For $s,b\geq 0$, we have
\begin{align*}
&
\lambda^{-1}\|u^{\lambda}\|_{X^{s,b}(\R\times\Tl)} \lesssim 
\|u\|_{X^{s,b}(\R\times\T)} \lesssim 
\lambda^{-1+s+2b} \|u^{\lambda}\|_{X^{s,b}(\R\times\Tl)} ,
\end{align*}
while for $s\geq 0$, $b<0$, we record 
\begin{align*}
&
\lambda^{-1+2b}\|u^{\lambda}\|_{X^{s,b}(\R\times\Tl)} \lesssim 
\|u\|_{X^{s,b}(\R\times\T)} \lesssim 
\lambda^{-1+s} \|u^{\lambda}\|_{X^{s,b}(\R\times\Tl)}. 
\end{align*}

We also use the following lemma when dealing with sharp time-cutoff functions:
\begin{lemma}
\label{sharpcutoffX112minus}
Let $s\in\R$ and suppose $\phi\in H_t^{\frac12-}(\R)$. Then:
$$ \|\phi u\|_{X^{s,\frac12-}(\R\times\Tl)} \lesssim \|\phi\|_{H_t^{\frac12-}(\R)} \|u\|_{X^{s,\frac12}(\R\times\Tl)} .$$
\end{lemma}
\begin{proof}
By \eqref{interactreprofXsbnorm},  
$$\|\phi u\|_{X^{s,b'}(\R\times\Tl)}  = \|\phi(t)U_{\lambda}(-t)u(t,x)\|_{H_x^sH_t^{b'}(\Tl\times\R)} $$
and let $J_t:=\langle \partial_t\rangle $.  
Then, via the fractional Leibniz rule, we have 
\begin{equation}
\label{LeibnizforLem2p4}
 \|\phi(t)U_{\lambda}(-t)u(t)\|_{H_t^{b'}} \lesssim 
 \|J_t^{b'}\phi\|_{L_t^{p}} \left\|U_{\lambda}(-t)u(t)\right\|_{L_t^{q}} + 
 \|\phi\|_{L_t^{q}} \left\|J_t^{b'} \big(U_{\lambda}(-t)u(t)\big)\right\|_{L_t^p} ,
\end{equation}
where $\frac1p + \frac1q=\frac12$. We take $b':=\frac12-<b<\frac12$ and $p>2$ so that we have
the continuous Sobolev embedding $H_t^{b-b'}(\R) \subset L_t^p(\R)$. 
Consequently, we also have the Sobolev embedding 
$H_t^b(\R)\subset L_t^q(\R)$. 
Then, the conclusion follows from \eqref{LeibnizforLem2p4} and triangle inequality for the $H_x^s(\R)$-norm. 
\end{proof}


\subsection{A bilinear $L^4$-Strichartz estimate} 
The following result is a key ingredient in the analysis 
of the almost conservation estimates  
as it is a refinement of the 
$L^4$-Strichartz estimate 
that provides a decaying factor in $\lambda$.  
Such an estimate is similar to 
the bilinear $L^4$-estimate 
in the non-periodic setting \cite[Lemma~7.1]{CKSTT}, 
and we point out that for $\lambda\to\infty$, we recover the same decay rate. 
For Schr\"{o}dinger evolutions on the one-dimensional torus, 
this estimate (but without pointing out the alternative (ii)) 
was first proved in \cite{deSilva2007}.  

\begin{lemma}
\label{bilinearL2Strichartz}
Let $\lambda\geq 1$, $N_1, N_2\in 2^{\Z}$ and suppose 
$\phi_1,\phi_2$ are smooth functions on $\Tl$ with 
$\supp(\widehat{\phi_j})\subset \{k\in\Zl : |k|\sim N_j\}$, $j=1,2$. 
Assume that either 
\begin{enumerate}
\item[\textup{(i)}] $N_1\gg N_2$, or 
\item[\textup{(ii)}] $N_1\sim N_2$ and $k_1k_2<0$ for all 
$k_1\in \supp(\widehat{\phi_1})$, $k_2\in \supp(\widehat{\phi_2})$.  
\end{enumerate}
Then
\begin{equation}
\label{bilinestoflinearsols}
\left\|\big(\eta(t)U_{\lambda}(t)\phi_1\big) \big(\eta(t)U_{\lambda}(t)\phi_2\big)\right\|_{L^2_{t,x}(\R\times\Tl)} 
\lesssim C(\lambda,N_1) 
\|\phi_1\|_{L^2_x(\Tl)} \|\phi_2\|_{L^2_x(\Tl)}
\end{equation}
where 
\begin{equation}
C(\lambda,N_1)= \begin{cases} 
1 &,\text{ if } N_1\lesssim 1\\
(\frac{1}{\lambda} +\frac{1}{N_1})^{\frac12} &,\text{ if } N_1\gg 1 
\end{cases} .
\end{equation}
Moreover, suppose $u_1,u_2\in \mathcal{S}_{\text{per}}$ are Fourier supported in 
$\{|k_1|\sim N_1\}$ and $\{|k_2|\sim N_2\}$, respectively, for all times $t$. Then,  
under the same assumption on the two frequency supports, 
we have
\begin{equation}
\label{interpbilinearest}
\|u_1u_2\|_{L^2_{t,x}(\R\times\Tl)} 
\lesssim_{\varepsilon} C(\lambda,N_1)^{1-2\varepsilon} 
\|u_1\|_{X^{0,\frac12-\varepsilon}(\R\times\Tl)}\|u_2\|_{X^{0,\frac12-\varepsilon}(\R\times\Tl)} ,
\end{equation}
for any $\varepsilon>0$ sufficiently small.
\end{lemma}

\begin{remark}
\label{rmk:bilinearmistakeinWinspaper}
In \cite[Proposition~2.1]{WinFE2010}, 
there seems to be a mistake in the case $N_1\sim N_2$: 
the two Fourier supports should be localized on opposite sides of the origin on the real line in order for \eqref{bilinestoflinearsols} to be true. 
The estimate with this additional assumption was used 
in proving Cases (2) and (3) of \cite[Lemma~7.5]{WinFE2010}. 
Although with simlar ideas as in 
the proof of \cite[Proposition~3.7]{deSilva2007}, 
we decided to present the proof 
so that this observation becomes clear.  
\end{remark}

\begin{proof}
By Plancherel's identity, the left hand side of \eqref{bilinestoflinearsols} becomes
\begin{gather*}
\begin{split}
\left\|\int_{\tau_1+\tau_2=\tau}\int_{k_1+k_2=k} 
 \widehat{\eta}(\tau_1+k_1^2)\widehat{\eta}(\tau_2+k_2^2)\widehat{\phi_1}(k_1)\widehat{\phi_2}(k_2) (dk_1)_{\lambda}d\tau_1
 \right\|_{L^2_{\tau}L^2_k} .
\end{split}
\end{gather*}
We denote $\psi:=\widehat{\eta}*\widehat{\eta}$, 
and without loss of generality, we can assume that $\psi$ is $\R$-valued and non-negative.  
\footnote{In general, we can write 
$\psi= \psi_+ - \psi_- +i \psi^+ -i \psi^-$ 
with the four components satisfying the non-negativity assumption, 
from where we can carry on analogous arguments for each of these terms.}
Then 
$$\int_{\R}  \widehat{\eta}(\tau_1+k_1^2)\widehat{\eta}(\tau- \tau_1+k_2^2)d\tau_1 
 = \psi(\tau +k_1^2+k_2^2)\geq0$$
and by H\"{o}lder's inequality, we have 
\begin{gather*}
\begin{split} 
&\left\| \int_{k_1+k_2=k} \psi(\tau+k_1^2+k_2^2)
  \widehat{\phi_1}(k_1)\widehat{\phi_2}(k_2) (dk_1)_{\lambda} \right\|_{L^2_{\tau}L^2_k} \\
&\hphantom{xxxxxxx}  
\leq 
\left\|\left(\int_{k_1+k_2=k} \psi(\tau+k_1^2+k_2^2(dk_1)_{\lambda}
   \right)^{\frac12} \right. \\
&\hphantom{xxxxxxxxxxxx}\times \left.
 \left( \int_{k_1+k_2=k} \psi(\tau + k_1^2+k_2^2)
   |\widehat{\phi_1}(k_1)|^2 
   |\widehat{\phi_2}(k_2)^2(dk_1)_{\lambda}
 \right)^{\frac12} 
   \right\|_{L^2_{\tau}L^2_k}\\
&\hphantom{xxxxxxx} \leq M \left(\int_{\Zl}\int_{\Zl}\int_{\R}\psi(\tau + k_1^2+k_2^2)
   |\widehat{\phi_1}(k_1)|^2 |\widehat{\phi_2}(k_2)|^2 \,d\tau\,(dk_1)_{\lambda}\,(dk)_{\lambda}\right)^{\frac12}\\
   &\hphantom{xxxxxxx}  \leq M \|\psi\|^{\frac12}_{L^1(\R)}\|\phi_1\|_{L^2(\Tl)} \|\phi_2\|_{L^2(\Tl)} ,
\end{split}
\end{gather*}
where we applied Fubini's theorem and we denoted
$$M:=\left(\sup_{k,\tau} \int_{k_1+k_2=k} \psi(\tau + k_1^2+k_2^2)(dk_1)_{\lambda}\right)^{\frac12} .$$
Thus, in order to obtain \eqref{bilinestoflinearsols}, it remains to show that $M\lesssim C(\lambda,N_1)$. 

Since $\psi$ is a Schwartz function, it is rapidly decaying, 
and so we can split $\R$ into disjoint intervals $I_j$ ($j\in\Z$) 
\footnote{If $\psi$ were compactly supported, 
it is enough to consider only one such interval, namely a finite-length interval which includes the support of $\psi$.} 
such that for all $j$ we have 
$|I_j|\sim 1$ and $\|\psi_{| I_j}\|_{L^{\infty}}\lesssim 2^{-|j|}$.   
Given $k\in\Zl$, $\tau\in\R$, and $j\in\Z$, 
we consider the set
$$S_{k,\tau,j} = \{k_1\in \Zl : k_1\in \supp(\widehat{\phi_1}) 
\,,\ k-k_1\in \supp(\widehat{\phi_2})\, ,\  
\tau + k_1^2 + (k-k_1)^2 \in I_j\} $$
and we estimate
$$M\lesssim 
\left(\sup_{k,\tau}\sum_{j\in\Z} \left(\frac{1}{\lambda}\#S_{k,\tau,j}\right) 2^{-|j|} \right)^{\frac12} ,$$
where $\#S_{k,\tau,j}$ denotes the cardinality of $S_{k,\tau,j}$. 

If $N_1\lesssim 1$, then clearly 
$$ \#S_{k,\tau,j} \leq \#\left\{k_1\in \frac{1}{\lambda}\Z : |k_1|\lesssim 1\right\} \lesssim \lambda$$ 
and thus $M\lesssim 1$. 

Now let us assume $N_1\gg 1$. 
To estimate the cardinality of a nonempty set $S_{k,\tau,j}$, 
we denote   
$$f_{k,\tau}(k_1):= \tau + k_1^2+(k-k_1)^2 .$$ 
Notice that 
\begin{equation} 
\label{counting:derivestimate}
|f_{k,\tau}'(k_1)| = 2|k_1 - (k-k_1)| \sim N_1,
\end{equation}
and  that this property holds not only when $N_1\gg N_2$ but also when 
$k_1$ and $k-k_1$ have opposite signs,  
and this is ensured by assumption (ii). 
From \eqref{counting:derivestimate} and the mean value theorem, 
we get that 
$$\# S_{k,\tau,j}\lesssim 1+\frac{\lambda}{N_1} ,$$ 
uniformly in $j$ 
(if $\lambda\lesssim N_1$ there might be inly one element in $S_{k,\tau,j}$). 

For the last part, by the transference principle for $X^{s,b}$ spaces 
(see for example \cite[Lemma~2.9]{TaoCBMS07}),  
the estimate \eqref{bilinestoflinearsols} implies 
\begin{equation}
\label{noninterpbilinearest}
\|u_1u_2\|_{L^2_{t,x}(\R\times\Tl)} \lesssim 
C(\lambda,N_1)\|u_1\|_{X^{0,\frac12+}(\R\times\Tl)} 
 \|u_2\|_{X^{0,\frac12+}(\R\times\Tl)} .
\end{equation}
On the other hand, 
by H\"{o}lder inequality and the $L^4$-Strichartz estimate, 
we have 
\begin{equation}
\label{L2Bourgain}
\|u_1u_2\|_{L^2_{t,x}(\R\times\Tl)} 
 \lesssim \|u_1\|_{X^{0,\frac38}(\R\times\Tl)}
  \|u_2\|_{X^{0,\frac38}(\R\times\Tl)}.
\end{equation}
By interpolating \eqref{noninterpbilinearest} and \eqref{L2Bourgain}, we 
obtain \eqref{interpbilinearest} 
for $\varepsilon>0$ sufficiently small.  
\end{proof}

\begin{remark}
We point out that the implicit constant in \eqref{interpbilinearest} depends on $\varepsilon$. 
Hence, we cannot 
disregard the logarithmic loss in the constant $C(\lambda, N_1)$. 
This loss is essentially the reason for which 
we need to introduce the second correction term in \eqref{defnofE3}
in the third iteration of the $I$-method (see also Remark~\ref{rmk:necessityofcorrectingforK41}).  
\end{remark}

\begin{remark}
Notice that, under assumption (i) of the above lemma, 
the estimate \eqref{interpbilinearest} holds 
if we replace one of the functions on the left hand side with its conjugate 
(or equivalently, one of the $X^{0,\frac12}$-norms in the right hand side 
with the $\overline{X}^{0,\frac12}$-norm as defined in \cite{HerrIMRN06}). 
This is no longer true under assumption (ii). 
\end{remark}

We use the above bilinear estimate essentially in the regime 
$1\leq \lambda\lesssim N_1$, 
and thus, in our estimates, $C(\lambda,N_1)\sim \lambda^{-\frac12}$. 

\subsection{Gagliardo-Nirenberg inequalities in the periodic setting}
\label{Sect:GNinequalities}
We recall that on the real-line, we have the sharp Gagliardo-Nirenberg inequalities
\begin{eqnarray}
\label{GNinequalityR}
\|f\|_{L^6(\R)} &\leq \left(\frac{2}{\pi}\right)^{\frac13} \|\partial_x f\|_{L^2(\R)}^{\frac13} \|f\|_{L^2(\R)}^{\frac23} ,\\
\label{GNinequalityRCGN}
\|f\|_{L^6(\R)} &\leq C_{\textup{GN}} \|\partial_x f\|_{L^2(\R)}^{\frac19} \|f\|_{L^4(\R)}^{\frac89} ,
\end{eqnarray} 
where $C_{\textup{GN}}:= 3^{\frac16} (2\pi)^{-\frac19}$. 
For \eqref{GNinequalityR} we refer to \cite{Weinstein83}, 
whereas for \eqref{GNinequalityRCGN}, see \cite{Agueh06}. 

\begin{remark}
\label{rmk:rigidityofQandW}
Extremising functions in  \eqref{GNinequalityR} are given by ground state solutions of 
$$-Q_{xx}+Q-\frac{3}{16}Q^5 =0\ , \ x\in\R$$ 
for which we have the Pohozaev identities 
$M[Q]=2\pi$ and $E[Q]=0$. 
We also know that  for any $u_0\in H^1(\R)$ with $M[u_0]<M[Q]$, we have $E[u_0]>0$. 
Such a $Q$ is also a stationary solution 
 of \eqref{gDNLS} on $\R$ with $\beta=\frac34$. 
Moreover, there exists $u_0\in H^1(\R)$ with $2\pi<M[u_0]<2\pi+\varepsilon$ such that 
$E[u_0]<0$. 

On the other hand, ground state solutions of 
$$-W_{xx}+W^3 -\frac{3}{16}W^5 =0\ , \ x\in\R$$ 
achieve the optimal constant $C_{\textup{GN}}$ in \eqref{GNinequalityRCGN} and have $M[W]=4\pi$. 
\end{remark}

On $\T$, 
inequalities of the above form cannot hold, 
simply for the fact that  constant functions provide counterexamples. 
However, the situation is similar to the Poincar\'{e} inequality, and in fact, 
using elementary arguments,
\footnote{Strictly speaking, \eqref{GNinequalityHerr} was proved on $\T$, but it is also true on $\Tl$ as the inequality is scale invariant. 
The same result can be obtained by using the pointwise Poincar\'{e} inequality followed by 
an application of the H\"{o}lder inequality.} 
it was shown in \cite[Appendix~C]{HerrIMRN06} the following inequality: 
\begin{equation}
\label{GNinequalityHerr}
\|(|f|^2-\mu[f])f\|_{L^2(\Tl)}\leq \|\partial_x f\|_{L^2(\Tl)} \|f\|_{L^2(\Tl)}^2.
\end{equation}
for any $2\pi\lambda$-periodic function $f$.
Although \eqref{GNinequalityHerr} can be used to study the coercivity of $E$ 
(see Lemma~\ref{EnergyscriptEcontrolshomogH1} below), 
we use here the following result (see e.g. Lebowitz, Rose and Speer \cite[Lemma~4.1]{LRS88}) 
since it yields the same mass threshold $M[u_0]<2\pi$ as in the Euclidean setting.

\begin{lemma}
\label{lem:LRSgagliardonirenberg}
For any $\varepsilon>0$, there exists a constant $K_{\varepsilon}>0$ (independent of $\lambda$) such that 
\begin{equation}
\label{GNinequalityT}
\|f\|_{L^6(\Tl)}^6 \leq \left(\frac{4}{\pi^2} +\varepsilon\right) \|\partial_x f\|^2_{L^2(\Tl)} \|f\|_{L^2(\Tl)}^4 + 
K_{\varepsilon}\|f\|_{L^2(\Tl)}^6,
\end{equation}
for all $f\in H^1(\Tl)$. 
\end{lemma}

With similar arguments, one can adapt \eqref{GNinequalityRCGN} to $\Tl$ as well.

\begin{lemma}Let $\delta>0$. We have 
\begin{equation}
\label{GNinequalityTCGN}
\|f\|_{L^6(\Tl)} \leq C_{GN}\left(1+\frac{\delta}{5\pi\lambda}\right)^{\frac29} 
 \left(\|\partial_x f\|_{L^2(\Tl)}^2+  \frac{1}{\pi\lambda\delta}\|f\|_{L^2(\Tl)}^2\right)^{\frac{1}{18}} \|f\|_{L^4(\Tl)}^{\frac89} 
\end{equation}
for all $f\in H^1(\Tl)$. 
\end{lemma}

The proof is a slight modification of \cite[Lemma~2.2]{MosincatOh2015} 
which used the $L^4$-norm rather than the $L^2$-norm of $f$ in the first factor on the right hand side above. 


\subsection{The gauge transformation on $\Tl$} 
Following Herr  \cite{HerrIMRN06, HerrThesis}, we consider
\begin{equation}
\label{defnofG}
\mathcal{G}_{\beta}: L_x^2(\Tl)\to L_x^2(\Tl)\quad,\quad 
\mathcal{G}_{\beta}(f)(x):= e^{-i\beta \mathcal{J}(f)(x)} f(x)\quad,
\end{equation}
where $\mathcal{J}(f)$ is the mean-zero antiderivative of $|f|^2-\mu(f)$, i.e. 
\begin{equation}
\label{defnofJ}
\mathcal{J}(f)(x):= \frac{1}{2\pi\lambda} \int_0^{2\pi\lambda} \int_{\theta}^x 
   				|f(y)|^2  - \mu[f] \,dy\,d\theta
\end{equation}
and 
\begin{equation}
\label{defnofmuf}
\mu[f]:=\frac{1}{2\pi\lambda} \|f\|_{L^2(\T_{\lambda})}^2 .
\end{equation}
Note that $|\mathcal{G}_{\beta}f| =|f|$ and therefore $\mu[\mathcal{G}_{\beta}f]=\mu[f]$; moreover 
$\mathcal{G}_{\beta}$  is inverted by $\mathcal{G}_{-\beta}$. 

Setting $w(t,x)=\mathcal{G}_{\beta}(u(t))(x)$,  
the derivative nonlinear Schr\"{o}dinger equation \eqref{DNLS} becomes 
\begin{gather}
\label{gDNLSw}
\begin{split}
i\partial_t w +\partial_x^2 w -2i\beta\mu[w] \,\partial_x w =& \,  
2i(1-\beta) |w|^2\partial_x w +i(1-2\beta)w^2\partial_x \overline{w} + \beta\mu[w] |w|^2w \\
& \ +(\frac{\beta}{2}-\beta^2)|w|^4w-\psi[w]w ,
\end{split}
\end{gather}
where 
\begin{equation}
\label{defnofpsi}
\psi[w]:= \frac{\beta}{2\pi\lambda} \int_{\Tl} \left(2\Im(w\overline{w}_x) + (\frac32-2\beta)|w|^4\right)dx + \beta^2\mu[w]^2 .
\end{equation}
Correspondingly, the momentum and energy functionals are
\begin{align}
\label{defn:Psubnu}
P[\mathcal{G}_{-\beta}(w)]&= \int_{\Tl}\left(\Im(w\overline{w}_x) +(\frac12-\beta)|w|^4 \right)dx + \beta \mu[w]M[w] 
 =:P_{\beta}[w] ,\\
\notag
E[\mathcal{G}_{-\beta}(w)]&= 
\int_{\Tl}\left(|w_x|^2 +(\frac32-2\beta)|w|^2\Im(w\overline{w}_x) +(\beta^2-\frac32\beta+\frac12)|w|^6 \right)dx\\
&\qquad +\frac{\beta}{2}\mu[w]\|w\|_{L_x^4}^4 
  +2\beta\mu[w]P_{\beta}[w] - \beta^2\mu[w]^2M[w]=:E_{\beta}[w].
\label{defn:Esubnu}
\end{align}
We point out that in the periodic setting, 
the terms coupled with $\mu[w]$ and $\psi[w]$ 
are new terms when comparing \eqref{gDNLSw} to the Euclidean setting.  

We can eliminate the auxiliary linear term on the left hand side of \eqref{gDNLSw} 
by  the translation transformation
\begin{equation}
\label{defn:transloperator}
w(t,x)\mapsto v(t,x+2\beta\mu[w]t).
\end{equation}
Correspondingly, we introduce the gauge transformation of spacetime functions  
\begin{equation}
\label{gaugetransfwithtranslation}
\mathcal{G}^{\beta}: C_t^0L_x^2(J\times\Tl)\to C_t^0L_x^2(J\times\Tl)\ ,\ 
\mathcal{G}^{\beta}(u)(t,x) := \mathcal{G}_{\beta}(u(t))(x-2\beta\mu[u(t)] t) .
\end{equation}

For the local well-posedness theory, 
it is necessary to use the gauge parameter $\beta=1$ 
so that the ``bad'' nonlinear term 
$|w|^2\partial_x w$ in \eqref{gDNLSw} is elliminated. 
Hence, in the sequel, we consider the equation on $\Tl$ corresponding to this gauge choice, namely
\begin{equation}
\label{g1DNLS}
i\partial_t v + \partial_x^2 v = -i v^2\partial_x \overline{v} -\frac12 |v|^4v +\mu[v]|v|^2v -\psi[v]v ,
\end{equation}
where we recall that $\mu[v]=\frac{1}{2\pi\lambda}\|v\|_{L^2(\Tl)}^2$ and 
\begin{equation}
\label{defnofpsinu1}
\psi[v]:= \frac{1}{2\pi\lambda} \int_{\Tl} \left(2\Im(v\overline{v}_x) -\frac12|v|^4\right)dx + \mu[v]^2 .
\end{equation}
\begin{remark} 
The nonlinearity of \eqref{g1DNLS} can be written in the form 
$\mathcal{N}=\widetilde{\mathcal{T}}+\widetilde{\mathcal{Q}}$ by grouping terms as follows:
\begin{align}
\label{defnofmathcalT}
\widetilde{\mathcal{T}}(v):=& -i \left(v\overline{v}_x -2i \fint_{\Tl} \Im(v\overline{v}_x)\,dx\right)v ,\\
\label{defnofmathcalQ}
\widetilde{\mathcal{Q}}(v):=& 
 -\frac12 \left(|v|^4 -\fint_{\Tl} |v|^4dx \right) + \fint_{\Tl}|v|^2dx \left(|v|^2 -  \fint_{\Tl}|v|^2dx\right)v.
\end{align}
An important observation here is that, on the Fourier side, 
by using the inclusion-exclusion principle we can write
\begin{equation}
\label{freqrestr}
\widehat{\widetilde{\mathcal{T}}(v)}(k) = \frac{1}{(2\pi\lambda)^2} 
\sum_{\substack{k=k_{123}\\ k\neq k_1, k_3}} k_2 \widehat{v}(k_1)\widehat{\overline{v}}(k_2)\widehat{v}(k_3) 
- \frac{1}{(2\pi\lambda)^2} k \widehat{v}(k)\widehat{\overline{v}}(-k)\widehat{v}(k) .
\end{equation}
It was made clear in \cite{GrunrockHerr} that 
the above frequency cancelations are essential 
in establishing the estimate that deals with the derivative-cubic nonlinearity 
in the scale  of Fourier-Lebesgue spaces $\mathcal{F}L^{s,r}(\T)$. 
However, in  Sobolev spaces $H^s(\T)$, for $s\geq\frac12$, 
one can handle the cubic-derivative term 
$u^2\partial_x{\overline{u}}$ without the frequency cancelations 
(see Lemma~\ref{HerrsLWPestimates} below due to Herr \cite{HerrIMRN06}). 
In \eqref{defnofmathcalQ}, 
the coefficients of the subtracted terms do not allow symmetrization,   
hence we do not have useful frequency cancelations. 
\end{remark}

The following lemma provides the continuity properties 
of the gauge transformation that  
allow to satisfactorily transfer the well-posedness results 
between various versions of \eqref{gDNLS}, including \eqref{DNLS} itself. 
We note that in order to have the Lipschitz continuity of $\mathcal{G}^{\beta}$ 
(rather than of $\mathcal{G}_{\beta}$) one needs to fix the $L^2$-norm of the functions at all times $t$.

\begin{lemma}{\cite[Lemma~2.3]{HerrIMRN06}}$ $
\label{lemma:gaugecontinuity}
\noi
Let $s,r,\mu_0\geq 0$, $T>0$. There exists $c=c(r,s,\lambda)>0$ such that: 
\begin{enumerate}
\item If $f,g\in B_r:=\{f\in H^s(\Tl) : \|f\|_{H^s(\Tl)} \leq r \}$, then 
\begin{equation}
\|\mathcal{G}_{\beta}(f) - \mathcal{G}_{\beta}(g) \|_{H^s(\Tl)} \leq c \|f-g\|_{H^s(\Tl)}.
\end{equation}
\item If $u,v\in B^{r,\mu_0}$, where 
$$B^{r,\mu_0}:=\{u\in C([-T,T]; H^s(\Tl)) : \|u\|_{L_t^{\infty}H_x^s} \leq r , \mu[u(t)]=\mu_0 \text{ for all }t\in[-T,T] \} ,$$ 
then
\begin{equation}
\label{continuityofGsuperscriptnu}
\|\mathcal{G}^{\beta}(u)(t) - \mathcal{G}^{\beta}(v)(t) \|_{H^s(\Tl)} \leq c \|u(t)-v(t)\|_{H^s(\Tl)} 
\end{equation}
for all $t\in [-T,T]$.
\end{enumerate}
\end{lemma}

We recall that the local well-posedness theory for \eqref{g1DNLS} 
via a fixed point argument in the space $Z^1$
was developed in \cite{HerrIMRN06, HerrThesis}
(see the estimates in Lemma~\ref{HerrsLWPestimates} below). 
Therefore, in order to get Theorem~\ref{thm1}, 
we aim to prove that the $H^s$-solutions of \eqref{g1DNLS} exist globally in time in the following sense:  
\begin{proposition}
\label{prop:GWPofg1DNLS}
Let $\frac12\leq s<1$ and $v_0\in H^s(\T)$ with $M[v_0]<4\pi$. Then 
for any $\varepsilon>0$, there exists $c=c(\|v_0\|_{H^s(\T)}, M[v_0],\varepsilon)<\infty$ 
such that for all $T>0$, the solution $v$ of \eqref{g1DNLS} with $v(0)=v_0$ satisfies
$$\sup_{0\leq t\leq T} \|v(t)\|_{H_x^s(\T)} \leq c (1+ T)^{2-2s+\varepsilon}.$$
\end{proposition}

Since the equation \eqref{g1DNLS} has the time reversibility symmetry 
$v(t,x)\mapsto \overline{v(-t,-x)}$ 
and the $L_x^2$-norm is conserved along the evolution, 
the above result 
implies that the $H_x^s$-norm of any solution $v$ of \eqref{g1DNLS} does not blow up in finite time.

\subsection{The $I$-operator}
\label{subsect:theIoperator}
For $0\leq s<1$ and $N\gg1$ a fixed dyadic number, 
we define the Fourier multiplication operator\footnote{
The operator $I$ depends on the regularity index $s$ and 
the parameters $N$ and $\lambda$, 
but we choose to omit  writing them as indices of $I$ whenever possible. 
However, in Lemma~\ref{CKSTTlemmaforIoperator} it becomes necessary to point them out explicitly.}
\begin{equation}
\label{defnofIoperator}
I: H^s(\T_{\lambda})\to H^1(\T_{\lambda})\ ,\ 
\widehat{If}(k) = m(k) \widehat{f}(k)\ , \ 
k\in\Zl
\end{equation}
where $m:\R \to (0,1]$ is an even, smooth, non-increasing function on $[0,\infty)$, 
chosen such that 
\begin{align*}
m(\xi)=
\begin{cases} 1 &\text{, if }  |\xi| \leq N \\
\left(\frac{N}{|\xi|}\right)^{1-s} &\text{, if } |\xi|\geq 2N
\end{cases}
\end{align*}
(and a smooth interpolant  for $|\xi|$ between $N$ and $2N$). 
Furthermore, for any $s\geq \frac12$, 
the Fourier multiplier $m(\,\cdot\,)$ can be chosen such that it satisfies the monotonicity property
\begin{equation}
\label{nondecrasingpropofmsgeq12}
\xi\mapsto m(\xi)\xi^{\frac12}  \text{ is non-decreasing on $[0,\infty)$.}
\end{equation}
One easily checks that, for any $0\leq \theta <1$, we have 
\begin{equation}
\label{symbol:lowbound}
m(k)\langle k \rangle^{1-\theta} \gtrsim 
\begin{cases}
N^{1-\theta} &, \text{ if } |k|\gg N\\
1 &, \text{ if } |k|\lesssim N
\end{cases}
\end{equation}
in the regularity range $\theta\leq s< 1$, with implicit constants independent of $\lambda$. 

We note that $I$ behaves like the identity operator on frequencies smaller than $N$ and integrates of order $1-s$ on frequencies much bigger than $N$. 
Indeed, 
$$\sum_{k\lesssim N} \langle k\rangle^{2s} |\widehat{u}(k)|^2 \lesssim 
   \sum_{k\lesssim N} \langle k\rangle^2 m(k)^2 |\widehat{u}(k)|^2 \lesssim 
   N^{2(1-s)} \sum_{k\lesssim N} \left(\frac{\langle k \rangle}{N}\right)^{2-2s} \langle k\rangle^{2s} m(k)^2 |\widehat{u}(k)|^2$$
and 
$$ 
\sum_{k\gg N} \langle k\rangle^{2}  \frac{1}{\langle k \rangle^{2-2s}}  |\widehat{u}(k)|^2
\lesssim 
   \sum_{k\gg N} \langle k\rangle^2 m(k)^2 |\widehat{u}(k)|^2 \lesssim  
   \sum_{k\gg N} \langle k\rangle^{2}  \left(\frac{N}{\langle k \rangle}\right)^{2-2s}  |\widehat{u}(k)|^2\ .$$
Therefore,  we have
\begin{equation}
\label{smoothingpropofI}
 \|u\|_{H^s(\Tl)} \lesssim \| I u \|_{H^1(\Tl)}\lesssim N^{1-s} \|u\|_{H^s(\Tl)} , 
\end{equation}
as well as
\begin{equation}
\label{smoothingpropofIhomog}
 \|u\|_{\dot{H}^s(\Tl)} \lesssim \| I u \|_{\dot{H}^1(\Tl)}\lesssim N^{1-s} \|u\|_{\dot{H}^s(\Tl)} .
\end{equation}

\subsection{Multilinear forms}
As in \cite{CKSTT, CKSTTrefined,deSilva2007,MiaoWuXu}, we use the shorthand notations 
$k_{12\ldots n}:=k_1+k_2+\ldots +k_n$, $k_{1-2}:=k_1-k_2$, etc., as well as  
$m_j:=m(k_j), m_{jh}:=m(k_{jh})$, etc. 
Also, we set
$$\Gamma_n(\Tl):=\{\mathbf{k}=(k_1,\ldots, k_n)\in (\Zl)^n : k_{12\ldots n}=0\}$$ 
and endow it with the Dirac measure $\delta_0(k_1+k_2+\ldots +k_n)$. 

For $n$ even integer, we define the $n$-linear form of $f_1,\ldots f_n:\Tl\to\C$ 
associated to the multiplier $M_n:\R^n\to\C$ by 
$$\Lambda_n(M_n;f_1,\ldots,f_n):=
\int_{\Gamma_n(\Tl)} M_n(k_1,k_2,...,k_n)\prod_{j=1}^n \widehat{f_j}(k_j) $$
and the shorthand $\Lambda_n(M_n;f):=\Lambda_n(M_n; f,\overline{f},\ldots,f,\overline{f})$. 
For example, we can write 
\begin{align*}
\int_{\T_{\lambda}} |v_x^2|dx &= - \Lambda_2(k_1k_2;v) ,\\ 
\Im \int_{\Tl}|v|^2v\bar{v}_xdx &= -\frac{1}{4}\Lambda_4(k_{13-24};v) .
\end{align*}

\begin{remark}
\label{rmk:SymmPropofLambdas}
We note that 
$$ \overline{\Lambda(M_n;f)} =\int_{\Gamma_n(\Tl)} \overline{M_n}(k_1,k_2,\ldots,k_n)
 \widehat{\overline{f}}(-k_1)\widehat{f}(-k_2) \cdots 
  \widehat{\overline{f}}(-k_{n-1})\widehat{f}(-k_n) $$
and thus, if the multiplier $M_n$ is such that   
$$\overline{M_n}(-k_2, -k_1,\ldots,-k_n, -k_{n-1})= \sigma M_n(k_1,k_2,\ldots,k_{n-1},k_n), $$
then we have that 
$\Lambda_n(M_n;f)$ is $\R$-valued ($i\R$-valued), 
provided that $\sigma=+1$ ($\sigma=-1$). 
\end{remark}

On the Fourier side, the equation \eqref{g1DNLS} is written as
\begin{gather*}
\begin{split}
\partial_t \widehat{v}(k) = -i k^2 \widehat{v}(k) &- i\int_{k_{123}=k} k_2 \widehat{v}(k_1)\widehat{\overline{v}}(k_2) \widehat{v}(k_3) 
-i\mu[v] \int_{k_{123}=k}\widehat{v}(k_1)\widehat{\overline{v}}(k_2) \widehat{v}(k_3) \\
& + \frac{i}{2}\int_{k_{12345}=k} 
 \widehat{v}(k_1)\widehat{\overline{v}}(k_2) \widehat{v}(k_3)\widehat{\overline{v}}(k_4) \widehat{v}(k_5) 
 \,+\,i\psi[v]\widehat{v}(k) .
\end{split}
\end{gather*}
Also, 
\begin{gather*}
\begin{split}
\partial_t \widehat{\overline{v}}(k) = +i k^2 \widehat{\overline{v}}(k) &- 
i\int_{k_{123}=k} k_2 \widehat{\overline{v}}(k_1)\widehat{v}(k_2) \widehat{\overline{v}}(k_3) 
-i\mu[v] \int_{k_{123}=k}\widehat{\overline{v}}(k_1)\widehat{v}(k_2) \widehat{\overline{v}}(k_3) \\
& - \frac{i}{2}\int_{k_{12345}=k} 
 \widehat{\overline{v}}(k_1)\widehat{v}(k_2) \widehat{\overline{v}}(k_3)\widehat{v}(k_4) \widehat{\overline{v}}(k_5) 
 \,-\,i\psi[v]\widehat{\overline{v}}(k) .
\end{split}
\end{gather*}

If $n,\ell$ are even integers and $1\leq j\leq n$, 
{the elongation at index $j$ with $\ell$ positions of the multiplier $M_n$}  
is defined by
$$\mathbb{X}_j^{\ell}(M_n)(k_1,k_2,\ldots,k_{n+\ell}):=
M_n(k_1,\ldots,k_{j-1},k_j+k_{j+1}+\ldots+k_{j+\ell},k_{j+\ell+1}, \ldots, k_{n+\ell}) .$$ 
Then, for a solution $v$ of \eqref{g1DNLS}, we have the {differentiation rule} 
\begin{gather}
\label{diffrule}
\begin{split}
\partial_t \Lambda_n(M_n;v(t)) =& i\Lambda_n\left(M_n\sum_{j=1}^n (-1)^jk_j^2 ; v(t)\right) \\
&-i\Lambda_{n+2}\left(\sum_{j=1}^n \mathbb{X}_j^2(M_n)k_{j+1};v(t)\right) -
    i\mu[v] \Lambda_{n+2}\left(\sum_{j=1}^n \mathbb{X}_j^2(M_n) ;v(t)\right)\\
&+\frac{i}{2} \Lambda_{n+4}\left(\sum_{j=1}^n (-1)^{j-1}\mathbb{X}_j^4(M_n);v(t)\right) 
  .
\end{split}
\end{gather}
In comparison with the similar rule in the Euclidean setting
(see \cite[Proposition~3.5]{CKSTTrefined}), 
we note that the additional term (i.e. the one coupled with $\mu[v]$) 
is due to the particularity of the gauge transformation \eqref{defnofG}-\eqref{defnofJ}. 
Since $\psi$ is $\R$-valued, 
the terms corresponding 
to the $\psi[v]v$ term in \eqref{g1DNLS} cancel each other. 

We introduce the following notation for the factor corresponding to 
the term $\partial_x^2v$ in the equation \eqref{g1DNLS}:
\begin{equation}
\alpha_n(\mathbf{k}):=-i(k_1^2-k_2^2+\ldots+k_{n-1}^2-k_n^2) .
\end{equation}
Note that $\alpha_2=0$ on $\Gamma_2(\Tl)$. 
A key property for the analysis of the second and third generation modified energies is the factorization of 
$\alpha_4$ on $\Gamma_4(\Tl)$: 
\begin{equation}
\alpha_4(\mathbf{k})=-i\left( (k_1-k_2)k_{12} +(k_3-k_4)k_{34}\right)=-2ik_{12}k_{14}.
\end{equation}
In this direction, 
we further introduce the {modulations}: 
\begin{align*}
&\omega_j := \tau_j + k_j^2\quad,\quad \text{for } j\text{ odd} ,\\
&\omega_j := \tau_j - k_j^2\quad,\quad \text{for } j\text{ even} ,
\end{align*}
for all $(\tau_1,\ldots,\tau_n)\in \Gamma_n(\R):=\{(\tau_1,\ldots,\tau_n)\in \R^n : \tau_1+\tau_2+\ldots+\tau_n=0 \}$.
Note that 
\begin{align*}
\omega_1+\omega_2+\omega_3+\omega_4 &= \tau_{1234} + k_1^2-k_2^2 +k_3^2-k_4^2\\ 
 	&= 2k_{12}k_{14}
\end{align*}
which implies
\begin{equation}
\max_{1\leq j\leq 4} |\omega_j|   \gtrsim   |k_{12}k_{14}|  .
\end{equation}

\section{Coercivity properties}
\label{Sect:coercivity}

We begin this section by revisiting the energy functional corresponding 
to the gauge equivalent DNLS equation \eqref{g1DNLS} on $\Tl$, 
namely 
\begin{equation}
\label{gaugedEnergy}
E_1[v] := \int_{\T_{\lambda}}\left( |\partial_x v|^2 -\frac{1}{2} |v|^2\Im(v\partial_x\overline{v})  \right)dx  
  +\frac12\mu[v] \|v\|_{L^4(\Tl)}^4 + 2\mu[v]P_1[v] - \mu[v]^2M[v].
\end{equation}
Compared to the real-line case, 
due to the particularity of the gauge transformation in the periodic setting, 
the terms coupled with the coefficient $\mu[v]$ are new. 
The last two terms are integrals of motion, 
so we could discard them. 
However, the term $\frac12\mu[v]\|v\|_{L^4}^4$ is not conserved by the flow of 
\eqref{g1DNLS}.

\begin{remark}
\label{rmk:droppinglasttermofE} 
If $v$ is a smooth solution of \eqref{g1DNLS}, 
$\|v\|_{L^4}$ is not necessarily conserved. 
Indeed, we have (see \eqref{defnofmathcalN}-\eqref{defnofmathcalN123})
\begin{align*}
\partial_t \|v\|_{L^4}^4 &= 
4\Re\int_{\Tl} |v|^2\overline{v}(i\partial_x^2v -i \mathcal{N}(v))\,dx\\
&= - 4\Im  \int_{\Tl} |v|^2\overline{v} \partial_x^2 v 
- |v|^2\overline{v}\Big( \mathcal{N}_1(v)+ \mathcal{N}_2(v) + \mu[v]\mathcal{N}_3(v) -\psi[v]v\Big)\,dx \\
&= 4\Im \int_{\Tl} \partial_x(v\overline{v}^2)\partial_x v\, dx 
   -4\Im \psi[v]\|v\|_{L^4(\Tl)}^4 + \textup{h.o.t.}\\
 &= 4\Im \int_{\Tl} \overline{v}^2(\partial_x v)^2\,dx + \textup{h.o.t.} ,
\end{align*}
where we used the fact that $\psi[v]$ is $\R$-valued; see \eqref{defnofpsinu1}.  
In general, the higher order terms (\textup{h.o.t.})  
cannot cancel the fourth order term 
$4\Im \int_{\Tl} \overline{v}^2v_x^2\,dx $. 
\end{remark}

Nevertheless, by Sobolev embedding and interpolation of $H^s$-norms, we have
$$\|v\|_{L^4}\lesssim \|v\|_{H^{\frac14}}\leq \|v\|_{L^2}^{\frac34}\|v\|_{H^1}^{\frac14}$$
and therefore, for any $\varepsilon>0$, 
\begin{equation}
\label{estimateofmuL4v}
\frac12\mu(v)\|v\|^4_{L^4} \lesssim \|v\|_{L^2}^5\|v\|_{H^1}\lesssim 
\varepsilon \|\partial_x v\|^2_{L^2} +\varepsilon \|v\|_{L^2}^2+  \frac{1}{\varepsilon} \|v\|_{L^2}^{10} .
\end{equation}
Therefore, we consider the essential part 
of the energy functional in \eqref{gaugedEnergy}, namely 
\begin{equation}
\label{defnofmathcalE}
\mathcal{E}[v]:= \int_{\Tl} \left( |\partial_x v|^2 -\frac12 |v|^2\Im(v\overline{v}_x)\right)dx .
\end{equation}
This is the same expression as the energy corresponding to 
\eqref{eqnintro:gaugedDNLSonR} on the real-line 
(see \cite{CKSTT}). 
In view of \eqref{estimateofmuL4v} and the conservation of mass, 
when controlling the $\dot{H}^1$-norm of a solution $v$ to \eqref{g1DNLS}, 
the above $\mathcal{E}[v]$ is just as good as the conservation law $E_1[v]$. 

Applying the same strategy to the mixed term  $|v|^2\Im(v\overline{v}_x)$ 
and by using the Gagliardo-Nirenberg inequality \eqref{GNinequalityT}, we get
$$\mathcal{E}[v] +1 \gtrsim_{\delta,M[v]} \left(4\pi^2(1-\varepsilon-\delta)\delta -M[v]^2\right) \|\partial_x v\|_{L^2}^2$$
for any $\varepsilon, \delta>0$, where the constant $1$ in the left hand side above 
hides a polynomial in $M[v]$. 
Since $\sup_{\varepsilon,\delta>0}(1-\varepsilon-\delta)\delta = \frac14$, 
this would yield the mass threshold condition $M[v]<\pi$. 

However, as was noticed by Hayashi and Ozawa \cite{Hayashi93,HayashiOzawa} in the Euclidean case, 
the choice  $\beta=\frac34$ for the gauge transformation 
yields a neat expression for the corresponding energy functional 
and a better mass threshold condition, namely $M[v]<2\pi$. 
In view of Remark~\ref{rmk:rigidityofQandW}, this mass threshold is sharp; 
it cannot be improved by any other gauge choice.  
Using the adaptation of a Gagliardo-Nirenberg inequality 
(Lemma~\ref{lem:LRSgagliardonirenberg} above), 
we show that this threshold also carries over to the periodic setting. 

\begin{lemma}
\label{EnergyscriptEcontrolshomogH1}
Let $\lambda\geq 1$. For any $f\in H^1(\Tl)$ with $M[f]=\|f\|_{L^2(\Tl)}^2 <2\pi$, we have:
\begin{equation}
\label{2ndIneqLem3p1}
\|\partial_x f\|_{L^2(\Tl)}^2 \lesssim \mathcal{E}[f] + 1.
\end{equation}
The implicit constant depends only on $M[f]$and blows up as $M[f]\nearrow 2\pi$. 
\end{lemma}
\begin{proof}
Consider $g=\mathcal{G}_{\beta}(f)$. 
Then $|g|=|f|$ and 
$$\partial_x f = e^{i\beta\mathcal{J}(g)}\Big(\partial_x g + i\beta(|g|^2-\mu[g])g\Big).$$
Using \eqref{GNinequalityHerr}, it follows that
\begin{gather}
\label{coercivityeqn1}
\begin{split}
\|\partial_x f\|_{L^2}^2 &= 
\|\partial_x g\|_{L^2}^2 + \beta^2 \|(|g|^2-\mu[g])g\|_{L^2}^2
-2\beta \int (|g|^2-\mu[g])\Im(g\partial_x\overline{g})\,dx\\
&\leq  \Big(1+\beta^2\|g\|_{L^2}^4 +2|\beta|\|g\|_{L^2}^2\Big) \|\partial_x g\|_{L^2}^2.
\end{split}
\end{gather}
Straightforward computations give us 
\begin{gather*}
\begin{split}
\mathcal{E}[f] &= 
 \|\partial_x g\|_{L^2}^2 +\beta^2\int \left|(|g|^2-\mu[g])g\right|^2 dx -2\beta\int(|g|^2-\mu[g])\Im(g\partial_x\overline{g})dx\\ 
  &\qquad - \frac{1}{2} \int |g|^2\Im(g\partial_x\overline{g})dx +\frac{\beta}{2} \int (|g|^2-\mu[g])|g|^4dx
\end{split}
\end{gather*}
By taking $\beta=-\frac{1}{4}$, we obtain 
\begin{gather*}
\begin{split}
\mathcal{E}[f] &=  \|\partial_x g\|_{L^2}^2 - \frac{1}{16}\|g\|_{L^6}^6 
 +\frac{1}{16} \left(\mu[g]^2-2\mu[g]\right)\|g\|_{L^2}^2 -\frac{1}{2}\mu[g] \int\Im(g\partial_x\overline{g})dx + 
 \frac{1}{8}\mu[g]\|g\|_{L^4}^4\\
 &\geq \|\partial_x g\|_{L^2}^2 - \frac{1}{16}\|g\|_{L^6}^6 -\frac18\mu[g] \|g\|_{L^2}^2 -\frac{1}{2}\mu[g] \int\Im(g\partial_x\overline{g})dx
\end{split}
\end{gather*}
and note that for any $\varepsilon>0$ and any $\lambda\geq 1$, 
$$
\left| \frac{1}{2}\mu[g] \int\Im(g\partial_x\overline{g})dx\right| 
\leq \|g\|^3_{L^2}\|\partial_x g\|_{L^2} \leq 
 \varepsilon \|\partial_x g\|_{L^2}^2 + C_{\varepsilon} \|g\|_{L^2}^6 ,$$
for some $C_{\varepsilon}\sim \varepsilon^{-1}$. 
We choose $\varepsilon>0$ such that 
$\|f\|_{L^2}^4(\frac{1}{4\pi^2}+\frac{\varepsilon}{16})<1-\varepsilon$, 
and by \eqref{GNinequalityT} 
\footnote{By using \eqref{GNinequalityHerr} at this point, the coercivity of $\mathcal{E}[\,\cdot\,]+1$  
would be obtained under $M[u_0]<2\sqrt{2}$.}
we then get 
\begin{gather*}
\begin{split}
\mathcal{E}[f] &\geq (1-\varepsilon)\|\partial_xg\|_{L^2}^2 
 - \frac{1}{16}(\frac{4}{\pi^2}+\varepsilon)\|\partial_x g\|_{L^2}^2\|g\|_{L^2}^4 -\frac{1}{16}K_{\varepsilon}\|g\|_{L^2}^6\\
 &\qquad -\frac{1}{8} \mu[g]\|g\|_{L^2}^2
   -C_{\varepsilon}\|g\|_{L^2}^6 
\end{split}
\end{gather*}
and thus 
\begin{gather}
\label{coercivityeqn2}
\begin{split}
\mathcal{E}[f] +M[f]^3 \gtrsim \left((1-\varepsilon)-(\frac{1}{4\pi^2}+\frac{\varepsilon}{16})\|f\|_{L^2}^4\right) \|\partial_xg\|_{L^2}^2 .
\end{split}
\end{gather}
By combining \eqref{coercivityeqn1} and \eqref{coercivityeqn2}, we deduce \eqref{2ndIneqLem3p1}. 
\end{proof}

Inspired by a recent paper of Guo and Wu \cite{GuoWu}, 
we can improve the mass threshold below which we can control the $\dot{H}^1$-norm of $f$ 
by using both the energy $\mathcal{E}[f]$ and the momentum 
\begin{equation}
\label{defnofmathcalP}
\mathcal{P}[f] :=  \int_{\Tl} \Im(f\partial_x{\overline{f}})dx - \frac12 \|f\|_{L^4(\Tl)}^4 
\end{equation} 
associated to \eqref{g1DNLS}, where we dropped the conserved term from \eqref{defn:Psubnu}. 
The key observation is to notice that by modulating $f$, 
the change in kinetic energy incurred resembles the main part of the momentum $\mathcal{P}[f]$ 
(see \eqref{eqn:changeinkineticenergy} below). 

\begin{lemma}
\label{lem:controlofdotH1wEandP}
Let $\lambda\geq 1$. 
For any $f\in H^1(\Tl)$ with $M[f]=\|f\|_{L^2(\Tl)}^2 <4\pi$, we have: 
\begin{equation}
\label{controlwithEandPsquared}
\|\partial_x f\|_{L^2(\Tl)}^2 \lesssim \left|\mathcal{E}[f]\right|+ \mathcal{P}[f]^2 +1. 
\end{equation}
The implicit constant  depends only on $M[f]$ and blows up as $M[f]\nearrow 4\pi$. 
\end{lemma}
\begin{proof}
As in the proof of Lemma~\ref{EnergyscriptEcontrolshomogH1} above, 
let us consider $g=\mathcal{G}_{-\frac14}(f)$ for which, according to \eqref{coercivityeqn1},  we have 
$$\|\partial_x f\|_{L^2(\Tl)}\sim \|\partial_x g\|_{L^2(\Tl)}.$$
The main part is showing that 
\begin{equation}
\label{lem4p3:mainestimate}
\|\partial_x g\|^2_{L^2(\Tl)} \lesssim |E_{\frac34}[g]| + P_{\frac34}[g]^2 +1 .
\end{equation}
Indeed, this suffices to get \eqref{controlwithEandPsquared} as we have
\begin{align*}
| P_{\frac34}[g]| =& |P_{1}[f] |\leq |\mathcal{P}[f]| + \mu[f]M[f] , \\
| E_{\frac34}[g]| =& |E_{1}[f] | \lesssim  |\mathcal{E}[f]| +\frac12 \mu[f]\|f\|_{L^4}^4 + 
 \mu[f]  |\mathcal{P}[f]| + \mu[f]^2 M[f] ,
\end{align*}
and we can use \eqref{estimateofmuL4v}. 

From \eqref{defn:Psubnu}-\eqref{defn:Esubnu}, we recall that
\begin{align*}
P_{\frac34}[g] &=\int_{\Tl}\Im(g\partial_x\overline{g})dx -\frac14\|g\|_{L^4}^4 + \frac34\mu[g]M[g],\\
E_{\frac34}[g] &=\|\partial_x g\|_{L^2}^2 - \frac{1}{16}\|g\|_{L^6}^6 +\frac38\mu[g]\|g\|_{L^4}^4 
    +\frac32\mu[g]P_{\frac34}[g] - \frac{9}{16}\mu[g]^2M[g] .
\end{align*}
In order to get \eqref{lem4p3:mainestimate}, 
we consider the modulated function 
$g_{\alpha}(x):=e^{i\alpha x}g(x)$ 
with $\alpha\in \Zl$ and $\alpha>0$ to be chosen later. 
We have 
\begin{equation}
\label{eqn:changeinkineticenergy}
\|\partial_x g_{\alpha}\|_{L^2}^2 = \|\partial_x g\|_{L^2}^2 
  + \alpha^2\|g\|_{L^2}^2 - 2\alpha\int_{\Tl}\Im(g\partial_x\overline{g})dx 
\end{equation}
and therefore
\begin{equation}
\label{eqn:diffofE34}
 E_{\frac34}[g_{\alpha}]- E_{\frac34}[g] =\alpha^2M[g] -2\alpha\int_{\Tl}\Im(g\partial_x\overline{g})dx +
 \frac38\mu[g] \Big( P_{\frac34}[g_{\alpha}]- P_{\frac34}[g]\Big). 
\end{equation}
Since
$$\Im(g_{\alpha}\partial_x\overline{g_{\alpha}}) = \Im(g\partial_x\overline{g}) - \alpha |g|^2 ,$$
we also have
\begin{equation}
\label{eqn:diffP34}
P_{\frac34}[g_{\alpha}]- P_{\frac34}[g] = -\alpha M[g]. 
\end{equation}
Therefore
$$
\frac{1}{2\alpha}E_{\frac34}[g_{\alpha}]- 
\frac{1}{2\alpha}E_{\frac34}[g] =\frac{\alpha}{2}M[g] -\int_{\Tl}\Im(g\partial_x\overline{g})dx 
 -\frac34\mu[g]M[g]$$
and thus we find that 
\begin{equation}
\label{eqn:P34g}
\frac{1}{2\alpha}E_{\frac34}[g_{\alpha}] - \frac{1}{2\alpha}E_{\frac34}[g] -\frac{\alpha}{2}M[g] + \frac14\|g\|_{L^4}^4 
 =-P_{\frac34}[g] . 
\end{equation}

We now  use the Gagliardo-Nirenberg inequality \eqref{GNinequalityTCGN} to give a lower bound to the first term in \eqref{eqn:P34g}; we drop the positive term $\frac38\mu[g_{\alpha}]\|g_{\alpha}\|_{L^4}^4$. 
Also, we use \eqref{eqn:diffP34}, 
and taking into account that the Lebesgue norms of $g_{\alpha}$ and $g$ coincide, 
we have
\begin{align*}
E_{\frac34}[g_{\alpha}] \geq& \|g\|_{L^6}^6\left(
C_{GN}^{-18}\left(1+\frac{\delta}{5\pi\lambda}\right)^{-4}\frac{\|g\|_{L^6}^{12}}{\|g\|_{L^4}^{16}} -\frac{1}{16}\right) 
  - \frac{1}{\pi\lambda\delta}M[g]\\
&\quad +\frac32\mu[g]\left(P_{\frac34}[g] - \alpha M[g]\right) -\frac{9}{16}\mu[g]^2M[g]
\end{align*}
By \eqref{eqn:P34g}, we then get
\begin{align*}
|P_{\frac34}[g] | +\frac34\mu[g]M[g] 
\geq& \frac14\|g\|_{L^4}^4 -\frac{\alpha}{2}M[g] 
 -\frac{1}{2\alpha}\varphi(\frac{\|g\|_{L^6}^6}{\|g\|_{L^4}^{8}}) \|g\|_{L^4}^8 \\
 &-\frac{1}{2\alpha} \left( |E_{\frac34}[g]| +\frac32\mu[g]|P_{\frac34}[g]| + \frac{9}{16}\mu[g]^2M[g]
 +\frac{1}{\pi\lambda\delta} M[g]  \right) ,
\end{align*}
where 
$$\varphi(x):= \left( \frac{1}{16}- C_{\textup{GN}}^{-18}\left(1+\frac{\delta}{5\pi\lambda}\right)^{-4} x^2\right)x $$
and for which we have
$$\max_{x>0} \varphi(x) \geq \max_{x>0} \left(\frac{1}{16}-C_{\textup{GN}}^{-18} x^2\right)x=\frac{1}{64\pi}.$$
We now balance the terms 
$\frac{\alpha}{2} M[g]$ and $\frac{1}{128\pi\alpha}\|g\|_{L^4}^8$ by choosing 
$$ \alpha_*:= \frac{\|g\|_{L^4}^4}{8\sqrt{\pi}\|g\|_{L^2}}.$$
However, in order to correctly define $g_{\alpha}$ as a periodic function on $\Tl$, we take 
\begin{equation}
\alpha:=  \frac{1}{\lambda} \Big( \left[\lambda\alpha_*\right] +1\Big) 
\end{equation}
(here, by $[x]$ we denote the integer part of $x$). 
Then 
\begin{align*}
-\frac{\alpha}{2} M[g] -\frac{1}{128\pi\alpha}\|g\|_{L^4}^8 &\geq 
-\frac{\alpha_*}{2} M[g] -\frac{1}{128\pi\alpha_*}\|g\|_{L^4}^8 -\frac{1}{2\lambda} M[g]\\
 &= -\alpha_* M[g]-\frac{1}{2\lambda} M[g]
\end{align*}
and taking into account that $\lambda\geq 1$, we  deduce
\begin{gather}
\label{lem4p3:ineqforP34}
\begin{split}
|P_{\frac34}[g] | +M[g]^2 +M[g]  
\geq& \frac14\|g\|_{L^4}^4 -\alpha_* M[g]\\
  &-\frac{1}{2\alpha_*}\left(  |E_{\frac34}[g]| + M[g]|P_{\frac34}[g]| +M[g]^3 + \delta^{-1} M[g]\right) .
\end{split}
\end{gather}
We consider the following positive reals
\begin{align*}
a &:= \frac14\left( 1- \frac{1}{2\sqrt{\pi}} \|g\|_{L^2} \right) ,\\ 
b &:= 4\sqrt{\pi}\|g\|_{L^2}\left(  |E_{\frac34}[g]| + M[g]|P_{\frac34}[g]| +M[g]^3 + \delta^{-1} M[g]\right) ,\\
c &:= |P_{\frac34}[g] | +M[g]^2 +M[g] .
\end{align*}
Thus, the inequality \eqref{lem4p3:ineqforP34} provides the following
$$ c\geq a \|g\|_{L^4}^4 - \frac{b}{\|g\|_{L^4}^4} .$$
It follows that 
$$\|g\|_{L^4}^4\leq \frac{c+\sqrt{c^2+4ab}}{2a} \lesssim c+ b^{\frac12} $$
and so we obtain 
\begin{equation}
\label{lem4p3:boundonL4g}
\|g\|_{L^4}^8\lesssim c^2+b\lesssim   |E_{\frac34}[g]| + P_{\frac34}[g]^2 +1.
\end{equation}
Therefore, by using again  \eqref{GNinequalityTCGN}, 
\begin{align*}
\|\partial_x g\|_{L^2}^2 +\frac{2}{\delta}\mu[g] &= E_{\frac34}[g] +\frac{1}{16}\|g\|_{L^6}^6 -\frac38\mu[g]\|g\|_{L^4}^4 
 -\frac32\mu[g] P_{\frac34}[g] +\frac{9}{16} \mu[g]^2M[g] +\frac{2}{\delta}\mu[g] \\
 &\lesssim |E_{\frac34}[g] | + |P_{\frac34}[g] | +1 + 
   \left( \|\partial_x g\|_{L^2}^2 + \frac{2}{\delta} \mu[g]\right)^{\frac13} \|g\|_{L^4}^{\frac{16}{3}}
\end{align*}
where the implicit constant can be taken to depend only on $M[g]$. 
Then either 
$$\|\partial_x g\|_{L^2}^2 +\frac{2}{\delta}\mu[g] \lesssim |E_{\frac34}[g] | + |P_{\frac34}[g] | +1$$
or 
$$\left( \|\partial_x g\|_{L^2}^2 +\frac{2}{\delta}\mu[g]\right)^{\frac23} \lesssim \|g\|_{L^4}^{\frac{16}{3}} $$
and we use \eqref{lem4p3:boundonL4g}. 
In both cases, 
\eqref{lem4p3:mainestimate} holds and the proof is completed.
\end{proof}

\section{Local well-posedness for the $I$-system}
\label{sect:proofofLWPofIsystem}

Since $Iv$ does not solve the gauge equivalent  equation \eqref{g1DNLS}, in general 
$P_1[Iv]$ and $E_1[Iv]$ are not conservation laws, 
even for a smooth (local) solutions $v$ of \eqref{g1DNLS}. 
Instead, we have
\begin{equation}
\label{Isystem}
i\partial_t (Iv) +\partial_x^2 (Iv) = 
-i I(v^2\partial_x\overline{v}) - \frac12 I(|v|^4v) 
+ \mu[v]I(|v|^2v)  -\psi[v](Iv) ,
\end{equation}
with $(Iv)_{|t=0}=Iv_0$ and $x\in\Tl$. 
We modify the local well-posedness proof for  
\eqref{g1DNLS} to obtain the following 
result for \eqref{Isystem}.

\begin{proposition}
\label{prop:LWPIsyst}
Let $B>0$.  There exist $\delta \sim B^{-\theta}$ (for some $\theta>0$) and $D>0$ 
(both independent of $N$ and $\lambda$) such that 
if $v_0\in H^s(\Tl)$ is such that $\|Iv_0\|_{H^1(\Tl)}\leq B$, then 
\begin{equation}
\|Iv\|_{Z^{1}([0,\delta]\times \Tl)}\leq D.
\end{equation} 
\end{proposition}

In order to prove this result, we use the estimates  of the local well-posedness theory for \eqref{g1DNLS} due to Herr {\cite{HerrIMRN06}} and an interpolation lemma 
of Colliander, Keel, Staffilani, Takaoka, and Tao {\cite[Lemma~12.1]{CKSTTjfa04}} for translation invariant multi-linear operators. 

\begin{lemma} {\cite[Section~4]{HerrIMRN06}}
\label{HerrsLWPestimates}
Let $\delta\in (0,1)$ and $\lambda\geq 1$. 
There exist $c,\varepsilon>0$ such that 
\begin{align}
&\|u_1(\partial_x\overline{u_2})u_3\|_{\widetilde{Z}^{\frac12}(\R\times\Tl)}
  \leq c \delta^{\varepsilon} \prod_{j=1}^3 
  	\|u_j\|_{X^{\frac12,\frac12}(\R\times\Tl)} ,\\
&\|u_1\overline{u_2}u_3\overline{u_4}u_5)\|_{\widetilde{Z}^{\frac12}(\R\times\Tl)}
   \leq c \delta^{\varepsilon} \prod_{j=1}^5 
    \|u_j\|_{X^{\frac12,\frac12}(\R\times\Tl)} ,\\
&\|u_1\overline{u_2}u_3\|_{\widetilde{Z}^{\frac12}(\R\times\Tl)}
   \leq c \delta^{\varepsilon} \prod_{j=1}^3 \|u_j\|_{X^{\frac12,\frac12}(\R\times\Tl)} ,
\end{align}
for all $u_j\in \mathcal{S}_{\lambda}$ with $\supp(u_j)\in \{(t,x)\in\R\times \Tl : |t|\leq \delta\}$, $1\leq j\leq 5$. 
\end{lemma}

\begin{remark}
One can check that the pointwise weights bounds provided by 
\cite[Lemma~4.1, Lemma~4.3]{HerrIMRN06} hold uniformly in $\lambda\geq1$ 
(although, in view of Remark~\ref{rmk:japanesebracketwithparam}, 
further sub-cases have to be addressed). 
Then, the multi-linear estimates above use 
only the $L^4$-Strichartz and Sobolev inequalities of Lemma~\ref{Lem:SobolevStrichartz} above, 
which are all scaling invariant.
\end{remark}

In order to state the interpolation lemma, 
let $I_N^s$ denote the $I$-operator introduced in \eqref{defnofIoperator}. 
Also, following \cite{CKSTTjfa04}, 
we let $S_x$ to denote the shift operator $S_x u(y,t) = u(y-x,t)$. 
A Banach space $X$ of time-space functions 
$u:J\times \Tl \to \C$ (where $J\subset \R$ is some time interval) 
is translation invariant if $\|S_x u\|_{X} = \|u\|_X$ for all $u\in X$ and all $x$.  
We  use the spaces 
$X=X^{1,\frac12}(J\times \Tl)$ and 
$Z=\widetilde{Z}^{1}(J\times \Tl)$ 
which clearly 
satisfy this requirement. 
An $n$-linear operator 
$T:X\times\ldots\times X\to Z$ is translation invariant if 
$S_x T(u_1,\ldots, u_n) = T(S_x u_1, \ldots, S_x u_n)$ for all $u_j\in X$.

\begin{lemma}
\label{CKSTTlemmaforIoperator} 
Let $s_0>0$, $n\geq 1$ and let 
$T:X\times \ldots \times X\to Z$ be a translation invariant 
$n$-linear operator.  Suppose
\begin{equation}
\label{assumedoundonI1s}
\|I_1^{s} T(u_1,\ldots u_n)\|_Z \leq C \prod_{j=1}^n  \|I_1^s u_j\|_{X} 
\end{equation}
for all $s_0\leq s\leq 1$ and all $u_j\in X$, for some $C>0$. 
Then, we also have
\begin{equation}
\label{inherittedboundonIns}
\|I_N^{s} T(u_1,\ldots u_n)\|_Z \leq DC \prod_{j=1}^n  \|I_N^s u_j\|_{X} 
\end{equation}
for all $s_0\leq s\leq 1$ and all $u_j\in X$, for some $D>0$ independent of $N$ and $\lambda$. 
\end{lemma}
To convince the reader that the proof of {\cite[Lemma~12.1]{CKSTTjfa04}} yields 
the constant $D$ independent of the parameter $\lambda$ (as well as $N$), 
we provide the following remark that uses the ``periodization'' procedure 
also encountered in the Poisson summation formula.

\begin{remark}
We know that  
the Littlewood-Paley projection operators $P_{\lesssim N}f:=\phi_N*f$ are uniformly bounded in $N$, where 
$\phi_N:=N\phi(N\cdot)$ and $\widehat{\phi}$ 
is a symmetric function on $\Zl$ equal to one on $\{|k|\leq 1\}$ and 
vanishes outside $\{|k| < 2\}$. 
However, the bound $\|\phi\|_{L^1(\Tl)}$ depends on $\lambda$. 
We modify slightly this usual definition in order 
to ensure uniform boundedness in the scaling parameter $\lambda$ as well. 
Thus, let $\psi$ be a Schwarz function on $\R$ such that $\widehat{\psi}$ 
is a symmetric bump function compactly supported in $\{\xi\in\R : |\xi|\leq 2\}$ 
and identically one for $|\xi|\leq 1$. 
Define $\psi_N:=N\psi(N\cdot)$ and for any $x\in\Tl$ we set
$$\varphi_N(x):= \sum_{k\in\Z} \psi_N(x+2\pi\lambda k) .$$
Note that 
$\widehat{\varphi_N}(k) = \widehat{\psi_N}(k)$
for any $k\in\Zl$, 
and thus the operator $P_{\lesssim N}f=\varphi_N*f$ acts as the identity operator 
when $\supp(\widehat{f})\subset \{k \in\Zl : |k|\lesssim N\}$ 
(this is compatible with the region where the operators $I_N^s$ also behave like the identity operator).  
Also, 
$$\|\varphi_N\|_{L^1(\Tl)} = \|\psi_N\|_{L^1(\R)} = \|\psi\|_{L^1(\R)}$$ 
and therefore 
$$\|P_{\lesssim N}\|_{X\to X} , \|P_{\lesssim N}\|_{Z\to Z}  \lesssim 1 ,$$
uniformly in $N$ and $\lambda$. Finally, 
by arguing as in \cite{CKSTTjfa04} that 
$I_1^sI_N^{2-s}$ and $N^{s-1}I_N^s I_1^{2-s}$ are bounded 
(uniformly in $N$ and $\lambda$), splitting 
$u_j=P_{\lesssim N}u_j + P_{\gg N} u_j$ for each $j$, 
and estimating each contribution separately,  
we obtain \eqref{inherittedboundonIns}.
\end{remark}

We apply the above interpolation lemma to the trilinear and quintilinear terms corresponding 
to the right hand side of \eqref{g1DNLS}, namely
\begin{eqnarray}
\label{defnofmathcalN}
\mathcal{N}(v) &= -i v^2\partial_x \overline{v} -\frac12 |v|^4v +\mu[v]|v|^2v -\psi[v]v\\
\label{defnofmathcalN123}
&=: \mathcal{N}_1(v)+ \mathcal{N}_2(v) +\mu[v]\mathcal{N}_3(v) - \psi[v] v. 
\end{eqnarray}
Note that the estimates of Lemma~\ref{HerrsLWPestimates} 
give \eqref{assumedoundonI1s} for $s_0=\frac12$. 
Since 
$I_N^1=\mathrm{Id}$ for any $N$, 
we obtain the estimate \eqref{assumedoundonI1s} for $s=1$  via the Leibniz rule 
and Lemma~\ref{HerrsLWPestimates}. For example,  
\begin{align*} 
\|u_1(\partial_x\overline{u_2})u_3\|_{\widetilde{Z}^1} 
&\lesssim 
\|\langle \partial_x\rangle^{\frac12}u_1(\partial_x\overline{u_2})u_3\|_{\widetilde{Z}^{\frac12}}
 +\|u_1(\partial_x\langle \partial_x\rangle^{\frac12}\overline{u_2})u_3\|_{\widetilde{Z}^{\frac12}}
 + \|u_1(\partial_x\overline{u_2})(\langle \partial_x\rangle^{\frac12}u_3)\|_{\widetilde{Z}^{\frac12}}\\
& \lesssim \delta^{\varepsilon} \prod_{j=1}^{3} \|\langle \partial_x\rangle^{\frac12}u_j\|_{X^{\frac12,\frac12}} 
 \sim \delta^{\varepsilon} \prod_{j=1}^{3} \|u_j\|_{X^{1,\frac12}}.
\end{align*}
One argues analogously for the other multi-linear estimates of Lemma~\ref{HerrsLWPestimates}. 
Hence, by applying Lemma~\ref{CKSTTlemmaforIoperator}, we obtain
\begin{eqnarray}
\label{ITZ1}
\|I\Big(u_1(\partial_x\overline{u_2})u_3\Big)\|_{\widetilde{Z}^1} 
  &\lesssim \delta^{\varepsilon} \prod_{j=1}^3 \|Iu_j\|_{X^{1,\frac12}} ,\\
\label{IQZ1}
\|I\Big(u_1\overline{u_2}u_3\overline{u_4}u_5\Big)\|_{\widetilde{Z}^1} 
  &\lesssim \delta^{\varepsilon} \prod_{j=1}^5 \|Iu_j\|_{X^{1,\frac12}} ,\\
\label{IQ3Z1}
\|I\Big(u_1\overline{u_2}u_3\Big)\|_{\widetilde{Z}^1} 
  &\lesssim \delta^{\varepsilon} \prod_{j=1}^3 \|Iu_j\|_{X^{1,\frac12}} .
\end{eqnarray}
We also need the following Lipschitz continuity properties for the coupling coefficients $\mu[v]$ and $\psi[v]$. 
We easily have  
\begin{equation}
\left| \mu[f]-\mu[g]\right|\leq \frac{1}{2\pi\lambda} \left(\|f\|_{L^2(\Tl)}+\|g\|_{L^2(\Tl)}\right)\|f-g\|_{L^2(\Tl)} ,
\end{equation}
while H\"{o}lder's inequality, Parseval's identity, and the $L^6$-Sobolev inequality give 
\begin{equation}
\label{boundedcouplingcoefficients}
\left|\psi[f]-\psi[g]\right|\lesssim \frac{1}{2\pi\lambda}
\left(\|f\|^3_{H^{\frac12}} + \|f\|_{L^2} + \|g\|^3_{H^{\frac12}} + \|g\|_{L^2}  \right)
 \|f-g\|_{H^{\frac12}(\Tl)}. 
\end{equation}
The reader can also consult \cite[Lemma~2.5]{HerrIMRN06}.

We can now proceed with the proof of Proposition~\ref{prop:LWPIsyst} 
by using the fixed point argument in a closed ball of the space 
$W=\{v : \eta_{\delta}(t)Iv(t,x)\in Z^1(\R\times\Tl)\}$ with norm 
$$\|v\|_{W}:= \|\eta_{\delta} Iv\|_{Z^1(\R\times\Tl)} ,$$
with  $\delta\in (0,1)$ and $D>0$ to be chosen later, and 
$\eta_{\delta}(t):=\eta(\frac{t}{\delta})$. 
Grouping terms as in \eqref{defnofmathcalT}-\eqref{defnofmathcalQ}, and using the Duhamel formulation, 
solutions of \eqref{Isystem} are those $v$ that satisfy
\begin{equation}
\label{DuhamelIsystem}
Iv(t) = U_{\lambda}(t)Iv_0 
  -i \int_0^t U_{\lambda}(t-t')I\mathcal{N}(v(t')) dt'  .
\end{equation}
Consider the mapping $v\mapsto \Gamma(v)$ given by
$$\Gamma(v) :=  
\eta(t) U_{\lambda}(t)v_0 -i \eta(t)\int_0^t U_{\lambda}(t-t')\mathcal{N}(\eta_{\delta}(t')v(t')) dt' . $$
By  \eqref{linearhomogest}-\eqref{linearinhomogest} 
and \eqref{ITZ1}-\eqref{IQ3Z1}, 
we have
\begin{gather}
\begin{split}
\label{boundonGamma}
\|\Gamma(v)\|_{W} &\leq \|\eta(t)U_{\lambda}(t)Iv_0\|_{Z^1} + 
	\left\|\eta(t)\int_0^t U_{\lambda}(t-\tau) I\mathcal{N}(\eta_{\delta}(t')v(t')) dt'\right\|_{Z^1}\\
	&\leq c_1 \left(\|Iv_0\|_{H^1(\Tl)} + \| I\mathcal{N}(\eta_{\delta} v)\|_{\widetilde{Z}^1} \right)\\
	&\leq c_1B + c_2\delta^{\varepsilon}\left(
	  \|\eta_{\delta}Iv\|_{X^{1,\frac12}}^3
	  +\|\eta_{\delta}Iv\|_{X^{1,\frac12}}^5
	   + \|\eta_{\delta}Iv\|_{\widetilde{Z}^1}\right) .
\end{split}
\end{gather}
Also, 
\begin{align*}
\|\Gamma(v_1)- \Gamma(v_2)\|_{W} &= 
 \left\|\eta(t)\int_0^t U_{\lambda}(t-\tau) 
  \left(I\mathcal{N}(\eta_{\delta}(t')v_1(t')) - I\mathcal{N}(\eta_{\delta}(t')v_2(t'))\right) dt'\right\|_{Z^1}\\
  &\lesssim \|I\left(\mathcal{N}_1(\eta_{\delta}v_1)-\mathcal{N}_1(\eta_{\delta}v_2)\right)\|_{\widetilde{Z}^1} +  
   \|I\left(\mathcal{N}_2(\eta_{\delta}v_1)-\mathcal{N}_2(\eta_{\delta}v_2)\right)\|_{\widetilde{Z}^1}\\
   &\qquad + \|I\left(\mathcal{N}_3(\eta_{\delta}v_1)-\mathcal{N}_3(\eta_{\delta}v_2)\right)\|_{\widetilde{Z}^1} 
      + \|I\left(\eta_{\delta}v_1- \eta_{\delta}v_2 \right)\|_{\widetilde{Z}^1} .
\end{align*}
We write 
\begin{align*}
\mathcal{N}_1(u_1)-\mathcal{N}_1(u_2)=
u_1(\partial_x\overline{u_1})(u_1-u_2) + u_1\partial_x(\overline{u_1-u_2})u_2 + (u_1-u_2)(\partial_x \overline{u_2})u_2 
\end{align*}
and by using \eqref{ITZ1}, 
we obtain
\begin{align*}
\|I\left(\mathcal{N}_1(\eta_{\delta}v_1)-\mathcal{N}_1(\eta_{\delta}v_2)\right)\|_{\widetilde{Z}^1} 
  \lesssim \delta^{\varepsilon}(\|\eta_{\delta}Iv_1\|_{Z^1}^2+ \|\eta_{\delta}Iv_2\|_{Z^1}^2)\|\eta_{\delta}I(v_1-v_2)\|_{Z^1}.
\end{align*}
Arguing similarly by using \eqref{IQZ1} and \eqref{IQ3Z1}, we also have 
\begin{align*}
\|I\left(\mathcal{N}_2(\eta_{\delta}v_1)-\mathcal{N}_2(\eta_{\delta}v_2)\right)\|_{\widetilde{Z}^1} 
  \lesssim\delta^{\varepsilon}(\|\eta_{\delta}Iv_1\|_{Z^1}^4+ \|\eta_{\delta}Iv_2\|_{Z^1}^4)\|\eta_{\delta}I(v_1-v_2)\|_{Z^1},\\
\|I\left(\mathcal{N}_3(\eta_{\delta}v_1)-\mathcal{N}_3(\eta_{\delta}v_2)\right)\|_{\widetilde{Z}^1} 
  \lesssim \delta^{\varepsilon}(\|\eta_{\delta}Iv_1\|_{Z^1}^2+ \|\eta_{\delta}Iv_2\|_{Z^1}^2)\|\eta_{\delta}I(v_1-v_2)\|_{Z^1}.
\end{align*}
It follows that 
\begin{align}
\label{differenceofGammas}
\|\Gamma(v_1)- \Gamma(v_2)\|_{W} &\lesssim \delta^{\varepsilon} 
\left(\|v_1\|^2_W + \|v_2\|^2_W + \|v_1\|^4_W +\|v_2\|_W^4\right)\|v_1-v_2\|_W .
\end{align}
By taking $D=2c_1B+1$ and $\delta$ such that 
$\delta^{\varepsilon} D^5 \sim 1$, 
from \eqref{boundonGamma} and \eqref{differenceofGammas}, 
we get
\begin{align*}
\|\Gamma(v)\|_{W}\leq D \quad\text{and}\quad 
\|\Gamma(v_1)-\Gamma(v_2)\|_{W}\leq \frac12 \|v_1-v_2\|_W ,
\end{align*}
for all $v,v_1,v_2\in\{w\in W :\|w\|_W\leq D\}$. 
Hence, by Banach's fixed point  theorem, there exists a unique 
$v$ with $\|v\|_{W}\leq D$ such that $v=\Gamma(v)$ in $W$. 
Thus, 
$$\|Iv\|_{Z^1([0,\delta]\times \Tl)} \leq \|\eta_{\delta}Iv\|_{Z^1(\R\times\Tl)}\leq D $$   
and it follows that 
\eqref{DuhamelIsystem} holds for all $t\in [0,\delta]$. 
The proof of Proposition~\ref{prop:LWPIsyst} is completed.

\section{Modified energy functionals via the $I$-operator and correction terms}
\label{sect:modifiedenergy}

In view of the discussion in Section~\ref{Sect:coercivity}, 
we consider the essential part 
of the energy functional associated to \eqref{g1DNLS}, namely 
\begin{equation}
\label{defnofmathcalE}
\mathcal{E}[v]:= \int_{\Tl} \left( |\partial_x v|^2 -\frac12 |v|^2\Im(v\overline{v}_x)\right)dx .
\end{equation}
The \emph{first modified energy} is defined to be the  $\R$-valued functional
\begin{equation}
\label{defnofE1}
\mathcal{E}^1[v]:= \mathcal{E}[Iv]= -\Lambda_2(k_1k_2m_1m_2;v) + \frac14\Lambda_4(k_{13}m_1m_2m_3m_4;v)
\end{equation}
and for $v$ sufficiently smooth solution of \eqref{g1DNLS}, 
one can compute its time increment from the fundamental theorem of calculus
\begin{equation}
\label{incrementeq}
\mathcal{E}^1[v(t_0+\delta)] -\mathcal{E}^1[v(t_0)] = \int_{t_0}^{t_0+\delta} \frac{d}{dt} \mathcal{E}^1[v(t)]\,dt .
\end{equation}
Using \eqref{diffrule}, we have 
\begin{gather}
\label{parttE1}
\begin{split}
\frac{d}{dt} \mathcal{E}^1[v(t)] =& \Lambda_4(M_4^1;v) + \Lambda_6(M_6^1;v) +
\Lambda_8(M_8^1;v)
 - i\mu[v]\Big( \Lambda_4(K_4^1;v) + \Lambda_6(K_6^1;v)\Big) , 
\end{split}
\end{gather}
with the multipliers $M_4^1, M_6^1, M_8^1$ given by \cite[Proposition~4.1]{CKSTT}, e.g. 
\begin{align}
M_4^1(\mathbf{k}):= -\frac{i}{2} m_1m_2m_3m_4k_{12}k_{13}k_{14} -\frac{i}{2}(m_1^2k_1^2k_3 +m_2^2k_2^2k_4+m_3^2k_3^2k_1+ m_4^2k_4^2k_2) .
\end{align}
Here, it is not particularly important to have the precise expression of the multipliers $M_6^1$, $M_8^1$. 
The multipliers $K_4^1$, $K_6^1$ are new to the periodic setting 
(due to a different expression of the gauge transformation) 
and are given by
\begin{align}
\label{defnofK41}
K_4^1(\mathbf{k}) &:=\frac12 \sum_{j=1}^4 (-1)^j m_j^2k_j^2  \ ,\\
\label{defnofK61}
K_6^1(\mathbf{k}) &:=\frac23 \sum_{\substack{\{a,c,e\}=\{1,3,5\}\\ \{b,d,f\}=\{2,4,6\}}} 
  m_am_b m_cm_{def} -m_dm_em_fm_{abc}\ .
\end{align}
Note that 
by Remark~\ref{rmk:SymmPropofLambdas}, 
$\Lambda_4(K_4^1;v)$ and $\Lambda_6(K^1_6,v)$ 
are purely imaginary, and that 
$\Lambda_4(M_4^1;v)$, $\Lambda_6(M_6^1;v)$ and 
$\Lambda_8(M_8^1;v)$
are $\R$-valued.

The rule of thumb when one tries to prove estimates on the various terms of \eqref{incrementeq} is that 
``different pieces of $\Lambda_n$ appearing in the right hand side of $\partial_t \mathcal{E}^1(v)$ are 
easier for $n$ larger'' \cite[p.~72]{CKSTTrefined}. This motivates the following procedure when one tries to refine the $I$-method.

A \emph{second instantiation of the $I$-method} modifies further the expression of the energy functional by taking 
\begin{equation}
\label{defnofE2}
\mathcal{E}^2[v]:= \mathcal{E}^1[v] +\Lambda_4(\sigma_4;v)
\end{equation}
where the ``correction'' multiplier $\sigma_4$ is chosen so that 
in the expression of $\frac{d}{dt} \mathcal{E}^2(v)$, no fourth order term 
$\Lambda_4(\,\cdot\,;v)$ appears. 
For the sake of keeping the equations compact, 
we choose to drop the reference to $v$ 
from $\Lambda_n(M_n;v)$, and the frequency arguments $\mathbf{k}=(k_1,\ldots,k_n)$ when the formulae get too long.

By the differentiation rule \eqref{diffrule}, we have
\begin{gather}
\label{parttsigma4}
\begin{split}
\frac{d}{dt} \Lambda_4(\sigma_4) =& \Lambda_4(\sigma_4 \alpha_4) - 
 i\Lambda_6\left(\sum_{j=1}^4 \mathbb{X}_j^2(\sigma_4)k_{j+1}\right) 
+\frac{i}{2}\Lambda_8\left(\sum_{j=1}^4 (-1)^{j-1} \mathbb{X}_j^4(\sigma_4)\right)\\
&\quad -i\mu[v] \Lambda_6\left(\sum_{j=1}^4 \mathbb{X}_j^2(\sigma_4) \right) .
\end{split}
\end{gather}
Note that if $\alpha_4(\mathbf{k})=0$, then either $k_{12}=0$ or $k_{14}=0$, and both imply that 
$\widetilde{M_4}(\mathbf{k})=0$. 
We define the first correction $\sigma_4=\sigma_4(\mathbf{k})$ for $\mathbf{k}\in\Gamma_4(\Tl)$  
by setting 
\begin{equation}
\label{defnofsigma4}
\sigma_4 := - \frac{M_4^1}{\alpha_4}= -\frac14 \left(m_1m_2m_3m_4 k_{13} +   
\frac{m_1^2k_1^2k_3 +m_2^2k_2^2k_4 +m_3^2k_3^2k_1 + m_4^2k_4^2k_2}{k_{12}k_{14}}\right)
\end{equation}
when $\alpha_4\neq0$, and 
$\sigma_4=0$ when $\alpha_4=0$. 
Thus, 
in the second iteration of the $I$-method there are no resonances for the correction term 
as we have 
$|M^1_4(\mathbf{k})|\lesssim |\alpha_4(\mathbf{k})|$ 
for all $\mathbf{k}\in\Gamma_4(\Tl)$.

Therefore, by  \eqref{defnofE1}, \eqref{defnofE2} and \eqref{defnofsigma4}, the second generation modified energy is 
given by
\begin{equation}
\mathcal{E}^2[v]= -\Lambda_2(k_1k_2m_1m_2) + \frac12 \Lambda_4(M_4) ,
\end{equation}
where we set
\begin{equation}
\label{defnofM4}
M_4 :=-\frac{m_1^2k_1^2k_3 +m_2^2k_2^2k_4 +m_3^2k_3^2k_1 + m_4^2k_4^2k_2}{2k_{12}k_{14}} 
\end{equation}
when the denominator does not vanish.  
Note that since $k_{12}k_{14}=0$ in $\Gamma_4(\Tl)$ implies 
$m_1^2k_1^2k_3 +m_2^2k_2^2k_4 +m_3^2k_3^2k_1 + m_4^2k_4^2k_2=0$,  we can set in this cases $M_4:=0$. 

Hence from \eqref{parttE1} and \eqref{parttsigma4}, we get
\begin{gather}
\label{eqn:parttcalE2}
\begin{split}
\frac{d}{dt}\mathcal{E}^2[v(t)] =& \Lambda_6(M^2_6) + \Lambda_8(M^2_8)
 -i\mu[v]\Big(\Lambda_4(K_4^1) +\Lambda_6(K_6^1)+\Lambda_6(K_6^2)\Big)  ,
\end{split}
\end{gather}
where $M^2_6$ and $M^2_8$ are the multipliers given 
by
\begin{align}
\label{defnofM62}
M^2_6 &:=\frac{i}{6}\sum_{j=1}^6 (-1)^jm_j^2k_j^2\\
&\quad -\frac{i}{72}
\sum\limits_{{\{a,c,e\}=\{1,3,5\}}\atop{\{b,d,f\}=\{2,4,6\}}}\;
\Big(M_4(k_{abc},k_d,k_e,k_f)k_b+M_4(k_a,k_{bcd},k_e,k_f)k_c
\nonumber\\
 & \qquad\qquad\qquad\qquad\qquad\qquad +
M_4(k_a,k_b,k_{cde},k_f)k_d+M_4(k_a,k_b,k_c,k_{def})k_e\Big),\nonumber\\
\label{defnofM82}
M^2_8 &:= C_8\sum\limits_{{\{a,c,e,g\}=\{1,3,5,7\}}\atop{
\{b,d,f,h\}=\{2,4,6,8\}}}\!\!
\Big(M_4(k_{abcde},k_f,k_g,k_h)
-M_4(k_a,k_{bcdef},k_g,k_h) \\
 & \qquad\qquad\qquad\qquad\qquad
 + M_4(k_a,k_b,k_{cdefg},k_h) -M_4(k_a,k_b,k_c,k_{defgh})\Big)\nonumber
\end{align}
(as in \cite[Proposition~3.7]{CKSTTrefined} or \cite[p.~2173]{MiaoWuXu}). 
Also, 
\begin{align}
\label{defnofK62}
K_6^2 &:=\sum_{j=1}^4 \mathbb{X}_j^2(\sigma_4).
\end{align}
We note that when proving the estimates on $M_6$ (see Lemma~\ref{pwestsM6}), 
cancelations between the large terms 
coming from the first term in \eqref{defnofM62} and the large terms coming from the sum of $M_4$'s are exploited, 
and thus the coefficients of the two pieces of $M_6$ are critical, 
whereas the constant $C_8=-\frac{i}{2(5!)^2}$ is irrelevant in our analysis. 

\begin{remark}[\textbf{Small Frequencies Remark}]
\label{rmk:smallfrequencies}
Notice that if $|k_j|\ll N$ for all $j$, 
we have $m(k_j)=1$ and thus 
\begin{align}
M_4(\mathbf{k})&=-\frac{k_1k_3k_{13} +k_2k_4(-k_{13})}{2k_{12}k_{14}}
 =\frac{k_{13}}{2}, 
 \text{ for all } \mathbf{k}\in\Gamma_4(\Tl). 
\end{align}
One can similarly check that if $|k_j|\ll N$ for all $j$, 
all the multipliers 
$M_n^g$, $K_n^g$ ($n=4,6,8$, $g=1,2$) vanish.  
\end{remark}

On $\Gamma_n(\Tl)$, 
the largest two frequencies must have comparable sizes 
and thus, without loss of generality,   
we may assume that 
\begin{equation}
\label{N1simN2}
\mathbf{k}\in \Upsilon_n(\Tl):=\{(k_1,\ldots,k_n)\in \Zl^n : |k^*_1|\sim |k^*_2|\gtrsim N\} ,
\end{equation} 
where $N$ is the frequency size threshold of the $I$-operator 
as defined in  Subsection~\ref{subsect:theIoperator}, and 
$(k_1^*,\ldots,k_n^*)$ denotes a rearrangement of $(k_1,\ldots,k_n)$ such that 
$$|k_1^*|\geq |k_2^*|\geq\ldots\geq |k_n^*|.$$ 
We'll also adopt the notation $N_j=|k_j^*|$. 

Due to Remark~\ref{rmk:smallfrequencies}, 
when proving the necessary estimates, 
it is enough to  consider $\mathbf{k}\in\Upsilon_n(\Tl)$, 
i.e. only the region $N_1\sim N_2\gtrsim N$.

\begin{remark}[\textbf{Symmetry Remark}]
\label{rmk:evenoddinvariance}
We point out that the multipliers $M_n^g$'s
that appear throughout this article, 
and consequently the associated multilinear forms 
$\Lambda_n(M_n^g;v_1,v_2,\ldots,v_n)$ are invariant under permutations of the even or of the odd $k_j$ (or $v_j$) indices. 
Also, the same is true (up to sign) if one swaps the set of all odd $k_j$'s (or $v_j$'s) with the set of all even $k_j$'s 
(respectively $v_j$'s). 

Hence, in addition to \eqref{N1simN2}, 
without loss of generality we may assume that 
$$|k_1| \geq |k_3|\geq\ldots \geq |k_{n-1}| \quad,\quad |k_2|\geq |k_4|\geq\ldots\geq |k_n| $$
and 
$$|k_1|\geq |k_2|.$$
If all these are in place, we have $k^*_1=k_1$, 
but  
either $k^*_2=k_2$ or $k^*_2=k_3$. 
\end{remark}

\subsection{Pointwise bounds on the multipliers} 
We provide here the multiplier estimates that are relevant in our analysis, 
namely for the almost conservation estimates of the mo\-di\-fied energy functional 
in Section~\ref{Section:AlmostConservationEstimates} 
and in the estimates of the correction terms in Section~\ref{Sect:AlmostconservedEandP}.   
We recall that we work under the symmetry assumptions on the multipliers $M_n^g$, $K_n^g$ 
mentioned in Remark~\ref{rmk:evenoddinvariance}. 
Also, 
since we rely on \eqref{nondecrasingpropofmsgeq12}, 
the assumption $s\geq \frac12$ is needed for all of the results below. 

Although the multiplier $M_4$ is not involved directly in \eqref{eqn:parttcalE2}, 
the refined bounds (ii) and (iii) below are crucial for  $M_6^2$ and $M_8^2$. 

\begin{lemma}{\cite[Lemma~4.1, 4.2]{CKSTTrefined}}
\label{pwestsM4} 
For $M_4$ defined by \eqref{defnofM4} and $\mathbf{k}\in\Gamma_4(\Tl)$, we have:
\begin{enumerate}
\item[\textup{(i)}] $|M_4(\mathbf{k})| \lesssim m(N_1)^2N_1$;
\item[\textup{(ii)}]  if $|k_1|\sim |k_3|\gtrsim N\gg N_3$, then $|M_4(\mathbf{k})| \lesssim m(N_1)^2N_3$;
\item[\textup{(iii)}]  if  $|k_1|\sim |k_2|\gtrsim N\gg N_3$, then $M_4(\mathbf{k}) = \frac{m(k_1)^2k_2^2}{2k_1} + O(N_3)$.
\end{enumerate}
\end{lemma}

By using the estimate (i) above, 
one can immediately obtain a crude bound for the symbol $M_6^2$ (see (i) below). 
We recall that in \cite{CKSTTrefined}, the refined estimate (ii) below,  
as well as using Bourgain's trick to provide additional denominators, 
make possible the global well-posedness result of \eqref{DNLS} on the real line for $s>\frac12$, 
but not at the end-point $s=\frac12$. 
It is worth mentioning that for (ii), 
in the case $N_3\ll N$ and the largest two frequencies have same parity, 
it was exploited the cancellation 
``between the large terms coming from $\beta_6$ and the large terms of the sum of the $M_4$.'' 
Hence the almost conservation estimate of $\mathcal{E}^2$ owes to 
the specific nonlinear structure $-iv^2\partial_x{\overline{v}} -\frac12 |v|^4v$ 
of the gauged DNLS equation \eqref{g1DNLS} in the Euclidean case. 

\begin{lemma}{\cite[Lemma~6.2]{CKSTTrefined}}
\label{pwestsM6}
For $M_6^2$ defined by \eqref{defnofM62} and $\mathbf{k}\in\Gamma_6(\Tl)$, we have:
\begin{enumerate}
\item[\textup{(i)}] $|M^2_6(\mathbf{k})| \lesssim m(N_1)^2N_1^2$;
\item[\textup{(ii)}] if $N_3\ll N$,  
then $|M^2_6(\mathbf{k})| \lesssim N_1 N_3$.
\end{enumerate}
\end{lemma}


\begin{lemma}
\label{pwestssigma4}
For $\sigma_4$ defined by \eqref{defnofsigma4} and $\mathbf{k}\in\Gamma_4(\Tl)$, we have:
$$|\sigma_4(\mathbf{k})|\lesssim m(N_1)^2N_1.$$
\end{lemma}
\begin{proof}
For $\sigma_4$, one easily notes that $\sigma_4^1:=-\frac14 m_1m_2m_3m_4k_{13}$ is bounded by $m(N_1)^2N_1$ and 
for $\sigma_4^2:=\sigma_4-\sigma_4^1$, we have
Lemma~\ref{pwestsM4} which gives $|\sigma_4^2|\sim |M_4|\lesssim m(N_1)^2N_1$. 
\end{proof}

Another immediate consequence of Lemma~\ref{pwestsM4} is the following:

\begin{lemma}
\label{pwestsM82}
For $M_8^2$ defined by \eqref{defnofM82} and  $\mathbf{k}\in\Gamma_8$, we have:
\begin{enumerate}
\item[\textup{(i)}]  $|M^2_8(\mathbf{k})| \lesssim m(N_1)^2N_1$;
\item[\textup{(ii)}]  if $N_3\ll N$, then $|M^2_8(\mathbf{k})|\lesssim N_3$. 
\end{enumerate}
\end{lemma}

\begin{lemma}
\label{pwestsK41}
For $K_4^1$ defined by \eqref{defnofK41} and $\mathbf{k}\in\Gamma_4(\Tl)$, we have 
\begin{enumerate}
\item[\textup{(i)}]  $|K_4^1(\mathbf{k})| \lesssim m(N_1)^2N_1^2$;
\item[\textup{(ii)}] if $|k_1|\sim|k_2|\gtrsim N\gg N_3$,  
 then $|K^1_4(\mathbf{k})| \lesssim m(N_1)^2 N_1 N_3$.
\end{enumerate}
\end{lemma}
\begin{proof}
The first statement is immediate as $\xi\mapsto m(\xi)^2\xi^2$ is increasing. 
For the second statement, 
$|m'(\xi)|\sim m(\xi)|\xi|^{-1}$ when $|\xi|\gg N$, and 
by the mean value theorem 
$$|m(k_1)^2k_1^2-m(k_2)^2k_2^2| \sim m^2(\theta)|\theta||k_1-(-k_2)|$$
for some $\theta$ between $k_1$ and $-k_2$; hence $|\theta|\sim N_1$ and $m(\theta)^2\sim m(N_1)^2$. 
Since we also have $|k_{12}|=|k_{34}|\lesssim N_3$, we get 
$|m(k_1)^2k_1^2-m(k_2)^2k_2^2| \lesssim m(N_1)^2N_1N_3$. Then, the crude bound 
$$|m(k_3)^2k_3^2-m(k_4)^2k_4^2|\leq m(k_3)^2k_3^2 + m(k_4)^2k_4^2\lesssim m(N_3)^2N_3^2$$
together with $m(N_3)^2N_3\leq m(N_1)^2N_1$, concludes the proof. 
\end{proof}

For the last lemma in this section, 
the first statement is immediate from $0<m(\cdot)\leq 1$, 
while the second follows from Lemma~\ref{pwestssigma4}.

\begin{lemma}
\label{pwestsK6162}
For $K_6^1$, $K_6^2$ defined by \eqref{defnofK61}, \eqref{defnofK62} respectively, 
and $\mathbf{k}\in\Gamma_6(\Tl)$, we have 
\begin{enumerate}
\item[\textup{(i)}]  $|K_6^1(\mathbf{k})| \lesssim 1$;
\item[\textup{(ii)}]  $|K_6^2(\mathbf{k})| \lesssim m(N_1)^2N_1$.
\end{enumerate}
\end{lemma}

\subsection{Necessity of the third iteration of the $I$-method}

To make the matters clear why we need to implement a third generation 
$I$-method, we prove here the decay estimate for $\int\Lambda_6(M^2_6)dt$. 
This part serves two purposes: first, to see how one applies the bilinear estimate 
in order to recover the result of 
\cite[Lemma~7.5]{WinFE2010}, 
and second to uncover the worst case scenarios 
and hence motivate the non-resonant subregions of Subsection~\ref{Sect:TheNonresonantSet}. 

\begin{proposition}
\label{prop:WinestsofLambda6}
For  $s>\frac12$ and $M_6^2$ defined by \eqref{defnofM62}, 
we have the estimate
\begin{equation}
\label{WinsestofLambda6}
\left| \int_0^{\delta}\Lambda_6(M_6^2 ; v(t))\,dt \right| \lesssim 
N^{-1+}\lambda^{-1+} \| I v\|^6_{Z^1([0,\delta] \times\Tl)} .
\end{equation}
\end{proposition}
\begin{proof}
We write $v=\sum_{k\in\Zl} v_j$, with $\supp(\widehat{v_j}) \subset\{ (\tau,k)\in\R\times\Zl : |k|\sim N_j\}$ 
for each $N_j$ dyadic number. Thus, it is enough to estimate 
\begin{equation}
\label{WinsestofLambda6v2}
\int_{\R}\Lambda_6(M_6^2 ; v_1,v_2,\ldots,v_6)\,dt 
\end{equation}
where without loss of generality we can assume, in addition to the frequency localization, that each $\widehat{v_j}$ is real-valued and non-negative. 
This step, as well as why it is enough to consider the time integral on $\R$ rather than on $[0,\delta]$  
can be justified by standard arguments as in Section~\ref{Section:AlmostConservationEstimates}. 

\textbf{Case 1.} $N_1\sim N_2 \gtrsim N\gg N_3$. 
By Lemma~\ref{pwestsM6} (ii) we have $|M_6^2|\lesssim N_1N_3$. 
Notice that for $s>\frac12$ and $\varepsilon>0$ small enough, one obtains 
\begin{equation}
\label{mN12N1gtrsimN1eps}
m(N_1)^2N_1 =N^{2-2s} N_1^{2s-1}= 
N^{1-\varepsilon}\left(\frac{N_1}{N}\right)^{2s-1-\varepsilon}
N_1^{\varepsilon} 
\gtrsim N^{1-\varepsilon} N_1^{\varepsilon} .
\end{equation}
In this case, $m(N_j)=1$ for $j\geq 3$ and therefore 
by \eqref{interpbilinearest} and \eqref{embeddinginLinfty}, 
we get 
\begin{align}
\nonumber
{\eqref{WinsestofLambda6v2}} &\lesssim  
  \int_*\int_{**} \frac{1}{m(N_1)^2N_1 \prod_{j=4}^6\langle k_j\rangle} \prod_{j=1}^6 \widehat{J_xIv_j}\\
 \nonumber
  &\lesssim \frac{N^{-1+}}{N_1^{0+}} \int_{\R}\int_{\Tl} 
   (J_xIv_1) (J_xIv_3)(J_xIv_2) (J_xIv_4) (Iv_5)(Iv_6) dxdt\\
 \nonumber
 &\lesssim \frac{N^{-1+}}{N_1^{0+}} \| (J_xIv_1) (J_xIv_3)\|_{L^2_{t,x}} \| (J_xIv_2) (J_xIv_4)\|_{L^2_{t,x}} 
   \|Iv_5\|_{L^{\infty}_{t,x}} \|Iv_6\|_{L^{\infty}_{t,x}}\\
  \nonumber
   &\lesssim \frac{N^{-1+}\lambda^{-1+}}{N_1^{0+}} \prod_{j=1}^4 \|Iv_j\|_{X^{1,\frac12}}  
       \prod_{j=5,6}  \|Iv_j\|_{Y^{\frac12+,0}}\\
       \label{eqn:prop5p9}
&\lesssim \frac{N^{-1+}\lambda^{-1+}}{N_1^{0+}} \prod_{j=1}^6 \|Iv_j\|_{Z^1}
\end{align}
where $\int_*$ and $\int_{**}$ stand for integration with respect to the measures 
$\delta_0(\tau_1+\ldots +\tau_6)$ and $\delta_0(k_1+\ldots +k_6)$ on $\Gamma_6(\R)$ 
and on $\Gamma_6(\Tl)$, respectively. 
The operator $J_x$ denotes the Bessel potential operator, i.e. $\widehat{J_xf}(k)=\langle k\rangle \widehat{f}(k)$. 

\begin{remark}
For $s=\frac12$, we only have $m(N_1)^2N_1\gtrsim N$ as we cannot afford to borrow an $N_1^{\varepsilon}$ factor as in \eqref{mN12N1gtrsimN1eps} above. 
Notice that 
since there are no other tools to obtain additional decaying factors, 
to make up for the logarithmic loss in $\lambda$, 
as well as to ensure summability, one would need to obtain a better estimate, for example 
\begin{equation}
\label{desiredboundonM62}
|M_6^2|\lesssim N_1^{1-\theta} N_3^{1+\theta}, 
\end{equation}
which gives the following factor in \eqref{eqn:prop5p9}: 
$$N^{-1}\lambda^{-1} \frac{N_3^{\theta}\lambda^{0+}}{N_1^{\theta}} 
\lesssim \frac{N^{-1}\lambda^{-1-} }{N_3^{0+}}$$
(recall that since $s\geq\frac12$, we have $1\leq \lambda\leq N$). 
We note that 
the decaying factor 
$N^{-1}\lambda^{-1-}$ 
would allow us to obtain the global well-posedness result at $s=\frac12$ 
(see Section~\ref{sect:proofofGWPg1DNLS}). 
Although  
the bound \eqref{desiredboundonM62} is not conceivable on the entire $\Gamma_6(\Tl)$, 
such an estimate can be established on a carefully chosen subset 
(see Section~\ref{Sect:TheNonresonantSet}). 
\end{remark}

\textbf{Case 2.} $N_3\gtrsim N\gg N_4$. By Lemma~\ref{pwestsM6} (i) we have $|M_6^2|\lesssim m(N_1)^2N_1^2$, 
and for $s\geq 0$, $m(N_3)N_3\gtrsim N^{-1+}N_3^{0+}$. We then have 
\begin{align*}
{\eqref{WinsestofLambda6v2}} &\lesssim  
  \int_*\int_{**} \frac{1}{m(N_3)N_3 \prod_{j=4}^6\langle k_j\rangle} \prod_{j=1}^6 \widehat{J_xIv_j}\\
  &\lesssim \frac{N^{-1+}}{N_3^{0+}} \int_{*}\int_{**} 
   \widehat{J_xIv_1}\widehat{J_xIv_2}\widehat{J_xIv_3}  \widehat{Iv_4}  \widehat{Iv_5}\widehat{Iv_6} .
\end{align*}
At this point we have to discuss the frequency separation of the first three factors. 

\textbf{Subcase 2.1} If $N_3\sim N_1$, then since $N_3\gg N_4$, two out of the three frequencies $k_1,k_2,k_3$ must have opposite signs, say $k_1$ and $k_2$. Thus $J_xIv_1$ and $J_xIv2$ are separated in frequency, and so are 
$J_xIv_3$ and $J_xI v_4$. We have 
\begin{align*}
{\eqref{WinsestofLambda6v2}} &\lesssim  
\frac{N^{-1+}}{N_1^{0+}} \int_{\R}\int_{\Tl} 
   (J_xIv_1) (J_xIv_2)(J_xIv_3) (J_xIv_4) (Iv_5)(Iv_6) dxdt\\ 
 &\lesssim  \frac{N^{-1+}}{N_1^{0+}} \| (J_xIv_1) (J_xIv_2)\|_{L^2_{t,x}} \| (J_xIv_3) (J_xIv_4)\|_{L^2_{t,x}} 
   \|Iv_5\|_{L^{\infty}_{t,x}} \|Iv_6\|_{L^{\infty}_{t,x}}\\
 &\lesssim \frac{N^{-1+}\lambda^{-1+}}{N_1^{0+}} \prod_{j=1}^6 \|Iv_j\|_{Z^1} .
\end{align*}

\textbf{Subcase 2.2}  If $N_3\ll N_1$, then as in Case 1, we can clearly apply the bilinear estimate \eqref{interpbilinearest} 
to the $L^2_{t,x}$-norms of both $(J_xIv_1)(J_xIv_3)$ and $(J_xIv_2)(J_xIv_4)$ and obtain 
\begin{align*}
{\eqref{WinsestofLambda6v2}} &\lesssim  \frac{N^{-1+}\lambda^{-1+}}{N_3^{0+}} 
\|Iv_1\|_{X^{1,\frac12}} \|Iv_2\|_{X^{1,\frac12}}   \prod_{j=3}^6 \|Iv_j\|_{Z^1}.
\end{align*}
Notice that in this sub case the factor $1/N_3^{0+}$ does not allow direct summation over the dyadic numbers 
$N_1\sim N_2$. However, exploiting the $L^2$-based norm of the space $X^{1,\frac12}$ of the first two factors, 
one can recover the claim 
(see Section~\ref{Section:AlmostConservationEstimates}) 
without any setback. 

\begin{remark}
Notice that although in Case 2 we have three large frequencies ($N_3\gtrsim N\gg N_4$), 
the bound on the weight $M_6^2$ is worse than in Case 1, 
and overall we obtain the same (insufficient) decaying factor of $N^{-1+}\lambda^{-1+}$. 
Therefore we also need to correct for this case. 
\end{remark}

\textbf{Case 3.} $N_4\gtrsim N$. 
By Lemma~\ref{pwestsM6} (i) we have $|M_6^2|\lesssim m(N_1)^2N_1^2$, 
and for $s\geq 0$, $m(N_j)N_j\gtrsim N^{-1+}N_j^{0+}$, $j=3,4$. It follows that 
\begin{align*}
{\eqref{WinsestofLambda6v2}} &\lesssim  
 \int_*\int_{**} \frac{1}{m(N_3)N_3 m(N_4)N_4 \prod_{j=5,6}\langle k_j\rangle} \prod_{j=1}^6 \widehat{J_xIv_j}\\
 &\lesssim \frac{N^{-2+}}{N_3^{0+}} \int_{*} \int_{**}
   \widehat{J_xIv_1}\widehat{J_xIv_2}\widehat{J_xIv_3}  \widehat{J_x Iv_4}  \widehat{Iv_5}\widehat{Iv_6} .
\end{align*}
Although when $\lambda\sim N$, the decaying factor obtained above is just as good as that in the previous cases, 
we can gain here another decaying factor $\lambda^{-\frac12+}$ 
by separating the analysis into subcases $N_3\sim N_1$ and $N_3\ll N_1$, 
as we did in Case~2. 
We obtain
\begin{align*}
{\eqref{WinsestofLambda6v2}} &\lesssim  
\frac{N^{-2+}\lambda^{-\frac12+}}{N_3^{0+}} 
\|Iv_1\|_{X^{1,\frac12}} \|Iv_2\|_{X^{1,\frac12}}   \prod_{j=3}^6 \|Iv_j\|_{Z^1} 
\end{align*}
and since we choose the parameters so that $1\leq \lambda\leq N$, 
we have in this case a better decaying factor. 
\end{proof}

The other sixth order term in \eqref{eqn:parttcalE2} is $\mu[v]\Lambda_4(K_4^1;v)$. 
The  coefficient $\mu[v] =\frac{1}{2\pi\lambda}\|v\|_{L^2(\Tl)}^2$ 
already provides a decaying factor of $\lambda^{-1}$.  In the remark below, 
we investigate  the worst case scenario corresponding to this term. 

\begin{remark}
\label{rmk:necessityofcorrectingforK41}
The pointwise bound $|K_4^1(\mathbf{k})|\lesssim m(N_1)^2N_1^2$ is optimal  
in the case $N_3\ll N$ and the largest two frequencies have the same parity 
(as we have, for example,  $|K_4^1(N_1,0,-N_1,0)|=m(N_1)^2N_1^2$). 
In this case, 
the best estimate that can be obtained is 
\begin{align}
\nonumber
\int_{\R} \Lambda_4(K_4^1 ; v_1,\ldots,v_4)\,dt &\lesssim 
\int_*\int_{**} \frac{1}{\langle k_3\rangle \langle k_4\rangle} \prod_{j=1}^4 \widehat{J_xIv_j}\\
\nonumber
&\lesssim \frac{1}{N_3^{0+}} \int_{\R}\int_{\Tl} 
(J_xIv_1)(J_x^{0+}Iv_3)(J_xIv_2)(Iv_4)dxdt\\
\nonumber
 &\lesssim \frac{\lambda^{-1+}}{N_3^{0+}} \| (J_xIv_1) (J_x^{0+}Iv_3)\|_{L^2_{t,x}} 
  \| (J_xIv_2) (Iv_4)\|_{L^2_{t,x}}\\
 \label{estK41}
&\lesssim \frac{\lambda^{-1+}}{N_3^{0+}} \prod_{j=1}^4\|Iv_j\|_{X^{1,\frac12-}}.
\end{align}
Hence, we have the estimate\footnote{In the region where we have the refined estimate $|K_4^1(\mathbf{k})|\lesssim m(N_1)^2N_1N_3$, one obtains the pre-factor $N^{-1+}\lambda^{-2+}$ in \eqref{estK41}.}  
$$\mu[v]\left| \int_0^{\delta}\Lambda_4(K_4^1;v(t))dt\right| \lesssim \lambda^{-2+} \|Iv\|^6_{Z^1([0,\delta]\times\Tl)}.$$ 
This decay rate is insufficient to reach the regularity index $s=\frac12$. 
Since the bound of $K_4^1$ is optimal and the available tools cannot yield a better estimate, 
we have to provide a second correction term that removes (at least) this case. 
\end{remark}

\subsection{The third generation modified energy} 
\label{sect:thirdmodifiedenergy}
We refine further the choice of modified energy for the $I$-method 
as a refinement of $\mathcal{E}^2$ of the form 
\begin{equation}
\label{defnofE3}
\mathcal{E}^3[v]:= \mathcal{E}^2[v]+ \Lambda_6(\sigma_6;v) + i\mu[v]\Lambda_4(\widetilde{\sigma_4};v) .
\end{equation}
In the same manner as above, we are lead to define the ``correction'' term $\sigma_6$ by imposing 
$M_6+\sigma_6\alpha_6=0$. 
In contrast to the situation of $\alpha_4$ discussed above,  
the set on which $\alpha_6$ vanishes is not small, 
in particular $\alpha_6=0$ does not imply $M_6=0$. 
The idea around this is to define a region $\Omega$   
in the hyperplane $\Gamma_6(\Tl)$ referred to as \emph{the non-resonant set of $\sigma_6$} 
where  $\alpha_6$ clearly does not vanish, 
but also with the property that on $\Omega^c:=\Gamma_6(\Tl)\setminus\Omega$ 
we have satisfactory pointwise estimates on $M^2_6$.  
We can then take 
\begin{align}
\label{defnofsigma6}
\sigma_6 &:=-\frac{M^2_6}{\alpha_6}\cdot \sharpind_{\Omega} \,,
\end{align}
where $\sharpind_{\Omega}$ 
denotes the characteristic function of the set $\Omega$ 
which is defined in Subsection~\ref{Sect:TheNonresonantSet}. 

For the second correction term in \eqref{defnofE3}, 
the situation is simpler 
(since $\alpha_4=0$ implies $K_4^1=0$) 
and we can define  
\begin{align}
\label{defnoftildesigma4}
\widetilde{\sigma_4} &:= \frac{K^1_4}{\alpha_4}
\end{align}
when $\alpha_4\neq0$, and $\widetilde{\sigma_4}:=0$ when $\alpha_4=0$. 

Using \eqref{diffrule}, we find that 
\begin{gather}
\label{parttE3}
\begin{split}
\frac{d}{dt} \mathcal{E}^3[v(t)] =& \Lambda_6(M^2_6\cdot\sharpind_{\Omega^c}) + 
 \Lambda_8(M^2_8)+ \Lambda_8(M^3_8) 
+\Lambda_{10}(M^3_{10})\\
& -i\mu[v]\Big(
\Lambda_4(K_4^1)+\Lambda_6(K_6^1)+ \Lambda_6(K_6^2) +\Lambda_6(\widetilde{K_6^3}) 
 +\Lambda_8(K^3_8) +\Lambda_8(\widetilde{K^3_8})\Big)\\
 & + \mu[v]^2 \Lambda_6(\widetilde{K^4_6})
\end{split}
\end{gather}
where the additional terms 
(i.e. the ones corresponding to the two correction terms 
$\sigma_6$ and $\widetilde{\sigma_4}$)
are given by 
\begin{align}
\label{defnofM83}
M^3_8 &:= -i\sum_{j=1}^6 \mathbb{X}_j^2(\sigma_6)k_{j+1} ,\\
\label{defnofK83}
K^3_8 &:= \sum_{j=1}^6 \mathbb{X}_j^2(\sigma_6) ,\\ 
\label{defnofM103}
M^3_{10} &:= \frac{i}{2} \sum_{j=1}^6 (-1)^{j+1} \mathbb{X}_j^4(\sigma_6) ,\\
\label{defnofK63tilde}
\widetilde{K_6^3} &:= i \sum_{j=1}^4 \mathbb{X}_j^2(\widetilde{\sigma_4})k_{j+1} ,\\
\label{defnofK64tilde}
\widetilde{K_6^4} &:= \sum_{j=1}^4 \mathbb{X}_j^2(\widetilde{\sigma_4}) ,\\
\label{defnofK83tilde}
\widetilde{K_8^3} &:= \frac{i}{2} \sum_{j=1}^4 (-1)^j \mathbb{X}_j^4(\widetilde{\sigma_4}) .
\end{align}

\subsection{A non-resonant set for $\alpha_6$} 
\label{Sect:TheNonresonantSet}
We now turn to describing the set $\Omega$, 
as it was introduced in \cite{MiaoWuXu}. 
With the simplifying assumptions of Remark~\ref{rmk:evenoddinvariance}
in place, 
let us analyze the expression 
$$i\alpha_6= k_1^2-k_2^2+k_3^2-k_4^2+k_5^2-k_6^2.$$

If precisely two frequencies have sizes above the threshold $N$, we distinguish the following two cases. 

\noindent
\textbf{Case 1.} If the largest two frequencies have the same parity, 
then clearly $|\alpha_6|\gtrsim N_1^2$. 
The corresponding non-resonant region is defined to be 
\begin{equation}
\label{defnofOmega1}
\Omega_1 := \{\mathbf{k}\in \Upsilon_6(\Tl) : |k_1|\sim |k_3| \gtrsim N\gg N_3 \}. 
\end{equation}
This definition is just slightly different from the analogous one in \cite[Section~3]{MiaoWuXu} 
and it does not affect the estimates. 

\noindent
\textbf{Case 2.} If the largest two frequencies have opposite parity, 
say $k_1$ and $k_2$, then on $\Gamma_6(\Tl)$ it must be that 
$k_1=-k_2+O(N_3)$ and 
$$i\alpha_6 =k_{12}(k_1-k_2)+O(N_3^2). $$
While $k_1-k_2=O(N_1)$, it is possible to have $k_{12}=0$ and 
$\alpha_6=0$. Even if the latter does not happen,  
a too weak lower bound on $\alpha_6$ renders an insufficiently good upper bound on $M_8^3$ 
(one of the multipliers that involve $\sigma_6=-\frac{M_6^2}{\alpha_6}$, 
see \eqref{defnofM83}).  
As in \cite{MiaoWuXu}, we consider the following subregion
\begin{equation}
\label{defnofOmega2}
\Omega_2 := \left\{\mathbf{k}\in \Upsilon_6(\Tl) : |k_1|\sim |k_2| \gtrsim N\gg N_3 
   \text{ and } |k_{12}| \gtrsim \left( \frac{N_3}{N_1}\right) ^{\frac12}N_3\right\}. 
\end{equation}

\begin{remark}
Notice that in the case $|k_1|\sim |k_2|\gtrsim N\gg N_3$,  
we have $|k_{12}|=|k_{3456}| \lesssim N_3$. 
On the other hand, by looking to ensure $\alpha_6\neq 0$, the natural bound to impose is 
$|k_{12}|\gtrsim \frac{N_3^2}{N_1}$. However,  while the latter gives a better bound on 
the remaining part $M_6^2\ind_{\Omega^c}$ of $\frac{d}{dt}\mathcal{E}^3$, it does not allow for a satisfactory bound on the correction multiplier $\sigma_6$ (which appears, for example, in $M_8^3$). 
At the other extreme, correcting only in the region 
$|k_{12}|\sim N_3$ does not produce a small enough bound on $M_6^2\ind_{\Omega^c}$.   
We would like to point out that 
(here, as well as in the Euclidean setting \cite{MiaoWuXu}), 
the choice of $\frac12$ in the exponent is not essential, as any lower bound of the form 
$$|k_{12}|\gtrsim  \left( \frac{N_3}{N_1}\right) ^{\theta}N_3$$
(for some $0< \theta <1$) produces the extra $N^{-\theta}$ decay factor needed 
to reach $s=\frac12$.
\end{remark}

\textbf{Case 3.} 
Finally, since the decay factors in the estimate of $\Lambda_6(M_6^2)$-term were also critical in the case 
$N_3\gtrsim N\gg N_4$ (see Case 2 in the proof of Proposition~\ref{prop:WinestsofLambda6}), 
we need to correct for it in this region as well. 
When three frequency sizes are much larger than the remaining frequency sizes,  
$\alpha_6$ does not vanish as we have $|\alpha_6|\gtrsim N_3^2$. 
Therefore, we define 
\begin{equation}
\label{defnofOmega3}
\Omega_3 := \{\mathbf{k}\in \Upsilon_6(\Tl) : N_3\gg N_4\}
\end{equation}
We point out that the correction is deliberately intended for the larger region $N_3\gg N_4$ (i.e. $\Omega_3$) 
rather than $N_3\gtrsim N\gg N_4$, 
since on $\Omega_3$ we have 
\begin{equation}
|k_1^* + k_2^*| = |k_3^*+k_4^*+k_5^*+k_6^*|\sim N_3 \gtrsim \left( \frac{N_3}{N_1}\right) ^{\frac12}N_3. 
\end{equation}

Correcting for $M_6^2$ in these three subregions of $\Upsilon_6(\Tl)$ is enough for our goal, 
hence we consider 
$\Omega: =\Omega_1\cup \Omega_2\cup \Omega_3$ 
to be the non-resonant set of $\alpha_6$, 
and in what follows we denote $\Omega^c:= \Upsilon_6(\Tl)\setminus\Omega$.

\subsection{Pointwise bounds on the multipliers (continued)} 
In this section we first recall the pointwise estimates obtained by Miao, Wu, and Xu \cite{MiaoWuXu}, 
and then we establish the bounds needed to handle the second correction term in \eqref{defnofE3}. 

\begin{lemma}{\cite[Corollary~4.1]{MiaoWuXu}}
\label{pwestsM63}
For $M_6^2$ defined by \eqref{defnofM62} and $\mathbf{k}\in\Gamma_6$, we have:
\begin{enumerate}
\item[\textup{(i)}] if $N_3\ll N$, then $|M^2_6(\mathbf{k})|\lesssim N_1|k_1^*+k_2^*| + N_3^2$; 
\item[\textup{(ii)}] if $N_3\ll N$ and $\mathbf{k}\in\Omega^c$, 
then $|M^2_6(\mathbf{k})|\lesssim  N_1^{\frac12} N_3^{\frac32}$. 
\end{enumerate}
\end{lemma}

\begin{lemma}{\cite[Lemma~4.9]{MiaoWuXu}}
\label{pwestssigma6}
For $\sigma_6$ defined by \eqref{defnofsigma6} and $\mathbf{k}\in\Gamma_6(\Tl)$, we have:  
\begin{enumerate}
\item[\textup{(i)}] $|\sigma_6(\mathbf{k})|\lesssim 1$;
\item[\textup{(ii)}] if $\mathbf{k}\in \Omega_1\cap \{N_3\ll N\}$, 
then $|\sigma_6(\mathbf{k})|\lesssim \frac{N_3}{N_1}$. 
\end{enumerate}
\end{lemma}

\begin{lemma}{\cite[Proposition~4.3]{MiaoWuXu}}
\label{pwestsM83}
For $M_8^3$ defined by \eqref{defnofM83} and $\mathbf{k}\in\Gamma_8$, we have:
\begin{enumerate}
\item[\textup{(i)}]  $|{M^3_8}(\mathbf{k})|\lesssim N_1$;
\item[\textup{(ii)}]  if $N_3\ll N$, then 
 $|{M^3_8}(\mathbf{k})|\lesssim N_1^{\frac12}N_3^{\frac12}$. 
\end{enumerate}
\end{lemma}

Also, as direct  consequences of the above Lemma~\ref{pwestssigma6}, we have the 
same bounds for $K_8^3$ and $M_{10}^3$ (see \eqref{defnofK83} and  \eqref{defnofM103} ) as for $\sigma_6$. 
Finally, we provide the pointwise estimates corresponding to the second correction term in \eqref{defnofE3}. 

\begin{lemma}
\label{pwestssigma4tilde}
For $\widetilde{\sigma_4}$ defined by \eqref{defnoftildesigma4} 
and $\mathbf{k}\in\Gamma_4(\Tl)$, we have: 
\begin{enumerate}
\item[\textup{(i)}]  $|\widetilde{\sigma_4}(\mathbf{k})| \lesssim m(N_1)^2N_1$;
\item[\textup{(ii)}]  if $N_3\ll N$, then $|\widetilde{\sigma_4}(\mathbf{k})| \lesssim m(N_1)^2$. 
\end{enumerate}
\end{lemma}
\begin{proof}
Let $\beta_4$ denote the numerator in \eqref{defnoftildesigma4}.  
We have the crude estimate \linebreak
$|\beta_4|\lesssim m(N_1)^2N_1^2$, and note that either $\alpha_4=0$ (in which case $\widetilde{\sigma_4}=0$) or $|\alpha_4|\geq 2N_1$. 
Depending on the parity of the largest two frequencies, we distinguish two cases.

If $k_1^*=k_1$ and $k_2^*=k_3$, then $|\alpha_4|\sim |k_{12}| |k_{14}|\sim N_1^2$
and  $|\beta_4|\sim m(N_1)^2N_1^2$.

If $k_1^*=k_1$ and $k_2^*=k_2$, then $|\alpha_4|\sim N_1|k_{34}|$, $k_1$ and $k_2$ have opposite signs and by the mean value theorem, we have 
\begin{align*}
|\beta_4| \leq& |m(k_1)^2k_1^2 - m(-k_2)^2(-k_2)^2| + |k_{34}| \cdot |k_3-k_4|\\
 \leq& |k_{12}|\cdot|(m(\xi)^2\xi^2)'| + |k_{34}| \cdot |k_3-k_4| ,
\end{align*}
where $|\xi|\sim N_1$ and thus 
$$\left|\frac{d}{d\xi}(m(\xi)^2\xi^2) \right| 
   \sim |m(\xi)^2\xi|\sim m(N_1)^2N_1.$$ 
Since 
$$|k_3-k_4|\lesssim N_3\ll N\lesssim m(N_1)^2N_1,$$ 
we get 
$$|\beta_4|\lesssim m(N_1)^2N_1|k_{34}|$$
and the conclusion follows. 
\end{proof}

Consequently, 
by simply referring to their definitions in \eqref{defnofK64tilde} and \eqref{defnofK83tilde}, 
we also have the same bounds 
for $\widetilde{K_6^4}$ and $\widetilde{K_8^3}$, respectively,  
as for $\widetilde{\sigma_4}$, 
In the same manner, we have the following lemma.

\begin{lemma}
For $\widetilde{K_6^3}$ defined by \eqref{defnofK63tilde} 
and $\mathbf{k}\in\Gamma_6(\Tl)$, we have: 
\begin{enumerate}
\item[\textup{(i)}] $|\widetilde{K_6^3}(\mathbf{k})| \lesssim m(N_1)^2N_1^2$;
\item[\textup{(ii)}] if $N_3\ll N$, then $|\widetilde{K_6^3}(\mathbf{k})| \lesssim m(N_1)^2N_1$.
\end{enumerate}
\end{lemma}

\section{Almost conservation estimates for  the third generation modified energy}
\label{Section:AlmostConservationEstimates}

The scope of  this section is to show that for a (local) $H^s$-solution $v$ of \eqref{g1DNLS}, 
with the life-span provided by Proposition~\ref{prop:LWPIsyst},  
the possible increase of $\mathcal{E}^3[v(\cdot)]$ is minuscule, i.e. that we have an estimate of the form 
\begin{equation}
\label{eqn:AlmostConservationE3}
\left| \mathcal{E}^3[v(\delta)] -\mathcal{E}^3[v(0)] \right| \lesssim 
 N^{-\gamma}\lambda^{-\kappa} 
\end{equation}
for some $\gamma,\kappa>0$. 
\footnote{The powers $\gamma, \kappa$ are  responsible for the level of regularity at which 
the global existence via the $I$-method is obtained. 
Hence the name of the game in subsequent iterations of the method is finding a functional 
that can provide good enough decay rates in order to reach $s=\frac12$. 
One cannot do better than the optimal local well-posedness result available, even though the powers in \eqref{eqn:AlmostConservationE3} may allow a better result. 
For an argument that shows that the local well-posedness theory via the contraction mapping argument 
for \eqref{g1DNLS} in the periodic setting is optimal, see Appendix~\ref{appdx:illposedness}.} 
On the right hand side we use (powers of) the $Z^1$-norm of  $Iv$ 
who, we recall, lives on the scaled spatial domain $\Tl$ and 
whose energy on frequencies $\gtrsim N$ is damped by the operator $I$. 
 
We decompose the solution $v$ using the Litllewood-Paley projectors in spatial frequencies:
$$v=\sum_{j=0}^{\infty} P_{2^j}v \quad,\quad \widehat{P_{2^j} v}(\tau,k)=\ind_{I_j}(n)\widehat{v}(\tau,k) , $$
where 
$I_0:=\{k\in\Zl : |k|<1\}$ and 
$I_j:=\{k\in\Zl : 2^{j-1}\leq |k|<2^j\}$ for $j\geq 1$. 
By the fundamental theorem of calculus, the proof of \eqref{eqn:AlmostConservationE3} 
reduces to estimating expressions of the form  
$$\int_{0}^{\delta}\Lambda_n(M_n;v(t))\,dt$$
corresponding to the multipliers ${M_n}$ 
that appear in \eqref{parttE3}. It is enough\footnote{ 
Indeed, one can take the functions $v_j$ such that the time restrictions 
${v_j}_{|[0,\delta]}= P_jv$ and 
$ \|v_j\|_{Z^1(\R\times\Tl)} \leq \|P_jv\|_{Z^1([0,\delta]\times\Tl)}  +\varepsilon$ 
for odd $j$'s,  
and similarly with $P_j\overline{v}$ for even $j$'s. 
Eventually one takes $\varepsilon\to0$ to obtain the estimate. 
} to obtain estimates 
for 
\begin{equation}
\label{genericmultilinexpression}
\int_{\R}\ind_{[0,\delta]}(t)\Lambda_n(M_n;v_1(t),\ldots,v_n(t))\,dt ,
\end{equation}
where each $v_j$ has Fourier support in the band $\{(\tau,k): |k|\sim N_j\}$, with $N_j\sim |I_j|$. 
If $N_j\ll N$ for all $j$, 
the multiplier $M_n$ vanishes, 
hence we assume $N_1\sim N_2\gtrsim N$ 
(see Remark~\ref{rmk:smallfrequencies}). 
Due to 
the Symmetry Remark \ref{rmk:evenoddinvariance}, 
we can also assume 
that $N_1\geq N_2\geq\ldots\geq N_n$. 

Regarding the sharp time-cutoff, we note that in each case, we are able to place 
at least a few factors in the $X^{1,\frac12-}$-norm (rather than in the $Z^1$-norm) 
and since we know that 
$$\|\ind_{[0,\delta]}(t)\|_{H_t^{\frac12-}} \lesssim \delta^{0+}, $$
by Lemma~\ref{sharpcutoffX112minus}, we have 
\begin{equation}
\|\ind_{[0,\delta]} Iv\|_{X^{1,\frac12-}} \lesssim \delta^{0+} \|Iv\|_{X^{1,\frac12}}.
\end{equation}
Therefore, in proving the results of this section, 
we are concerned with estimates of the form 
\begin{equation}
\label{genericmultilinearestimate}
\int_{\R}\Lambda_n(M_n;v_1(t),\ldots,v_n(t))\,dt \lesssim 
N^{-\gamma}\lambda^{-\kappa}\prod_{j=1}^n \|Iv_j\|_{Z^1(\R\times\Tl)} ,
\end{equation}
where $v_j=P_{N_j}v_j$ for all $j$. 

Before starting to prove them one by one for each term that appears in \eqref{parttE3}, we make some further reductions common to all of them.   

\begin{remark}
Since the norms on the right hand side of \eqref{genericmultilinearestimate} 
depend on $|\widehat{v_j}|$, 
for the sake of simplified writing, we assume that all $\widehat{v_j}$'s are real-valued and non-negative.  
\end{remark}

\begin{remark}
To ensure summability over all dyadics $N_1\geq N_2\geq\ldots \geq N_n$,  
we can most of the times obtain a factor of 
$1/N_1^{0+}$ 
on the right hand side above 
since then
$$\frac{1}{N_1^{0+}} \prod_{j=1}^n \|IP_{N_j}v_j\|_{Z^1(\R\times\Tl)}
  \lesssim 
   \left(\prod_{j=1}^n \frac{1}{N_j^{0+}} \right)\|Iv_j\|^n_{Z^1(\R\times\Tl)} \ ,$$
and the summation (first over $N_n$, lastly over $N_1$) is straightforward. 
However, having $L^2_{\tau,k}$-based norms on the largest two frequency factors $Iv_1$ and $Iv_2$ 
allows one to relax the summability factor to $1/N_3^{0+}$ in the region $N_1\sim N_2$.  
This essentially follows from an application of Cauchy-Schwarz inequality. 
Indeed, 
suppose that we have 
\begin{equation*}
|\mathcal{L}(P_{N_1}v_1,P_{N_2}v_2)|\leq A 
 \|IP_{N_1}v_1\|_{X^{1,\frac12}} \|IP_{N_2}v_2\|_{X^{1,\frac12}} 
\end{equation*}
for the bilinear functional $\mathcal{L}$ defined by fixing $v_3,\ldots, v_n$ 
in the left hand side of \eqref{genericmultilinearestimate}. 
Let $N_1=2^{j_1}$ and $N_2=2^{j_2}$. 
Summing over the pair of dyadic numbers  $(N_1,N_2)$ 
in the region $N_1\sim N_2$ 
amounts to summing over the pair of integers $(j_1,j_2)$ with $|j_1-j_2|\leq 4$. \footnote{For $n\leq 10$, 
on $\Gamma_n(\Tl)$ we have  $\frac{1}{9}N_2\leq N_1\leq 9 N_2$ 
and thus  the universal constant bounding $|j_1-j_2|$.}
Therefore
 \begin{align*}
\sum_{N_1\sim N_2}|\mathcal{L}(P_{N_1}v_1,P_{N_2}v_2)| 
 &\leq A \left(\sum_{j_1\in\Z} \|P_{2^{j_1}}w_1\|_{L^2_{t,x}}^2 \right)^{\frac12}
  \left(\sum_{{j_2\in\Z, |j_2-j_1|\leq 4}} \|P_{2^{j_2}}w_2\|_{L^2_{t,x}}^2 \right)^{\frac12}\\
  &\lesssim A \|w_1\|_{L^2_{t,x}} \|w_2\|_{L^2_{t,x}} ,
\end{align*}
where we have taken $w_j$ to be defined by 
$\widehat{w_j}(\tau,k) =m(k)\langle k\rangle \langle \tau+k^2\rangle^{\frac12} \widehat{v_j}(\tau,k)$. 
\end{remark}

With these reduction remarks at hand, we can proceed to the proof of almost conservation estimates. 
We denote by $J_x$ the Bessel potential operator in the spatial variable, 
i.e. $\widehat{J_xf}(k)=\langle k\rangle \widehat{f}(k)$. 
For simplicity, $\int_*$ and $\int_{**}$ stand for integration with respect to the measures 
$\delta_0(\tau_1+\ldots +\tau_n)$ and $\delta_0(k_1+\ldots +k_n)$ 
on $\Gamma_n(\R)$ and on $\Gamma_n(\Tl)$, respectively. 

\begin{lemma}
\label{prop:estofM62}
Let $s\geq\frac12$ and $\delta>0$. For 
$M_6^2$ defined by \eqref{defnofM62}, 
and $\Omega^c$ as in Subsection \ref{Sect:TheNonresonantSet}, 
we have the estimate
\label{prop:EstimateSigma6}
\begin{equation}
\label{MultilinearEst:M62Omegac}
\left| \int_0^{\delta}\Lambda_6(M_6^2\ind_{\Omega^c};v(t))\,dt \right| \lesssim 
N^{-\frac32+}\lambda^{-1+} \delta^{0+} \|Iv\|^6_{Z^1([0,\delta]\times\Tl)} .
\end{equation}
\end{lemma}
\begin{proof}
We distinguish several subregions of $\Omega^c$, but first note that 
for all $\mathbf{k}\in \Upsilon_6(\Tl)\setminus \Omega_3$ 
we have $N_3\sim N_4$. 

\textbf{Case 1:} $N_1\sim N_2\gtrsim N\gg N_3$. 
Note that $m(N_j)=1$ for $j\geq 3$, and 
$m(N_1)^2N_1\gtrsim N.$ 
By Lemma~\ref{pwestsM63}, 
we have $|M^2_6\sharpind_{\Omega^c}|\lesssim N_1^{\frac12} N_3^{\frac32}$
and by using \eqref{nondecrasingpropofmsgeq12}, 
we get
\begin{gather*}
\begin{split}
\int_{\R}\Lambda_6 (M_6^2\ind_{\Omega^c};v_1,\ldots,v_6)\,dt
&\lesssim
 \int_*\int_{**} \frac{1}{m(N_1)^2N_1^{\frac32}N_3^{\frac12}
  \prod_{j=5,6}\langle k_j\rangle} \prod_{j=1}^6 \widehat{J_xIv_j} \\
&\lesssim \frac{N^{-\frac32+}}{N_1^{0+}} \int_* \int_{**}
\frac{1}{\langle k_5\rangle \langle k_6\rangle} 
\prod_{j=1}^6 \widehat{J_xIv_j}\\
&\lesssim \frac{N^{-\frac32+}}{N_1^{0+}} \int_{\R}\int_{\Tl}
   ({J_xIv_1}) ({J_xIv_3}) ({J_xIv_2}) ({J_xIv_4})(Iv_5)(Iv_6)dxdt\\
&\lesssim \frac{N^{-\frac32+}}{N_1^{0+}} 
 \|(J_xIv_1)(J_xIv_3)\|_{L^2_{t,x}} \|(J_xIv_2)(J_xIv_4)\|_{L^2_{t,x}} 
\prod_{j=5,6}\|Iv_j\|_{L^{\infty}_{t,x}} . 
\end{split}
\end{gather*}
By \eqref{interpbilinearest} and \eqref{embeddinginLinfty}, we thus get
\begin{gather*}
\begin{split}
\int_{\R}\Lambda_6 &(M_6^2\ind_{\Omega^c};v_1,\ldots,v_6)\,dt \lesssim 
 \frac{N^{-\frac32+}\lambda^{-1+}}{N_1^{0+}} 
 \prod_{j=1}^4 \|Iv_j\|_{X^{1,\frac12}} 
 \prod_{j=5,6}\|Iv_j\|_{Y^{\frac12+,0}}.
\end{split}
\end{gather*}

The case $N_3\gtrsim N\gg N_4$ is vacuous on $\Omega^c$ 
and thus the next case we have to consider is the one in which precisely 
four of the frequencies have sizes larger than the threshold $N$.  

\textbf{Case 2:} $N_4\gtrsim N\gg N_5$. 
We also have $N_1\sim N_2$, $N_3\sim N_4$ and for $j=3,4$,
\begin{equation}
\label{obtainingN30plus}
m(N_j)N_j=  N^{1-}\left(\frac{N_j}{N}\right)^{s-} N_j^{0+} 
\gtrsim N^{1-}N_j^{0+}.
\end{equation}
By using the crude estimate $|M^2_6|\lesssim m(N_1)^2N_1^2$ of Lemma~\ref{pwestsM6},  
we  estimate
\begin{gather}
\begin{split}
\int_{\R}\Lambda_6 (M_6^2\ind_{\Omega^c};v_1,\ldots,v_6)\,dt\  &\lesssim 
 \int_*\int_* \frac{1}{m(N_3)^2N_3^2 \langle k_5\rangle \langle k_6\rangle} \prod_{j=1}^6 \widehat{J_xIv_j}\\
&\lesssim \frac{N^{-2+}}{N_3^{0+}} \int_* 
 \frac{1}{\langle k_5\rangle \langle k_6\rangle}
 \prod_{j=1}^6  \widehat{J_xIv_j}.
\end{split}
\end{gather}
We now discuss two subcases. 

\textbf{Subcase 2.1.} If $N_3\sim N_1$,  since $N_5\ll N_4$, 
two out of the four frequencies $k_1,k_2,k_3,k_4$ must have opposite signs, say $k_1$ and $k_2$. 
Therefore $v_1$ and $v_2$ are separated in frequency and 
we use the bilinear estimate \eqref{interpbilinearest},  
and together with the $L^4$-Strichartz estimate \eqref{L4Strichartz}, we obtain 
\begin{gather*}
\begin{split}
\int_{\R}\Lambda_6 (M_6^2\ind_{\Omega^c};v_1,\ldots,v_6)\,dt
&\lesssim 
\frac{N^{-2+}}{N_1^{0+}} \|(J_xIv_1)(J_xIv_2)\|_{L^2_{t,x}} 
\prod_{j=3,4} \|J_xIv_j\|_{L^4_{t,x}} 
\prod_{j=5,6} \|Iv_j\|_{L^{\infty}_{t,x}}\\  
 &\lesssim \frac{N^{-2+}\lambda^{-\frac12+}}{N_1^{0+}} 
  \prod_{j=1}^4 \|J_xIv_j\|_{X^{0,\frac12}} 
  \prod_{j=5,6}\|Iv_j\|_{Y^{\frac12+,0}} .
\end{split}
\end{gather*}

\textbf{Subcase 2.2.} If $N_3\ll N_1$, 
then we apply the bilinear estimate \eqref{interpbilinearest} twice and get
\begin{gather*}
\begin{split}
\int_{\R}\Lambda_6 &(M_6^2\ind_{\Omega^c};v_1,v_2,\ldots,v_6)\,dt\\
&\lesssim 
\frac{N^{-2+}}{N_3^{0+}} \|(J_xIv_1)(J_xIv_3)\|_{L^2_{t,x}} 
\|(J_xIv_2)(J_xIv_4)\|_{L^2_{t,x}} 
\prod_{j=5,6} \|Iv_j\|_{L^{\infty}_{t,x}}\\
 &\lesssim \frac{N^{-2+}\lambda^{-1+}}{N_3^{0+}} 
  \prod_{j=1}^4 \|J_xIv_j\|_{X^{0,\frac12}} 
  \prod_{j=5,6}\|Iv_j\|_{Y^{\frac12+,0}} .
\end{split}
\end{gather*}

\textbf{Case 3:} $N_5\gtrsim N$. 
We use \eqref{obtainingN30plus} for $j=3,4,5$,  
$m(k_6)\langle k_6\rangle^{\frac12} \gtrsim 1$, and $N_5\geq N_6$ 
to deduce
\begin{gather*}
\begin{split}
\int_{\R}\Lambda_6 (M_6^2\ind_{\Omega^c};v_1,\ldots,v_6)\,dt
&\lesssim 
\int_*\int_{**} \frac{1}{\prod_{j=3}^6  m(k_j)\langle k_j\rangle} 
\prod_{j=1}^6 \widehat{J_xIv_j} \\
&\lesssim \frac{N^{-3+}}{N_3^{0+}} 
\int_*\int_{**}  \widehat{J_xIv_1} \widehat{J_xIv_2}
\prod_{j=3}^5\left(
\frac{\lambda^{0+}}{\langle k_j\rangle^{0+}}\widehat{J_xIv_j}\right)\,
 \left(\frac{1}{\langle k_6\rangle^{\frac12+}}\widehat{J_xIv_6}\right)\\ 
&\lesssim \frac{N^{-3+}\lambda^{0+}}{N_3^{0+}} 
 \prod_{j=1,2}\|J_xIv_j\|_{L^4_{t,x}} 
 \prod_{j=3}^5 \|J_x^{1-}Iv_j\|_{L^6_{t,x}} 
 \|J_x^{\frac12-}Iv_6\|_{L^{\infty}_{t,x}} .
\end{split}
\end{gather*}
The factors $\lambda^{0+}$ above appear 
due to the application of \eqref{uniformboundsofjapanesebracket}. 
By using the  Strichartz estimates \eqref{L4Strichartz} and \eqref{interpL6Strichartz}, 
as well as the embedding \eqref{embeddinginLinfty}, 
we have
\begin{gather*}
\begin{split}
\int_{\R}\Lambda_6 (M_6^2\ind_{\Omega^c};v_1,\ldots,v_6)\,dt
&\lesssim 
\frac{N^{-3+}\lambda^{0+}}{N_3^{0+}} \prod_{j=1,2} \|J_xIv_j\|_{X^{0,\frac38}}
\prod_{j=3}^5 \|J_x^{1-}Iv_j\|_{X^{0+,\frac12}} 
\|J_x^{\frac12-}Iv_6\|_{Y^{\frac12+,0}}\\
&\lesssim \frac{N^{-3+}\lambda^{0+}}{N_3^{0+}} 
\prod_{j=1}^5 \|Iv_j\|_{X^{1,\frac12}} \|Iv_6\|_{Y^{1,0}} .
\end{split}
\end{gather*}

Since in Section~\ref{sect:proofofGWPg1DNLS} we choose $\lambda, N$ such that 
$1\leq \lambda\leq N$ (for $s\geq\frac12$), 
in the second and third cases we have faster decaying factors than in Case 1.  
\end{proof}

\begin{lemma}
Let $s\geq\frac12$ and $\delta>0$. For $M_8^3$ defined by \eqref{defnofM83}, 
we have the estimate
\begin{equation}
\label{MultilinearEst:M83}
\left| \int_0^{\delta}\Lambda_8(M_8^3;v(t))\,dt \right| \lesssim 
{N^{-\frac32+}\lambda^{-1+}} \delta^{0+}
\|Iv\|^8_{Z^1([0,\delta]\times\Tl)} .
\end{equation}
The same estimate holds if $M_8^3$ is replaced by $M_8^2$. 
\end{lemma}
\begin{proof}
By Lemma~\ref{pwestsM83}, we have 
$|M_8^3(\mathbf{k})|\lesssim N_1$ for all $\mathbf{k}\in\Gamma_8(\Tl)$, 
and if $N_3\ll N$, then $|M_8^3(\mathbf{k})|\lesssim N_1^{\frac12}N_3^{\frac12}$. 

We distinguish three cases and in all of them we use that 
$m(N_1)^2N_1\gtrsim N$, and 
when $N_3\gtrsim N$,   
$m(N_3)N_3\gtrsim N^{1-}N_3^{0+}$ as in \eqref{obtainingN30plus}. 
 
\textbf{Case 1:} $N_1\sim N_2\gtrsim N\gg N_3$. We have
\begin{align*}
\int_{\R}\Lambda_8 (M_8^3;v_1,\ldots,v_8)\,dt&\lesssim 
\int_*\int_{**} \frac{1}{m(N_1)^2N_1^{\frac32}  N_3^{\frac12} 
 \prod_{j=4}^8 \langle k_j\rangle} \prod_{j=1}^8 \widehat{J_xIv_j}\\
 &\lesssim \frac{N^{-\frac32+}}{N_1^{0+}} 
 \int_*\int_{**} (\widehat{J_xIv_1}\widehat{J_xIv_3}) 
 (\widehat{J_xIv_2}\widehat{J_xIv_4})  
 \prod_{j=5}^8 \widehat{Iv_j}\\
 &\lesssim \frac{N^{-\frac32+}}{N_1^{0+}} 
 \|(J_xIv_1)(J_xIv_3)\|_{L^2_{t,x}}\|(J_xIv_2)(J_xIv_4)\|_{L^2_{t,x}} 
 \prod_{j=5}^8 \|Iv_j\|_{L^{\infty}_{t,x}}\\
 &\lesssim \frac{N^{-\frac32+}\lambda^{-1+}}{N_1^{0+}} 
 \prod_{j=1}^4 \|Iv_j\|_{X^{1,\frac12}} 
 \prod_{j=5}^8 \|Iv_j\|_{Y^{\frac12+,0}} .
\end{align*}

\textbf{Case 2:} $N_3\gtrsim N\gg N_4$. 
Here, we get
\begin{align*}
\int_{\R}\Lambda_8 (M_8^3;v_1,\ldots,v_8)\,dt &\lesssim 
\int_*\int_{**} \frac{1}{m(N_1)^2N_1 m(N_3)N_3 \prod_{j=4}^8 \langle k_j\rangle }
\prod_{j=1}^8 \widehat{J_xIv_j}\\
&\lesssim 
\frac{N^{-2+}}{N_3^{0+}} \int_*\int_{**} (\widehat{J_xIv_1} \widehat{J_xIv_4})  (\widehat{J_xIv_2})(\widehat{J_xIv_3}) 
 \prod_{j=5}^8 \widehat{Iv_j}\\
 &\lesssim \frac{N^{-2+}}{N_3^{0+}} 
 \|(J_xIv_1)(J_xIv_4)\|_{L^2_{t,x}}\|J_xIv_2\|_{L^4_{t,x}}\|J_xIv_3\|_{L^4_{t,x}}  
 \prod_{j=5}^8 \|Iv_j\|_{L^{\infty}_{t,x}}\\
 &\lesssim \frac{N^{-2+}\lambda^{-\frac12+}}{N_3^{0+}} \prod_{j=1}^4 \|Iv_j\|_{X^{1,\frac12}} 
  \prod_{j=5}^8 \|Iv_j\|_{Y^{\frac12+,0}} .
\end{align*}

\textbf{Case 3:} $N_4\gtrsim N$. 
In this case, we additionally have that $m(N_4)N_4\gtrsim N$. 
For $5\leq j\leq 8$, 
since $m(k_j)\langle k_j\rangle^{\frac12} \gtrsim 1$,   
by taking into account \eqref{uniformboundsofjapanesebracket}, we have
\begin{equation}
N_3^{0+}m(k_j)\langle k_j\rangle 
 \gtrsim \lambda^{0-} \langle k_j\rangle^{\frac12+} .
\end{equation}
Thus, we obtain
\begin{align*}
\int_{\R}\Lambda_8 (M_8^3;v_1,\ldots,v_8)\,dt &\lesssim 
\frac{N^{-3+}}{N_3^{0+}} \int_*\int_{**} \prod_{j=1}^4 \widehat{J_xIv_j} 
\prod_{j=5}^8 \frac{\lambda^{0+}}{\langle k_j\rangle^{\frac12+}}\widehat{J_xIv_j}\\
 &\lesssim \frac{N^{-3+}\lambda^{0+}}{N_3^{0+}} 
 \prod_{j=1}^4 \|J_xIv_j\|_{L^4_{t,x}} 
 \prod_{j=5}^8 \|J_x^{\frac12-}Iv_j\|_{L^{\infty}_{t,x}}\\
 &\lesssim \frac{N^{-3+}\lambda^{0+}}{N_3^{0+}}  
 \prod_{j=1}^4 \|J_xIv_j\|_{X^{0,\frac38}} 
  \prod_{j=5}^8 \|J_x^{\frac12-}Iv_j\|_{Y^{\frac12+,0}} \\
  &\lesssim \frac{N^{-3+}\lambda^{0+}}{N_3^{0+}}  
 \prod_{j=1}^4 \|Iv_j\|_{X^{1,\frac12}} 
  \prod_{j=5}^8 \|Iv_j\|_{Y^{1,0}} .
\end{align*}

We recall  that for the multiplier $M_8^2$ we have better bounds than for $M_8^3$ 
(see Lemma~\ref{pwestsM6} and Lemma~\ref{pwestsM63}),  
hence it is enough to consider only the latter.
\end{proof}

\begin{lemma}
\label{prop:estofM103}
Let $s\geq \frac12$ and $\delta>0$. For $M_{10}^3$ defined by \eqref{defnofM103}, 
we have the estimate
\begin{equation}
\label{MultilinearEst:M103}
\left|\int_0^{\delta}\Lambda_{10}(M^3_{10};v(t)) \,dt \right| \lesssim 
{N^{-2+}\lambda^{-1+}} \delta^{0+} \|Iv\|^{10}_{Z^1([0,\delta]\times\Tl)} .
\end{equation}
\end{lemma}
\begin{proof}
By \eqref{defnofM103} and Lemma~\ref{pwestssigma6}, we have 
$|M^3_{10}(\mathbf{k})|\lesssim 1$ and 
thus we  gain the factor $N^{-2+}$ from 
$m(N_j)N_j\gtrsim N^{1-}N_j^{0+}$, $j=1,2$. 
For additional decaying factors, 
it is enough to discuss two cases. 

\textbf{Case 1:} $N_2\gtrsim N\gg N_3$. We have
\begin{align*}
\int_{\R}\Lambda_{10} (M^3_{10};v_1,\ldots,v_{10})\,dt &\lesssim 
 \frac{N^{-2+}}{N_1^{0+}} \int_*\int_{**} (\widehat{J_xIv_1}\widehat{J_xIv_3}) (\widehat{J_xIv_2}\widehat{J_xIv_4}) 
 \prod_{j=5}^{10} \widehat{Iv_j} \\
 &\lesssim \frac{N^{-2+}}{N_1^{0+}} 
 \|(J_xIv_1)(J_xIv_3)\|_{L^2_{t,x}} \|(J_xIv_2)(J_xIv_4)\|_{L^2_{t,x}} 
  \prod_{j=5}^{10} \|Iv_j\|_{L^{\infty}_{t,x}} \\
&\lesssim  \frac{N^{-2+}\lambda^{-1+}}{N_1^{0+}} 
 \prod_{j=1}^4 \|Iv_j\|_{X^{1,\frac12}} 
  \prod_{j=5}^{10} \|Iv_j\|_{Y^{\frac12+,0}}
\end{align*}

\textbf{Case 2:} $N_3\gtrsim N$. 
In this case, 
we additionally have $m(N_3)N_3\gtrsim N$. 
Also, we use $m(k_4)\langle k_4\rangle \gtrsim 1$,  
and $m(k_j)\langle k_j\rangle^{\frac12}\gtrsim 1$ 
for $5\leq j\leq 10$. 
By using 
$1/N_1^{\varepsilon}\leq \prod_{j=5}^{10} 1/N_j^{\varepsilon/6}$, 
we get 
\begin{align*}
\int_{\R}\Lambda_{10} (M^3_{10};v_1,\ldots,v_{10})\,dt&\lesssim  
\frac{N^{-3+}}{N_1^{0+}} 
\prod_{j=1}^4 \|J_xIv_j\|_{L^4_{t,x}} 
\prod_{j=5}^{10} \|J_x^{\frac12-}Iv_j\|_{L^{\infty}_{t,x}}\\
 &\lesssim  \frac{N^{-3+}}{N_1^{0+}} 
 \prod_{j=1}^4 \|Iv_j\|_{X^{1,\frac38}}
 \prod_{j=5}^{10}\|Iv_j\|_{Y^{1,0}} .
\end{align*}
Note that in Case 2 (by discussing various subregions), 
we could provide at least an additional $\lambda^{-\frac12+}$ factor, 
but since $N^{-1+}\lesssim \lambda^{-1+}$  
and the decaying factor in Case 1 is optimal, 
we limit ourselves to the above estimate.   
\end{proof}

For the remaining terms that appear in \eqref{parttE3} 
(i.e. the ones due to the gauge transformation in the periodic setting), 
we have a decaying factor $\lambda^{-1}$ 
thanks to the coupling coefficient $\mu[v]$.  
Indeed, by  \eqref{embeddinginLtinftyHxs} 
and by using $1\leq m(k)\langle k\rangle$, we have 
$$\mu[v]=\frac{1}{2\pi\lambda} \|v\|^2_{L_t^{\infty}L^2_x} 
\lesssim \lambda^{-1} \|J_xIv\|^2_{Y^{0,0}} 
\leq \lambda^{-1} \|Iv\|^2_{Z^1}.$$

\begin{lemma} 
Let $s>0$ and $\delta>0$. For $K_4^1$ as defined by \eqref{defnofK41}, 
we have the estimate 
\begin{equation}
\label{MultilinearEst:K41}
\left| \int_0^{\delta}\Lambda_4(K_4^1 ; v(t))\,dt  \right| 
\lesssim {N^{-1+}\lambda^{-1+}} \delta^{0+} \|Iv\|^4_{X^{1,\frac12}([0,\delta]\times\Tl)} .
\end{equation}
\end{lemma}
\begin{proof}
By Lemma~\ref{pwestsK41} we have 
$|K_4^1(\mathbf{k})|\lesssim m(N_1)^2N_1^2$ for all $\mathbf{k}\in\Gamma_4(\Tl)$, and if $N_3\ll N$ then $|K_4^1(\mathbf{k})|\lesssim m(N_1)^2N_1N_3$. 
We need to discuss three cases. 

\textbf{Case 1:} $N_1\sim N_2\gtrsim N\gg N_3$. 
Due to the refined estimate, we have 
\begin{align*}
\int_{\R} \Lambda_4(K_4^1 ; v_1,\ldots,v_4)\,dt &\lesssim 
\int_*\int_{**} \frac{1}{N_1\langle k_4\rangle} \prod_{j=1}^4 \widehat{J_xIv_j}\\
&\lesssim \frac{N^{-1+}}{N_1^{0+}} \int_{\R}\int_{\Tl} 
(J_xIv_1)(J_xIv_3)(J_xIv_2)(J_xIv_4)dxdt\\
 &\lesssim \frac{N^{-1+}}{N_1^{0+}} \| (J_xIv_1) (J_xIv_3)\|_{L^2_{t,x}} 
  \| (J_xIv_2) (J_xIv_4)\|_{L^2_{t,x}}\\
&\lesssim \frac{N^{-1+}\lambda^{-1+}}{N_1^{0+}}
  \prod_{j=1}^4\|Iv_j\|_{X^{1,\frac12-}}.
\end{align*}

\textbf{Case 2:} $N_3\gtrsim N\gg N_4$. By using \eqref{obtainingN30plus}, 
we obtain  
\begin{align*}
\int_{\R} \Lambda_4(K_4^1 ; v_1,\ldots,v_4)\,dt &\lesssim 
\int_*\int_{**} \frac{1}{m(N_3)N_3 m(k_4)\langle k_4\rangle} 
  \prod_{j=1}^4 \widehat{J_xIv_j}\\
&\lesssim \frac{N^{-1+}}{N_3^{0+}} \int_{\R}\int_{\Tl} 
(J_xIv_1)(J_xIv_2)(J_xIv_3)(v_4)dxdt.
\end{align*}

\textbf{Subcase 2.1.} If $N_3\sim N_1$, then since $N_3\gg N_4$, two out of the three frequencies $k_1,k_2,k_3$ must have opposite signs, say $k_1$ and $k_2$. Thus $J_xIv_1$ and $J_xIv_2$ are separated in frequency, 
and so are 
$J_xIv_3$ and $J_xI v_4$. 
By also using $m(k_4)\langle k_4\rangle \gtrsim 1$, we have 
\begin{align}
\nonumber
\int_{\R} \Lambda_4(K_4^1 ; v_1,\ldots,v_4)\,dt &\lesssim 
\frac{N^{-1+}}{N_1^{0+}} \| (J_xIv_1) (J_xIv_2)\|_{L^2_{t,x}} \| (J_xIv_3) (J_xIv_4)\|_{L^2_{t,x}}\\
\label{eqn:estK41case21}
 &\lesssim \frac{N^{-1+}\lambda^{-1+}}{N_3^{0+}} 
  \prod_{j=1}^4 \|Iv_j\|_{X^{1,\frac12-}} .
\end{align}

\textbf{Subcase 2.2.}  If $N_3\ll N_1$, then as in Case 1, we can clearly apply the bilinear estimate \eqref{interpbilinearest} 
to the $L^2_{t,x}$-norms of both $(J_xIv_1)(J_xIv_3)$ and $(J_xIv_2)(J_xIv_4)$ and obtain the same 
bound as in \eqref{eqn:estK41case21}. 

\textbf{Case 3:} $N_4\gtrsim N$. We have 
$m(N_j)N_j\gtrsim N^{-1+}N_j^{0+}$ for $j=3,4$ 
and  thus 
\begin{align*}
\int_{\R} \Lambda_4(K_4^1 ; v_1,\ldots,v_4)\,dt &\lesssim 
\int_*\int_{**} \frac{1}{m(N_3)N_3 m(N_4)N_4} 
  \prod_{j=1}^4 \widehat{J_xIv_j}\\
&\lesssim \frac{N^{-2+}}{N_3^{0+}} \int_{\R}\int_{\Tl} 
(J_xIv_1)(J_xIv_2)(J_xIv_3)(J_xIv_4)dxdt\\
&\lesssim \frac{N^{-2+}}{N_3^{0+}} 
\prod_{j=1}^4 \|J_xIv_j\|_{L^4_{t,x}}\\
&\lesssim \frac{N^{-2+}}{N_3^{0+}} 
\prod_{j=1}^4 \|Iv_j\|_{X^{1,\frac38}} .
\end{align*}
\end{proof}

\begin{lemma}
\label{Prop:estGamma6K61} 
Let $s\geq \frac38$ and $\delta>0$. For $K_6^1$ defined by \eqref{defnofK61}, 
we have the estimate
\begin{equation}
\label{MultilinearEst:K61}
\left| \int_0^{\delta} \Lambda_{6}(K_6^1;v(t)) \,dt\right|  
\lesssim {N^{-2+}} \delta^{0+}\|Iv\|^6_{X^{1,\frac12}([0,\delta]\times\Tl)} .
\end{equation}
\end{lemma}
\begin{proof}
By Lemma~\ref{pwestsK6162}, we have
$|K_6^1(\mathbf{k})|\lesssim 1$ for all $\mathbf{k}\in\Gamma_6(\Tl)$. 
By using \eqref{symbol:lowbound}, \eqref{Sobolev2} and \eqref{L4Strichartz}, 
we estimate
\begin{align*}
 \int_{\R}\Lambda_{6}(K_6^1;v_1,\ldots,v_{6}) \,dt&\lesssim 
\int_*\int_{**} \frac{1}{(m(N_1)N_1)^2}  
  \prod_{j=1,2} \widehat{J_xIv_j} \prod_{j=3}^6 \widehat{v_j}\\
&\lesssim  \frac{N^{-2+}}{N_1^{0+}} \int_{\R}\int_{\Tl} \prod_{j=1,2} J_xIv_j
  \prod_{j=3}^6 J_x^{\frac58} Iv_j\,dxdt\\
&\lesssim  \frac{N^{-2+}}{N_1^{0+}} \prod_{j=1,2} \|J_xIv_j\|_{L^4_{t,x}} 
\prod_{j=3}^6 \|J_x^{\frac58}Iv_j\|_{L^8_{t,x}}\\
&\lesssim  \frac{N^{-2+}}{N_1^{0+}} \prod_{j=1}^6 \|Iv_j\|_{X^{1,\frac38}} .
\end{align*}
\end{proof}

\begin{lemma}
Let $s\geq\frac12$ and $\delta>0$. For $K_6^2$ defined by \eqref{defnofK62}, we have the estimate
\label{prop:gamma62}
\begin{equation}
\label{MultilinearEst:K62}
\left|\int_0^{\delta} \Lambda_{6}(K_6^2;v(t)) \,dt \right| \lesssim 
{N^{-1+}\lambda^{-1+}} \delta^{0+} \| I v\|^6_{Z^{1}([0,\delta]\times\Tl)}.
\end{equation}
\end{lemma}
\begin{proof}
By Lemma~\ref{pwestsK6162}, 
we have $|K_6^2|\lesssim m(N_1)^2N_1$.  

\textbf{Case 1:} $N_2\gtrsim N\gg N_3$. By using 
$1\lesssim m(k_j)\langle k_j\rangle $ for $j=3,4$, and 
$\frac{1}{N_1^{0+}} \lesssim m(k_j) \langle k_j\rangle^{\frac12-}$ for $j=5,6$, we estimate
\begin{align*}
\int_{\R}\Lambda_{6}(K_6^2;v_1,\ldots,v_{6})dt
&\lesssim 
\int_*\int_{**} \frac{1}{N_1} \prod_{j=1,2}\widehat{J_xIv_j}\prod_{j=3}^6 \widehat{v_j}\\
&\lesssim \frac{N^{-1+}}{N_1^{0+}} \int_{\R}\int_{\Tl} \prod_{j=1}^4 J_xIv_j \prod_{j=5,6} J_x^{\frac12-}Iv_j\,dxdt\\
&\lesssim \frac{N^{-1+}}{N_1^{0+}} \|(J_xIv_1)(J_xIv_3)\|_{L^2_{t,x}} \|(J_xIv_2)(J_xIv_4)\|_{L^2_{t,x}} 
  \prod_{j=5,6} \|J_x^{\frac12-}Iv_j\|_{L^{\infty}_{t,x}}\\
&\lesssim \frac{N^{-1+}\lambda^{-1+}}{N_1^{0+}} \prod_{j=1}^4 \|Iv_j\|_{X^{1,\frac38}}  
 \prod_{j=5,6} \|Iv_j\|_{Y^{1,0}} .
\end{align*}

\textbf{Case 2:} $N_3\gtrsim N$. We make use of $m(N_3)N_3\gtrsim N$ and 
thus we get 
\begin{align*}
\int_{\R}\Lambda_{6}(K_6^2;v_1,\ldots,v_{6})dt
&\lesssim \frac{N^{-2+}}{N_1^{0+}} \int_{\R}\int_{\Tl} \prod_{j=1}^3 J_xIv_j \prod_{j=4}^6 v_j\,dxdt\\
&\lesssim \frac{N^{-2+}}{N_1^{0+}} \prod_{j=1}^3 \|J_xIv_j\|_{L^4_{t,x}} \prod_{j=4}^6 \|v_j\|_{L^{12}_{t,x}}\\
&\lesssim 
\frac{N^{-2+}}{N_1^{0+}} \prod_{j=1}^3 \|Iv_j\|_{X^{1,\frac38}} 
 \prod_{j=4}^6 \| v_j\|_{X^{\frac{5}{12},\frac{5}{12}}} .
\end{align*}
\end{proof}

For  the next lemma, we make the following remark. 
The proof follows identically in Case 1, 
but we only have $|\widetilde{K_6^3}|\lesssim m(N_1)^2N_1^2$ in Case 2. 
By splitting the discussion into the subcases $N_3\sim N_1$ and $N_3\ll N_1$ 
as in Case 2 of the proof of Lemma~\ref{prop:estofM62}, 
we can provide at least an additional $\lambda^{-\frac12+}$ factor. 
Hence, we have:

\begin{lemma}
Let $s\geq\frac12$ and $\delta>0$. For $\widetilde{K_6^3}$ defined by \eqref{defnofK62}, we have the estimate
\begin{equation}
\label{MultilinearEst:K63tilde}
\left|\int_0^{\delta} \Lambda_{6}(\widetilde{K_6^3};v(t)) \,dt \right| \lesssim 
{N^{-1+}\lambda^{-\frac12+}} \delta^{0+} \| I v\|^6_{Z^{1}([0,\delta]\times\Tl)}.
\end{equation}
\end{lemma}

The estimates for $\int_0^{\delta}\Lambda_6(\widetilde{K_6^4})dt$ 
and  $\int_0^{\delta}\Lambda_8(\widetilde{K_8^3})dt$ 
follow identically to that of Lemma~\ref{prop:gamma62} above, 
since we have the same upper bound (see Lemma~\ref{pwestssigma4tilde} and the subsequent comment). 

\begin{lemma}
Let $s\geq \frac{5}{12}$ and $\delta>0$. For $K_8^3$ defined by \eqref{defnofK83}, we have the estimate
\label{prop:K83}
\begin{equation}
\left| \int_0^{\delta} \Lambda_{8}(K^3_8;v(t)) \,dt\right| \lesssim 
   N^{-2+} \delta^{0+}\| I v\|^8_{X^{1,\frac12}([0,\delta]\times\Tl)} .
\end{equation}
\end{lemma}
\begin{proof}
By Lemma~\ref{pwestssigma6}, we have 
$|K_8^3(\mathbf{k})|\lesssim 1$ for all $\mathbf{k}\in\Gamma_8(\Tl)$. Hence, similarly to the proof of 
Lemma~\ref{Prop:estGamma6K61}, we get
\begin{align*}
\int_{\R}\Lambda_{8}(K^3_8;v_1,\ldots,v_{8}) \,dt &\lesssim 
   \frac{N^{-2+}}{N_1^{0+}} \int_{\R}\int_{\Tl} \prod_{j=1,2} J_xIv_j  \prod_{j=3}^8 J_x^{\frac{7}{12}} Iv_j\,dxdt\\
&\lesssim  \frac{N^{-2+}}{N_1^{0+}} \prod_{j=1,2} \|J_xIv_j\|_{L^4_{t,x}} 
\prod_{j=3}^8 \|J_x^{\frac{7}{12}}Iv_j\|_{L^{12}_{t,x}}\\
&\lesssim  \frac{N^{-2+}}{N_1^{0+}} \prod_{j=1}^8 \|Iv_j\|_{X^{1,\frac{1}{2}}} .
\end{align*}
\end{proof}

We put all the results of this section together in the following:

\begin{proposition}
\label{prop:slowlyvaryingincremsofE3}
Let $s\geq\frac12$ and $\delta>0$. Suppose $v$ is a solution to \eqref{g1DNLS} on $[0,\delta]$. 
For $\mathcal{E}^3$ defined by \eqref{defnofE3}, 
we have 
\begin{equation}
\left| \mathcal{E}^3[v(\delta)] - \mathcal{E}^3[v(0)] \right| \leq 
   N^{-\frac32+}\lambda^{-1+} \delta^{0+} P(\|Iv\|_{Z^1([0,\delta]\times\Tl)}), 
\end{equation}
for some polynomial $P$ with non-negative coefficients. 
\end{proposition}

\section{Control of the almost conserved energy and of the almost conserved momentum}
\label{Sect:AlmostconservedEandP}
In this section we show that 
$\mathcal{E}[Iv(t)]$ stays close to $\mathcal{E}^3[v(t)]$ (which is very slowly varying in time)   
and that $\mathcal{P}[Iv(t)]$ stays close to $\mathcal{P}[v(t)]=\mathcal{P}[v_0]$, at any time $t$.  
For the sake of efficiency, we adopt in the proofs below the reduction remarks  from the previous section.

\begin{lemma}
\label{lem:sigma4}
For $\sigma_4$ defined by \eqref{defnofsigma4}, we have 
\begin{equation}
\left| \Lambda_4(\sigma_4;f) \right| \lesssim N^{-1+}  \|If\|^4_{H^1(\Tl)} .
\end{equation}
\end{lemma}
\begin{proof}
By Lemma~\ref{pwestssigma4}, 
we have $|\sigma_4(\mathbf{k})|\lesssim m(N_1)^2N_1$ for all $k\in\Gamma_4(\Tl)$. 
Then, by H\"{o}lder and Sobolev inequalities, 
and using $\frac{1}{N_1^{0+}}\lesssim  m(k_j)\langle k_j\rangle^{\frac12-}$ for $j=3,4$,  
we have
\begin{align*}
\Lambda_4(\sigma_4;f_1,\ldots, f_4) &\lesssim  
\int_{**} \frac{m(N_1)^2N_1}{\prod_{j=1}^4 m(k_j)\langle k_j\rangle } \prod_{j=1}^4 \widehat{J_xIf_j} \\
&\lesssim  \frac{1}{N_1} \int_{\Tl} (J_xIf_1) (J_xIf_2) (J_x^{\frac12-}If_3) (J_x^{\frac12-}If_4)dx\\
 &\lesssim  \frac{N^{-1+}}{N_1^{0+}} \|J_xIf_1\|_{L^2_x} \|J_xIf_2\|_{L^2_x} \|J_x^{\frac12-}If_3\|_{L_x^{\infty}} 
  \|J_x^{\frac12-}If_4\|_{L_x^{\infty}}\\
  &\lesssim \frac{ N^{-1+}}{N_1^{0+}} \prod_{j=1}^4 \|If_j\|_{H^1_x} .
\end{align*}
\end{proof}

The estimate for $\Lambda_4(\widetilde{\sigma_4}; f)$ follows similarly 
since, by Lemma~\ref{pwestssigma4tilde} (i),  
we have the same pointwise bound, that is  $|\widetilde{\sigma_4}(\mathbf{k})|\lesssim m(N_1)^2N_1$. 

\begin{lemma}
For $\sigma_6$ defined by \eqref{defnofsigma6}, we have 
\begin{equation}
\left| \Lambda_6(\sigma_6;f)\right| \lesssim N^{-2+}  \|If\|^6_{H^1(\Tl)} .
\end{equation}
\end{lemma}
\begin{proof}
By Lemma~\ref{pwestssigma6}, 
we have $|\sigma_6(\mathbf{k})|\lesssim 1$ for all $k\in\Gamma_6(\Tl)$. 
Similarly to the proof of Lemma~\ref{lem:sigma4}, we have 
\begin{align*}
\Lambda_6(\sigma_6;f_1,\ldots, f_4) 
 &\lesssim  \frac{N^{-2+}}{N_1^{0+}} \int_{\Tl} (J_xIf_1) (J_xIf_2) \prod_{j=3}^6(J_x^{\frac12-}If_j) dx\\
 &\lesssim \frac{ N^{-2+}}{N_1^{0+}} \prod_{j=1}^6 \|If_j\|_{H^1_x} .
\end{align*}
\end{proof}

Hence, we proved that all the correction terms are small, and thus we obtain:

\begin{proposition}
\label{prop:mathcalE3isclosetomathcalEI}
For $\mathcal{E}$ and $\mathcal{E}^3$ defined by \eqref{defnofmathcalE} and \eqref{defnofE3}, 
we have
\begin{equation}
\label{eqn:AlmostConservationEstar}
\left| \mathcal{E}[If] -\mathcal{E}^3[f] \right| \lesssim 
   N^{-1+} \left(\|If\|^4_{H_x^1(\Tl)} + \|If\|^6_{H_x^1(\Tl)}\right) , 
\end{equation}
for all $f\in H_x^s(\Tl)$.  
\end{proposition}

Next, we turn to the analysis of $\mathcal{P}[I(\cdot)]$ for which, as in \cite{GuoWu}, we prove:

\begin{proposition}
\label{prop:mathcalPIstaysclosetoP}
Let $s\geq \frac12$. For $\mathcal{P}$ defined by \eqref{defnofmathcalP}, we have 
\begin{equation}
\left|\mathcal{P}[If] - \mathcal{P}[f] \right| \lesssim N^{-1} \left(\|If\|_{H^1(\Tl)}^2 +\|If\|_{H^1(\Tl)}^4 \right), 
\end{equation}
for all $f\in H_x^s(\Tl)$. 
\end{proposition}
\begin{proof}
We have 
\begin{equation}
\label{separatingintotwo}
\left| \mathcal{P}[If] - \mathcal{P}[f]  \right| \leq  
\left|\Im \int_{\Tl} \left(If\partial_x(\overline{If}) -f\partial_x\overline{f}\right)dx\right| 
   + \frac12\left| \int_{\Tl} \left( |If|^4 - |f|^4\right)dx\right|
\end{equation}
and we can estimate the two terms separately. 

First, using integration by parts, we write
\begin{align*}
\Im \int_{\Tl} \left(If\partial_x(\overline{If}) -f\partial_x\overline{f}\right)dx  &= 
 \Im \int_{\Tl} If \partial_x\left(\overline{If} -\overline{f}\right) dx + \Im\int_{\Tl} \partial_x \overline{f} \left( If-f \right)dx\\
 &=\Im \int_{\Tl} If \partial_x\left(\overline{If} -\overline{f}\right) dx 
    +\Im\int_{\Tl}  f \partial_x\left( \overline{If-f} \right)dx\\
 &= \Im\int_{\Tl}  (If+f) \partial_x\left( \overline{If-f} \right)dx .
\end{align*}
Notice that $I-\textup{Id}=P_{\textup{hi}}(I-\textup{Id})$, where $\textup{Id}$ 
is the identity operator and we take $P_{\textup{hi}}:=P_{\gtrsim N}$. Thus, 
by commuting Fourier multiplier operators, using the self-adjointness of Littlewood-Paley operators 
and duality properties of Sobolev norms,  
we have 
\begin{align*}
\left| \Im \int_{\Tl} (If+f) \partial_x\left( \overline{If-f} \right)dx \right| &\leq  
  \left| \langle P_{\textup{hi}}(If+f), (I-\textup{Id})\partial_x\overline{f}\rangle_{L^2(\Tl)} \right|\\
 &\leq \|P_{\textup{hi}}(If+f) \|_{{H}^{\frac12}} \|P_{\textup{hi}}(I-\textup{Id})\partial_xf\|_{{H}^{-\frac12}} \\
 &\leq \left(\|P_{\textup{hi}}If \|_{{H}^{\frac12}}  + \|P_{\textup{hi}}f \|_{{H}^{\frac12}} \right)^2 .
\end{align*}
Since 
$1\lesssim N^{-\frac12}\langle k\rangle^{\frac12}\ ,\ 
1\lesssim N^{-\frac12}m(k)\langle k\rangle^{\frac12}$ 
for all $|k|\gtrsim N$, 
we have
\begin{align}
\label{bds:momentumIdiff}
\|P_{\textup{hi}}If \|_{{H}^{\frac12}} &\lesssim  N^{-\frac12} \|If \|_{H^{1}} ,\\
\label{bds:momentumIdiff2}
\|P_{\textup{hi}}f \|_{{H}^{\frac12}} &\lesssim  N^{-\frac12} \|If \|_{H^{1}} .
\end{align}
Thus the first term in the right hand side of \eqref{separatingintotwo} is bounded by 
$ N^{-1} \|If \|_{H^{1}}^2$.

For the second term in the right hand side of \eqref{separatingintotwo}, we write
$$|If|^4 -|f|^4 = |If|^2If(\overline{If}-\overline{f}) + |If|^2(If-f)\overline{f} + If(\overline{If}-\overline{f})|f|^2 
 + (If-f)\overline{f}|f|^2$$
and we treat, for example, the second term (modulo complex conjugation, it has all three possible factors involved); 
the others can be argued for analogously. 
By H\"{o}lder's inequality, 
\begin{align*}
\left| \int_{\Tl} |If|^2\overline{f} P_{\textup{hi}}(I-\textup{Id})f \,dx\right| &\leq 
  \left| \langle (I-\textup{Id})f , P_{\textup{hi}}(|If|^2 f)\rangle_{L^2}\right|  \\
&\lesssim \|(I-\textup{Id})f\|_{L^6} \|P_{\textup{hi}}(|If|^2 f)\|_{L^{\frac65}}  .
\end{align*}
Then, by Sobolev embedding, 
\begin{equation}
\label{L6Sobolevembforthesecondterm}
\|(I-\textup{Id})f\|_{L^6} \lesssim \|P_{\textup{hi}}(I-\textup{Id})f\|_{H^{\frac13}} \leq 
 \|P_{\textup{hi}}If \|_{H^{\frac12}} + \|P_{\textup{hi}}f \|_{H^{\frac12}}
\end{equation}
and we can use the estimates \eqref{bds:momentumIdiff}-\eqref{bds:momentumIdiff2}
to gain a factor of $N^{-\frac12}$. 
Another decaying factor is obtained via a Bernstein estimate, 
and then by Leibniz and H\"{o}lder inequalities, we get
\begin{align}
\notag
\|P_{\textup{hi}}(|If|^2 f)\|_{L^{\frac65}} &\lesssim N^{-\frac12} \| J_x^{\frac12}P_{\textup{hi}} (|If|^2 {f})\|_{L^{\frac65}}\\ 
\label{jkgfsflkjsdfkl}
& \lesssim N^{-\frac12} \left( \|J_x^{\frac12}If\|_{L^2}\|If\|_{L^6} \|f\|_{L^6} +  \|J_x^{\frac12}f\|_{L^2}\|If\|_{L^6}^2\right)\\
\notag
& \lesssim  N^{-\frac12}  \|If\|^3_{H^1},
\end{align}
where in the last step we used the Sobolev embedding as in \eqref{L6Sobolevembforthesecondterm} and 
$\| If\|_{H^{\frac12}} \lesssim \|f\|_{H^{\frac12}} \lesssim \|If\|_{H^1}$. 
Notice that if we do not drop the frequency restriction when passing to \eqref{jkgfsflkjsdfkl}, 
at least one factor (in both terms) has to be supported on frequencies $\gtrsim N$, 
hence by arguing as for \eqref{bds:momentumIdiff}, we could get another factor of $N^{-\frac12}$. 
Therefore, we obtain that the second term of  \eqref{separatingintotwo} is bounded by $N^{-\frac32}\|If\|^4_{H^1}$.
\end{proof}

\section{Proof of Proposition~\ref{prop:GWPofg1DNLS} via the $I$-method}
\label{sect:proofofGWPg1DNLS}

In order to prove that blowup of the $H^{1/2}$-norm of a solution $v$  
of \eqref{g1DNLS} does not occur in finite time,  
we adapt the $I$-method of \cite{CKSTT,CKSTTrefined} 
(therein also referred to as  ``the almost conserved energy method'') 
to also incorporate the almost conservation of $\mathcal{P}[Iv]$.

For initial data $v_0\in \mathcal{H}^s(\T):=\{ f\in H^s(\T) : M[f] <4\pi \}$,  $s<1$, 
its energy $E[v_0]$ might not even be defined. 
However, 
the functionals $\mathcal{E}[Iv(t)]$ and $\mathcal{P}[Iv(t)]$ 
are well-defined and  via Lemma~\ref{lem:controlofdotH1wEandP} provide
\begin{equation}
\label{controlofIv}
\|Iv(t)\|^2_{H^1} \lesssim |\mathcal{E}[Iv(t)]|+ \mathcal{P}[Iv(t)]^2 + 1,
\end{equation} 
where the smoothing operator $I$ is defined by \eqref{defnofIoperator} in Section~\ref{subsect:theIoperator} 
and $v$ is a (local) solution of \eqref{g1DNLS} with $v(0)=v_0$. 
This control allows us to iterate the local well-posedness theory for any initial data in $\mathcal{H}^s(\T)$ 
and prove that the corresponding solution exists for arbitrarily large times. 

Since \eqref{smoothingpropofIhomog} allows for $\|Iv_0\|_{\dot{H}^1}\sim N^{1-s}$,   
which would give 
a time of existence $\delta\downarrow 0$  as $N\uparrow\infty$, 
we use the  scaling transformation  \eqref{naturalscaling} and notice that 
\begin{equation}
\|Iv_0^{\lambda}\|_{\dot{H}^1(\Tl)} \lesssim N^{1-s}\lambda^{-s} \|v_0\|_{\dot{H}^s(\T)} .
\end{equation}
Therefore, we choose the scaling parameter 
\begin{equation}
\label{choiceoflambda}
\lambda = N^{\frac{1-s}{s}}
\end{equation} 
which ensures that $\delta\gtrsim 1$, uniformly in $N$ and $\lambda$. 
We then have $1\ll \lambda \leq N$ in the regularity range $\frac12\leq s<1$, 
(in particular, $\lambda=N$ for $s=\frac12$). 
We also record that 
$\|v_0^{\lambda}\|_{H^s(\Tl)}$, $P[Iv_0^{\lambda}]$, $E[Iv_0^{\lambda}]$ are bounded 
by constants depending only on $\|v_0\|_{H^s(\T)}$.

A slightly modified iteration argument concludes the proof of 
Proposition~\ref{prop:GWPofg1DNLS}. Indeed, consider $B>0$ such that
$$B^2\sim \|v_0\|^2_{H^s(\T)} + |\mathcal{E}[Iv_0]|+\mathcal{P}[Iv_0]^2+1 $$ 
and suppose that at step $j$, we have 
$$\|Iv^{\lambda}(j\delta)\|_{H^1(\Tl)}\leq B.$$ 
Then, by Proposition~\ref{prop:LWPIsyst}, 
$$\|Iv^{\lambda}\|_{Z^1([j\delta,j\delta+\delta]\times\Tl)}\leq D$$
and according to Proposition~\ref{prop:slowlyvaryingincremsofE3}, 
$$|\mathcal{E}^3[v^{\lambda}(j\delta+\delta)]| \leq 
  |\mathcal{E}^3[v^{\lambda}(j\delta)]| + \delta^{0+} N^{-\gamma}\lambda^{-\kappa} C_1(D)$$
with $\gamma=\frac32-$, $\kappa=1-$. 
Assuming that we run this iteration $J$ times so that we cover the scaled time interval $[0,\lambda^2T]$, 
i.e.  we choose 
\begin{equation}
\label{bdonJ:frombelow}
J \gtrsim \lambda^2 T , 
\end{equation}
we have 
$$|\mathcal{E}^3[v^{\lambda}(J\delta)]| \leq 
 |\mathcal{E}^3[v^{\lambda}(0)] | + J \delta^{0+} N^{-\gamma}\lambda^{-\kappa} C_1(D). $$
Notice that $|\mathcal{E}^3[v(t)]|$ stays bounded 
(e.g. $|\mathcal{E}^3[v(t)]|\leq 2|\mathcal{E}^3[v^{\lambda}(0)] |$) over the entire $[0,\lambda^2T]$ 
if we further impose that $N$ is chosen such that 
\begin{equation}
\label{bdonJ:fromabove}
J \lesssim N^{\gamma}\lambda^{\kappa} .
\end{equation}
At each iteration step, due to Proposition~\ref{prop:mathcalE3isclosetomathcalEI} 
and Proposition~\ref{prop:mathcalPIstaysclosetoP}, 
we have in particular that
\begin{align}
\label{EIjplus1deltaisbounded}
|\mathcal{E}[Iv^{\lambda}((j+1)\delta)]| &\leq  2|\mathcal{E}^3[v^{\lambda}(0)]| + N^{-1+} C_2(D), \\
\label{PIjplus1deltaisbounded}
|\mathcal{P}[Iv^{\lambda}((j+1)\delta)]| &\leq  |\mathcal{P}[v^{\lambda}(0)]| + N^{-1} C_3(D), 
\end{align}
where we used the time interval restricted to $[j\delta, (j+1)\delta]$ version of \eqref{embeddinginLtinftyHxs} to get
$$\|Iv^{\lambda}((j+1)\delta)\|_{H^1(\Tl)}\lesssim D .$$ 
We choose $N$ large enough so that the second terms are dominated by the corresponding first terms in 
\eqref{EIjplus1deltaisbounded}, \eqref{PIjplus1deltaisbounded}. 
By Lemma~\ref{lem:controlofdotH1wEandP}, 
 we then deduce 
 $$\|Iv^{\lambda}((j+1)\delta)\|_{H^1(\Tl)}\leq B$$
and thus we get to perform the iteration again. 

Note that \eqref{bdonJ:frombelow}, \eqref{bdonJ:fromabove} and $s\geq\frac12$ yield
$$T\lesssim N^{\gamma-(\kappa-2)+\frac{1}{s}(\kappa-2)}\lesssim N^{\gamma+\kappa-2}.$$
In our case, $\gamma+\kappa-2=\frac12-$; hence, given any large $T$, 
we can choose a frequency threshold $N=N(T)\gg1$ for the $I$-operator. 

Notice that for all $t\in[0,\lambda^2T]\subset [0,J\delta]$, 
we have
$\mathcal{E}[Iv^{\lambda}(t)]\lesssim \mathcal{E}[Iv^{\lambda}_0]\lesssim 1$ and  
$\mathcal{P}[Iv^{\lambda}(t)]\lesssim \mathcal{P}[Iv^{\lambda}_0]\lesssim 1$, thus 
$\|Iv^{\lambda}(t)\|_{H^1(\Tl)}\lesssim 1$. 
Also, 
we recall that we owe to undo the scaling:
\begin{equation*}
\|v(t)\|_{{H}^s(\T)} \lesssim \lambda^s \|v^{\lambda}(\lambda^2 t)\|_{H^s(\Tl)} \lesssim 
\lambda^s \|Iv^{\lambda}(t)\|_{H^1(\Tl)} \lesssim N^{1-s}, 
\end{equation*}
for all $t\in[0,T]$, 
where we used \eqref{smoothingpropofI} and \eqref{choiceoflambda}. 
The above numerology allows us to take $N\sim T^{2+}$ and thus
$$\sup_{t\in[0,T]} \|v(t)\|_{{H}^s(\T)} \lesssim T^{2-2s+}$$
for any $\frac12\leq s<1$.

\appendix
\section{Mild ill-posedness below $H^{1/2}(\T)$}
\label{appdx:illposedness}

The scope of this section is to provide the analogue to the periodic setting 
of the ill-posedness result of Biagioni and Linares \cite{BiagioniLinares} 
where it was shown that 
the flow map $u_0\mapsto u$ of \eqref{DNLS} is 
not uniformly continuous from bounded subsets of $H^s(\R)$ into $C_tH_x^s([-T,T]\times\R)$, 
for any $T>0$ and $0\leq s<\frac12$. 
However, 
the method (which was introduced in \cite{KPVduke01}) 
uses the family of soliton solutions of \eqref{DNLS} on $\R$ 
(see \cite{KaupNewell, vanSaarloosHohenberg}), 
for which the corresponding initial data are not compactly supported, 
hence this strategy cannot be adapted to the periodic setting.

We also recall an observation of Gr\"{u}nrock and Herr \cite[Remark~2]{GrunrockHerr}, 
that in the periodic setting, 
due to the presence of a translation in the gauge transformation 
$\mathcal{G}^{\beta}$ (see \eqref{gaugetransfwithtranslation}), 
at any regularity level, 
the uniform continuity of the solution map of \eqref{DNLS} fails 
without fixing the mass on bounded subsets of $H^s(\T)$ 
(see also \cite[Theorem~3.1.1.(ii)]{HerrThesis}). 
Nevertheless, for the gauge equivalent equation \eqref{g1DNLS} 
one does not face the uniform continuity bottleneck due to the translation operator 
and it was for this equation that the contraction mapping argument was applied in \cite{HerrIMRN06}.

Using ideas similar to those in 
\cite{BurqGerardTzvetkovmrl02,CCTajm03},  
we construct smooth solutions 
that prove the failure of uniform continuity of the solution map of  \eqref{g1DNLS}
on bounded subsets of $H^s(\T)$, for $0\leq s<\frac12$. 
Since the solutions we construct are supported on single frequencies (monochromatic waves), 
the same result holds true for the Fourier-Lebesgue spaces $\mathcal{F}L^{s,r}(\T)$ for $s<\frac12$, $b\in\R$. 

\begin{lemma}
\label{prop:notuc}
Let $0\leq s<\frac12$ and $T>0$. 
For any $0<\delta\ll \varepsilon < 1$, there exist 
smooth initial data $v_0, \widetilde{v}_0$ such that 
\begin{align}
\|v_0\|_{H^s(\T)},\, \|\widetilde{v}_0\|_{H^s(\T)} \lesssim \varepsilon ,&\\
\|v_0-\widetilde{v}_0\|_{H^s(\T)} \lesssim \delta, &
\end{align}
and for which the corresponding solutions $v,\widetilde{v}$ to \eqref{g1DNLS} have the property
\begin{align}
\|v-\widetilde{v}\|_{L_t^{\infty}([-T,T];H_x^s(\T))}\gtrsim \varepsilon . 
\end{align}
 \end{lemma}
 \begin{proof}
 Let $a\in\C$ and $N\in \Z$, $N\gg 1$ (to be chosen later) and consider 
 functions supported on a single frequency of the form 
 $$v_{N_,a}(t,x) = ae^{i(Nx+\theta(N)t)},$$
 for some $\R$-valued $\theta(\cdot)$. We have 
 $$\mu[v_{N,a}] = |a|^2\ ,\ \psi[v_{N,a}]=-2|a|^2N+\frac12|a|^4 $$
 and thus we compute the corresponding nonlinearity of \eqref{g1DNLS}:
 $$\mathcal{N}(v_{N,a})= 
  |a|^2aNe^{i(Nx+\theta(N)t)}.$$
Then,  by taking $\theta(N)=-N^2-|a|^2N$, the function 
 \begin{equation}
 v_{N,a}(t,x) = ae^{i(Nx -N^2t - |a|^2Nt)}
 \end{equation}
 is a solution of \eqref{g1DNLS} with 
 $$\|v_{N,a}(t,x)\|_{L_x^2(\T)}\sim |a|\ ,\ \|v_{N,a}(t,x)\|_{\dot{H}_x^s(\T)}\sim |a| N^s$$
 and since $s\geq0$, we also have 
 $$\|v_{N,a}(t,x)\|_{H_x^s(\T)}\sim |a| N^s.$$
 Now let $a= b N^{-s}$ and $\widetilde{a} =\widetilde{b} N^{-s}$ with $b,\widetilde{b}\in\C$ 
 such that $|b|\sim |\widetilde{b}|\sim \varepsilon$ and $|b-\widetilde{b}|\lesssim \delta$. 
 We find
 $$\|v_{N,a}(0,x) - v_{N,\tilde{a}}(0,x)\|_{H_x^s(\T)} 
 =  |b-\tilde{b}|N^{-s} \|e^{iNx}\|_{H_x^s(\T)} \lesssim \delta .$$
On the other hand, by setting $\varphi(N,b):=|bN^{-s}|^2N$ to simplify the writing, we obtain
 \begin{align*}
 \|v_{N,a}(t,x) - v_{N,\tilde{a}}(t,x)\|_{H_x^s(\T)} &=
  \left| b e^{-i \varphi(N,b)t} - \tilde{b} e^{-i\varphi(N,\tilde{b})t}\right| N^{-s}  \|e^{iNx}\|_{H_x^s(\T)}\\
  &\gtrsim |b| \left|e^{-i\varphi(N,b)t}-e^{-i\varphi(N,\tilde{b})t}\right| - |b-\tilde{b}|\\
  &\gtrsim \varepsilon \left|e^{i(\varphi(N,b)-\varphi(N,\tilde{b}))t}-1\right| -\delta .
\end{align*}
Note that
 \begin{equation}
 \label{illposed:phasediff}
 \varphi(N,b)-\varphi(N,\tilde{b}) = N^{1-2s}(|b|^2-|\tilde{b}|^2)
 \end{equation}
and that at $t=t_N$, where 
\begin{equation}
t_N:=\frac{\pi}{\varphi(N,b)-\varphi(N,\tilde{b})},
\end{equation}
the two solutions have opposite phases, and thus 
$$ \|v_{N,a}(t_N,x) - v_{N,\tilde{a}}(t_N,x)\|_{H_x^s(\T)}  \gtrsim\varepsilon-\delta \sim \varepsilon.$$
Indeed, since the power of $N$ in \eqref{illposed:phasediff} is positive, 
we can 
choose an integer $N=N(\varepsilon,T)$ (independent of $\delta$) such that 
$|t_N|\leq T/2$, or equivalently 
$$| \varphi(N,b)-\varphi(N,\tilde{b})  | \gtrsim T^{-1} .$$
 \end{proof}

\begin{remark}
One can easily adapt the above argument to any other gauge equivalent equation, including \eqref{DNLS} itself. 
Indeed, it is enough to take 
$$\theta_{\beta}(N)=\theta(N) -(\beta^2-\frac32\beta+\frac12)|a|^4.$$
Correspondingly, we take
$$ \varphi_{\beta}(N,b)= \varphi(N,b)-(\beta^2-\frac32\beta+\frac12)|b|^4N^{-4s}$$
and note that for $N\gg 1$, the difference in phase is essentially as above, i.e.
$$\varphi_{\beta}(N,b)-\varphi_{\beta}(N,\tilde{b}) \sim \varphi(N,b)-\varphi(N,\tilde{b}).$$
\end{remark}

\bibliographystyle{siam}
\bibliography{/Users/razvanmosincat/Dropbox/Razvan_bibfile}

\begin{thebibliography}{10}

\bibitem{Agueh06}
{\sc M.~Agueh}, {\em Sharp {G}agliardo--{N}irenberg inequalities and mass
  transport theory}, J. Dyn. Differ. Equ., 18 (2006), pp.~1069--1093.

\bibitem{BiagioniLinares}
{\sc H.~Biagioni and F.~Linares}, {\em Ill-posedness for the derivative
  {S}chr{\"o}dinger and generalized {B}enjamin-{O}no equations}, Trans. Amer.
  Math. Soc., 353 (2001), pp.~3649--3659.

\bibitem{BourgainGAFA93}
{\sc J.~Bourgain}, {\em Fourier transform restriction phenomena for certain
  lattice subsets and applications to nonlinear evolution equations}, Geom.
  Func. Anal.,  (1993).

\bibitem{BurqGerardTzvetkovmrl02}
{\sc N.~Burq, P.~G{\'e}rard, and N.~Tzvetkov}, {\em An instability property of
  the nonlinear {S}chr{\"{o}}dinger equation on ${S}^{d}$}, Math. Res. Lett., 9
  (2002), pp.~323--336.

\bibitem{CherSimpsonSulem}
{\sc Y.~Cher, G.~Simpson, and C.~Sulem}, {\em Local structure of singular
  profiles for a derivative nonlinear {S}chr{\"{o}}dinger equation}, arXiv
  preprint arXiv:1602.02381,  (2016).

\bibitem{CCTajm03}
{\sc M.~Christ, J.~Colliander, and T.~Tao}, {\em Asymptotics, frequency
  modulation, and low regularity ill-posedness for canonical defocusing
  equations}, Amer. J. Math.,  (2003), pp.~1235--1293.

\bibitem{CKSTT}
{\sc J.~Colliander, M.~Keel, G.~Staffilani, H.~Takaoka, and T.~Tao}, {\em
  Global well-posedness for {S}chr{\"o}dinger equations with derivative}, SIAM
  J. Math. Anal., 33 (2001), pp.~649--669.

\bibitem{CKSTTrefined}
\leavevmode\vrule height 2pt depth -1.6pt width 23pt, {\em A refined global
  well-posedness for {S}chr{\"o}dinger equations with derivative}, SIAM J.
  Math. Anal., 34 (2002), pp.~64--86.

\bibitem{CKSTTjfa04}
\leavevmode\vrule height 2pt depth -1.6pt width 23pt, {\em Multilinear
  estimates for periodic {K}d{V} equations, and applications}, J. Funct. An.,
  211 (2004), pp.~173--218.

\bibitem{deSilva2007}
{\sc D.~{De Silva}, N.~Pavlovic, G.~Staffilani, and N.~Tzirakis}, {\em Global
  well-posedness for a periodic nonlinear {S}chr{\"o}dinger equation in 1{D}
  and 2{D}}, Discrete Contin. Dyn. Syst., 19 (2007), pp.~37--65.

\bibitem{GrunrockIMRN}
{\sc A.~Gr{\"{u}}nrock}, {\em Bi-and trilinear {S}chr{\"{o}}dinger estimates in
  one space dimension with applications to cubic nls and dnls}, Int. Math. Res.
  Not., 2005 (2005), pp.~2525--2558.

\bibitem{GrunrockHerr}
{\sc A.~Grunrock and S.~Herr}, {\em Low regularity local well-posedness of the
  derivative nonlinear {S}chr{\"o}dinger equation with periodic initial data},
  SIAM J. Math. Anal., 39 (2008), pp.~1890--1920.

\bibitem{GuoWu}
{\sc Z.~Guo and Y.~Wu}, {\em Global well-posedness for the derivative nonlinear
  {S}chr{\"{o}}dinger equation in ${H}^{\frac12}(\mathbb{R})$},
  arXiv:1606.07566,  (2016).

\bibitem{Hayashi93}
{\sc N.~Hayashi}, {\em The initial value problem for the derivative nonlinear
  {S}chr{\"o}dinger equation in the energy space}, Nonlinear Anal., 20 (1993),
  pp.~823--833.

\bibitem{HayashiOzawa}
{\sc N.~Hayashi and T.~Ozawa}, {\em On the derivative nonlinear
  {S}chr{\"o}dinger equation}, Physica D., 55 (1992), pp.~14--36.

\bibitem{HayashiOzawaSIAM}
\leavevmode\vrule height 2pt depth -1.6pt width 23pt, {\em Finite energy
  solutions of nonlinear {S}chr{\"o}dinger equations of derivative type}, SIAM
  J. Math. Anal., 25 (1994).

\bibitem{HerrIMRN06}
{\sc S.~Herr}, {\em On the {C}auchy problem for the derivative nonlinear
  {S}chr{\"o}dinger equation with periodic boundary condition}, Int. Math. Res.
  Not.,  (2006).

\bibitem{HerrThesis}
\leavevmode\vrule height 2pt depth -1.6pt width 23pt, {\em Well-posedness
  results for dispersive equations with derivative nonlinearities}, PhD thesis,
  Univeristy of Dortmund, 2006.

\bibitem{KaupNewell}
{\sc D.~Kaup and A.~Newell}, {\em An exact solution for a derivative nonlinear
  {S}chr{\"o}dinger equation}, J. Math. Phys., 19 (1978), pp.~789--801.

\bibitem{KPVjams96}
{\sc C.~Kenig, G.~Ponce, and L.~Vega}, {\em A bilinear estimate with
  applications to the {K}d{V} equation}, J. Amer. Math. Soc., 9 (1996),
  pp.~573--603.

\bibitem{KPVduke01}
{\sc C.~E. Kenig, G.~Ponce, and L.~Vega}, {\em On the ill-posedness of some
  canonical dispersive equations}, Duke Math. J., 106 (2001), pp.~617--633.

\bibitem{KwonWu}
{\sc S.~Kwon and Y.~Wu}, {\em Orbital stability of solitary waves for
  derivative nonlinear {S}chr{\"{o}}dinger equation}, arXiv preprint
  arXiv:1603.03745,  (2016).

\bibitem{LRS88}
{\sc J.~Lebowitz, H.~Rose, and E.~Speer}, {\em Statistical mechanics of the
  nonlinear {S}chr{\"o}dinger equation}, J. Stat. Phys., 50 (1988),
  pp.~657--687.

\bibitem{LeeThesis}
{\sc J.~Lee}, {\em Analytic properties of {Z}akharov-{S}habat inverse
  scattering problem with a polynomial spectral dependence of degree 1 in the
  potential}, PhD thesis, Yale University, 1983.

\bibitem{Lee89}
{\sc J.~Lee}, {\em Global solvability of the derivative nonlinear
  {S}chr{\"o}dinger equation}, Trans. Amer. Math. Soc., 314 (1989),
  pp.~107--118.

\bibitem{LiuPerrySulem15}
{\sc J.~Liu, P.~Perry, and C.~Sulem}, {\em Global existence for the derivative
  nonlinear {S}chrodinger equation by the method of inverse scattering}, arXiv
  preprint arXiv:1511.01173,  (2015).

\bibitem{LiuSimpsonSulem}
{\sc X.~Liu, G.~Simpson, and C.~Sulem}, {\em Focusing singularity in a
  derivative nonlinear {S}chr{\"o}dinger equation}, Physica D: Nonlinear
  Phenomena, 262 (2013), pp.~48--58.

\bibitem{MiaoWuXu}
{\sc C.~Miao, Y.~Wu, and G.~Xu}, {\em Global well-posedness for
  {S}chr{\"o}dinger equation with derivative in {$H^{1/2}(\mathbb{R})$}}, J.
  Differential Equations, 251 (2011), pp.~2164--2195.

\bibitem{Mio76}
{\sc K.~Mio, T.~Ogino, K.~Minami, and S.~Takeda}, {\em Modified nonlinear
  {S}chr{\"o}dinger equation for {A}lfv{\'e}n waves propagating along the
  magnetic field in cold plasmas}, J. Phys. Soc. Japan, 41 (1976),
  pp.~265--271.

\bibitem{Mjolhus76}
{\sc E.~Mjolhus}, {\em On the modulational instability of hydromagnetic waves
  parallel to the magnetic field}, J. Plasma Physics, 16 (1976), pp.~321--334.

\bibitem{MosincatOh2015}
{\sc R.~Mosincat and T.~Oh}, {\em A remark on global well-posedness of the
  derivative nonlinear {S}chr{\"o}dinger equation on the circle}, C.R. Acad.
  Sci. Paris, Ser. I, 353 (2015), pp.~837--841.

\bibitem{NORS}
{\sc A.~R. Nahmod, T.~Oh, L.~Rey-Bellet, and G.~Staffilani}, {\em Invariant
  weighted {W}iener measures and almost sure global well-posedness for the
  periodic derivative {N}{L}{S}}, J. Eur. Math. Soc.,  (2010).

\bibitem{PelinovskyShimabukuro16}
{\sc D.~Pelinovsky and Y.~Shimabukuro}, {\em Existence of global solutions to
  the derivative {N}{L}{S} equation with the inverse scattering transform
  method}, arXiv preprint arXiv:1602.02118,  (2016).

\bibitem{Rogister71}
{\sc A.~Rogister}, {\em Parallel propagation of nonlinear low-frequency waves
  in high-$\beta$ plasma}, Phys. Fluids, 14 (1971).

\bibitem{TakaokaADE}
{\sc H.~Takaoka}, {\em Well-posedness for the one dimensional {S}chr{\"o}dinger
  equation with the derivative nonlinearity}, Adv. Diff. Eqs., 4 (1999),
  pp.~561--680.

\bibitem{TakaokaEJDE}
\leavevmode\vrule height 2pt depth -1.6pt width 23pt, {\em Global
  well-posedness for {S}chr{\"o}dinger equations with derivative in a nonlinear
  term and data in low-order {S}obolev spaces}, Electron. J. Differential
  Equations, 42 (2001).

\bibitem{Takaoka2015}
\leavevmode\vrule height 2pt depth -1.6pt width 23pt, {\em A priori estimates
  and weak solutions for the derivative nonlinear {S}chr{\"o}dinger equation on
  torus below {$H^{\frac12}$}}, arXiv:1508.03076,  (2015).

\bibitem{Takaoka2016}
\leavevmode\vrule height 2pt depth -1.6pt width 23pt, {\em Energy transfer
  model for the derivative nonlinear {S}chr{\"{o}}dinger equations on the
  torus}, arXiv:1603.01947,  (2016).

\bibitem{TaoCBMS07}
{\sc T.~Tao}, {\em Nonlinear Dispersive Equations: Local and Global Analysis},
  no.~106 in CBMS, Amer. Math. Soc., 2007.

\bibitem{TsutsumiFukuda1}
{\sc M.~Tsutsumi and I.~Fukuda}, {\em On solutions of the derivative nonlinear
  {S}chr{\"o}dinger equation: {E}xistence and uniqueness theorems}, Funkcial.
  Ekvac., 23 (1980), pp.~259--277.

\bibitem{TsutsumiFukuda2}
\leavevmode\vrule height 2pt depth -1.6pt width 23pt, {\em On solutions of the
  derivative nonlinear {S}chr{\"o}dinger equation {I}{I}}, Funkcial. Ekvac., 24
  (1981), pp.~85--94.

\bibitem{vanSaarloosHohenberg}
{\sc W.~{van Saarloos} and P.~Hohenberg}, {\em Fronts, pulses, sources and
  sinks in generalized complex {G}inzburg-{L}andau equations}, Physica D:
  Nonlinear Phenomena, 56 (1992), pp.~303--367.

\bibitem{Weinstein83}
{\sc M.~I. Weinstein}, {\em Nonlinear {S}chr{\"{o}}dinger equations and sharp
  interpolation estimates}, Comm. Math. Phys., 87 (1983), pp.~567--576.

\bibitem{WinKyoto08}
{\sc Y.~Win}, {\em Unconditional uniqueness of the derivative nonlinear
  {S}chr{\"{o}}dinger equation in energy space}, J. Math. Kyoto Univ., 48
  (2008), pp.~683--697.

\bibitem{WinFE2010}
{\sc Y.~Win}, {\em {G}lobal well-posedness of the derivative nonlinear
  {S}chr{\"o}dinger equations on {T}}, Funkcial. Ekvac., 53 (2010), pp.~51--88.

\bibitem{Wu2013}
{\sc Y.~Wu}, {\em Global well-posedness of the derivative nonlinear
  {S}chr{\"o}dinger equations in energy space}, Anal. PDE, 6 (2013),
  pp.~1989--2002.

\bibitem{WuAPDE2}
{\sc Y.~Wu}, {\em Global well-posedness on the derivative nonlinear
  {S}chr{\"{o}}dinger equation}, Anal. PDE, 8 (2015), pp.~1101--1112.

\end{thebibliography}

\end{document}